\documentclass[a4paper,12pt]{article}
\usepackage{tikz}
\usetikzlibrary{plotmarks}
\usetikzlibrary{arrows,shapes,positioning}
\usetikzlibrary{decorations.markings}
\tikzstyle arrowstyle=[scale=2]
\tikzstyle directed=[postaction={decorate,decoration={markings,
    mark=at position .65 with {\arrow[arrowstyle]{stealth}}}}]

\tikzstyle reverse directed=[postaction={decorate,decoration={markings,
    mark=at position .65 with {\arrowreversed[arrowstyle]{stealth};}}}]
    \tikzstyle left directed=[postaction={decorate,decoration={markings,
    mark=at position -.62 with {\arrow[arrowstyle]{stealth}}}}]

\tikzstyle left reverse directed=[postaction={decorate,decoration={markings,
    mark=at position -.62 with {\arrowreversed[arrowstyle]{stealth};}}}]
\usepackage{indentfirst}
\usepackage[plainpages=false]{hyperref}
\usepackage{amsfonts,latexsym,rawfonts,amsmath,amssymb,amsthm,mathrsfs}
\usepackage{amsmath,amssymb,amsfonts,latexsym,lscape,rawfonts}

\usepackage[all]{xy}
\usepackage{eufrak}
\usepackage{makeidx}         
\usepackage{graphicx,psfrag}

\usepackage{array,tabularx}

\usepackage{setspace}

\newtheorem{thm}{Theorem}[section]
\newtheorem{cor}[thm]{Corollary}
\newtheorem{lem}[thm]{Lemma}
\newtheorem{clm}[thm]{Claim}
\newtheorem{prop}[thm]{Proposition}

\newtheorem{defi}[thm]{Definition}
\theoremstyle{remark}
\newtheorem{rmk}[thm]{Remark}

\theoremstyle{definition}
\newtheorem{Def}[thm]{Definition}                                        %

\title{Bessel Functions, Heat Kernel and  the Conical K\"ahler-Ricci Flow.}
\author{Xiuxiong Chen, Yuanqi Wang}
\begin{document}
\maketitle{}
\begin{abstract}
   Following Donaldson's oppenness theorem on deforming a conical K\"ahler-Einstein metric, we prove a parabolic Schauder-type estimate with respect to conical metrics. As a corollary, we show that the  conical K\"ahler-Ricci Flow exists for short time. The key is to establish  the relevant  heat kernel estimates,  where we use the Weber's formula on Bessel function of the second kind and   Carslaw's heat kernel representation in \cite{Car}.
\end{abstract}
\section{Introduction.}
Let $(M, [\acute{\omega}])$ be a compact  K\"ahler manifold with a smooth  K\"ahler metric $\acute{\omega}$. Let $D$ be a smooth divisor in $M$. Let  $L$ be the line bundle associated to $D$ and let   $S$ be the holomorphic section 
which defines $D$. Following Donaldson\cite{Don}, we want to  study K\"ahler metrics on $M \setminus  D$ with cone singularities of cone angle $2\beta\pi$ transverse to $D,\;$ where $
0< \beta < 1.\;$ 
In \cite{Don}, Donaldson studied the linear theory of how conical K\"ahler-Einstein metrics deform when the cone angle varies.  In particular, he proved that the set of cone angles where  $M$ admits a K\"ahler-Einstein metric is open if the  only holomorphic vector field tangent to $D$ is the zero vector field.  This ``openness theorem"
 fits well into an ambitious program \cite{Don08} of Donaldson in which he proposed a new continuity approach to attack
the renown K\"ahler-Einstein problem via deforming cone angles. Namely, first show the existence of K\"ahler-Einstein metrics when conical angle is arbitrary small. Then, gradually open up the angle and show this process is both open and closed, using stability conditions.  Donaldson's work \cite{Don}, as well as his program, inspired a lot
new research activities surrounding the existence of conical K\"ahler
-Einstein metrics.  The later topic (Singular K\"ahler Einstein metrics) really goes
back much earlier in history (c.f., Yau \cite{Yau}, Tian-Yau, \cite{Tian Yau1}\cite{TY87},  etc).  For more recent references, we refer to two recent works  \cite{CGP} and \cite{EGZ} and reference therein.
\\

In this paper, following Donaldson's work on the green function in \cite{Don}, we  study the  existence
of  an evolution process.  On one hand, we want to deform this K\"ahler metric in the negative direction of its Ricci form; on the other hand, we want to keep the cone structure
fixed during the``evolution" process. The second goal poses the main analytic challenge.     A crucial estimate in \cite{Don} is the estimate of the Green function  for the standard flat conical  metric. It is somewhat surprising that the proof of such
an estimate relies heavily on  direct calculations and estimate of special functions (Bessel functions). In our approach of evolution process, inevitably we need to  estimate the heat Kernel for
conical metrics. As expected, such an estimate need to employ heavy calculation on Bessel functions. The main conclusion of ours are summarized in Theorem \ref{short time existence of CKRF: metric version with a single divisor}, \ref{Schauder estimate of the linear equation: single divisor}, \ref{all the properties of the h.k1}. The diagram in Figure 1 is an  overall  description of the organization of this article.  \\ 

It is perhaps worthwhile to give a brief account to the history of the K\"ahler- Einstein metric problem first.  In 1950s, E. Calabi asked a famous question: if the first Chern class $C_1 <0, = 0, > 0$, does there exist a K\"ahler-Einstein metric with Ricci curvature $<0, =0, >0$ respectively? The case of $C_1 <0$ is settled by Aubin and
by Yau , while Yau solved the case $C_1 = 0.\;$ For the case $C_1 >0$, among work
of others, Tian give a complete solution to the existence of KE metric on Fano surfaces. In higher
dimensions,  the existence problem is very hard.  In early 1980s, Yau conjectured that the existence of K\"ahler-Einstein metric is related to certain algebraic notions of ``stability". Tian \cite{Tian97} introduced a notion called K-stability and proved that it is a necessary condition for the existence of KE metrics. The K-stability was reformulated later in more algebraic ways  by Donaldson \cite{donaldson02}.  The necessary part
has been established by Tian\cite{Tian97},  Stoppa\cite{Stoppa}, Berman\cite{Berman12}  with various
generalities. The existence part is much more difficult. Recently, Chen-Donaldson-Sun 
confirmed this conjecture of Yau in
a series of work \cite{CDS0}, \cite{CDS1},\cite{CDS2}, \cite{CDS3}.\\

While the main goal stated in Donaldson's program has been successfully tackled, it opens a ``door" for exciting future
study in K\"ahler geometry. Following Calabi, one might ask if $(M, (1-\beta)D)$ supports a K\"ahler-Einstein metric with cone angle $2\beta\pi$ transversal to $D$. This existence problem can certainly be reduced to a complex Monge-Ampere equation in
$M$ and one may attempt to solve its existence via the continuous method.  This approach has been taken by a number of authors.  As examples,  we list a few: the work of  Berman \cite{Berman}, Brendle\cite{Brendle} Jeffres-Mazzeo-Rubinstein \cite{JMR}, Li-Sun \cite{LS}, Song-Wang\cite{SongWang}, and Campana- Guenancia-Paun \cite{CGP}. Calamai -Zheng studied the geodesics in the space of K\"ahler cone metrics in \cite{KaiZheng} and A. Hugh  \cite{AHugh}  studied the geodesic in the space of K\"ahler cusp metrics. Historically, existence and uniqueness of  conical KE metric over Riemann surfaces has been extensively studied. See Troyanov's work \cite{Troyanov}  and McOwen's work \cite{McOwen} for the existence, and Luo-Tian's work \cite{LuoTian} for the uniqueness.\\

First, let us define a  conical K\"ahler metric.

\begin{defi}\label{definition of (alpha,beta) conical metric:single divisor.}  For any $\alpha \in (0, \min\{{1\over \beta}-1,1\}),\;$ a   K\"ahler form $\omega$ is said to be an $(\alpha, \beta)$  conical K\"ahler metric on $(M, (1-\beta) D)$  if it satisfies the following conditions
\begin{enumerate}
\item $\omega$ is a closed positive $(1,1)$ current on $M$.
\item For any point $p\in D$, there exists a small holomorphic chart $({\cal U}, \{z_i\})$
such that in this chart, $\omega$ is quasi conformal to the standard cone metric
\[
{\sqrt{-1}\over 2}  |z_1|^{2\beta - 2} d z_1 \wedge d \bar z_1 + {\sqrt{-1}\over 2}  \displaystyle \sum_{j=2}^n d z_j \wedge d \bar z_j.
\]
\item There exists a $\phi\in C^{2,\alpha,\beta}(M)$ and a smooth
K\"ahler metric $\acute{\omega}$ such that \[\omega=\acute{\omega}+ i \partial \bar \partial \phi.\]   See Definition \ref{Def of Schauder norms} for the
definition of the function space $C^{2,\alpha,\beta}(M).\;$
\end{enumerate}
\end{defi}

 Notice that the following model metric defined in \cite{Don} satisfies the above definition.
 \begin{equation}\label{Model conical metric defined by Donaldson}\omega_{D}=\omega^{'}+\delta i \partial \bar \partial |S|^{2\beta}
 , \end{equation}
 $\omega^{'}$ is a smooth K\"ahler metric over M and
$\delta $ is sufficiently small. Similar to  (\ref{Regularity of h omega 0}), $\omega_{D}$ is a $(\alpha_0,\beta)$ metric, $ \alpha_0=\min\{\frac{2}{\beta}-2,\ 1 \}$.\\

Now we are ready to introduce the conical K\"ahler-Ricci flow precisely as
\begin{equation}\label{Definition of CKRF}
{{\partial \omega_g}\over {\partial t}}  = \beta\omega_g -   Ric(g)+ 2\pi(1-\beta)[D],
\end{equation}
where the initial metric $g(0)$ is an $(\alpha, \beta)$ type  K\"ahler metric in $M.\;$ (\ref{Definition of CKRF})  should be considered as an equaiton of closed currents over the whole $M$ (not only over $M\setminus D$). Over $M\setminus D$,   (\ref{Definition of CKRF}) reduces exactly to the usual Ricci flow equation: 
\begin{equation*}
{{\partial \omega_g}\over {\partial t}}  = \beta\omega_g-   Ric(g) .
\end{equation*}

We can also consider more general conical K\"ahler-Ricci flows as 
\begin{equation*}
{{\partial \omega_g}\over {\partial t}}  = \mu \omega_g -   Ric(g)+ 2\pi(1-\beta)[D].
\end{equation*} 
for any number $\mu$, and the short time existence as Theorem  \ref{short time existence of CKRF: metric version with a single divisor} holds equally well. For the sake of brevity and to make the reader understand better, we only consider the flow 
(\ref{Definition of CKRF}) in this article. \\

At the level of potentials, we have
\begin{equation}
 \frac{\partial \phi}{\partial t}=\log\frac{(\omega_D+\sqrt{-1}\partial\bar{\partial}\phi)^n}{\omega_D^n}+\beta \phi+ h_{\omega_D}.
\end{equation}
Here $h_{\omega_D}$ satisfies $\sqrt{-1} \partial \bar \partial h_{\omega_D} = \beta\omega_D - Ric(\omega_D)+2\pi(1-\beta)[D]$ over $M.\;$ The $h_{\omega_D}$ has a nice expression as 
\[h_{\omega_D}=\beta\delta |S|^{2\beta}+\log \frac{\omega_D^n |S|^{2-2\beta} }{\omega^n}+\acute{F},\]
where $\acute{F}$ is a smooth function over $M$. Routine calculation shows that
\begin{equation}\label{Regularity of h omega 0}h_{\omega_D}\in C^{\alpha_0,\beta},\ \alpha_0=\min\{\frac{2}{\beta}-2,\ 1 \}.
\end{equation}

The main problem in this paper is to show the following: if the initial metric  $\omega_0$ is of $(\alpha, \beta)$ type, do we have
a one parameter family of   $(\alpha, \beta)$ type conical K\"ahler metrics $\{\omega(t) (t\in [0, T])\}$ initiated from $\omega_0$ which satisfies the Ricci flow 
equation?
Unlike the classical settings, the short time existence of such a flow is a real challenge:  the nature of  heat
flows is to ``smooth" singularities, while the key point in this flow is to preserve the singularity structure.\\

Our main theorem on the short time existence of conical K\"ahler-Ricci flow can be formulated as follows:
\begin{thm}\label{short time existence of CKRF: metric version with a single divisor}   Suppose  $g_0$ is an $(\acute{\alpha},\beta)$ conical K\"ahler metric in $(M, (1-\beta) D)\;$ where $\alpha' \in  (0, \min\{\frac{1}{\beta}-1,1\}).\;$
For any $\alpha \in (0, \alpha') ,$ there exists an $T_0>0$  (which depends on $g_0$) such that the conical K\"ahler-Ricci flow (\ref{Definition of CKRF})
initiated from $g_0$ admits a solution $g(t)$, $t\in [0,T_0]$ which is smooth (both in space and time) in $(M\setminus D) \times [0,T_0]$. Moreover we have \begin{itemize}
       \item  for every $t\in [0,T_0]$, $g(t)$ is  an $(\alpha,\beta)$ conical metric in $(M, (1-\beta) D);\;$
       \item $g(t)$ is a $C^{\alpha,\frac{\alpha}{2},\beta}[0,T_0]$-family of conical  metrics (see Definition \ref{Def of a C 2 alpha, alpha over 2 family of  metrics}).
       \item  for any $0<\widehat{\alpha}\leq \alpha$, $g(t)$ is the unique solution of (\ref{Definition of CKRF}) in the class of $C^{\widehat{\alpha},\frac{\widehat{\alpha}}{2},\beta}[0,T]$-family of metrics.  
     \end{itemize}

\end{thm}

\begin{rmk}
For short time existence theorem, we allow $D$ to be a reducible,  smooth divisor.  
 In \cite{Don}, one needs the kernel of the linearized operator to be trivial to show that
 the angle can be perturbed.  This is not needed in our case since we don't change angles. We only need the Schauder estimates for Linearized operator (c.f. Theorem 1.7  below). \\

In fact, a more general version is true (see Theorem \ref{short time existence of CKRF: metric version with multiple divisors and multiple angles}). In Theorem \ref{short time existence of CKRF: metric version with multiple divisors and multiple angles}, we allow the angles on each component of $D$ to be different.
\end{rmk}
\begin{rmk}
When $n=1$ the short time existence in a different function space  is proved by Yin in 2009 in \cite{Yin}, using a very different method from ours. Recently, Mazzeo announced  another short time existence theorem on conical Ricci flows over conical Riemann surfaces,  which is proved in an unpulished work jointly by Mazzeo, Rubinstein, and Sesum.  Both results are very interesting to us and are  relevant to the dimension 2
(complex dimension 1) part of our work.
\end{rmk}
\begin{rmk} It's interesting to fit Theorem \ref{short time existence of CKRF: metric version with a single divisor} into a more general picture.  To be precise,   we have at least three kinds of Ricci flows starting from an ($\alpha, \beta$) metric as follows. The first kind  is as the solution in Theorem \ref{short time existence of CKRF: metric version with a single divisor} (c.f. the preceding Remark), which perserves the cone structure; the second kind is as the solutions studied  by Topping over incomplete Riemann surfaces (cf. \cite{Topping}), which flow  incomplete (2-dimensional) metrics 
instaneously to be complete;  the third kind is as the solutions considered by 
Chen-Tian-Zhang in \cite{CTZ} (and Chen-Ding in \cite{ChenDing}), which smooth out the conical singularities immediately after the flows start.
\end{rmk}

\begin{rmk}

When $\beta=1$, this reduces to the renown K\"ahler-Ricci flow. After the fundamental work of G. Perelman in the K\"ahler-Ricci flow, this has become a powerful tool in attacking the existence of K\"ahler-Einstein problems in Fano manifolds.  There has been  extensive research done in this subject. We refer interested readers to  \cite{TZ3},  \cite{CWB},  \cite{SunWang}, and references therein for further readings. \\

In a sequel of this paper \cite{CYW}, we will prove the long time existence of this conical K\"ahler-Ricci flow. As consequence, we prove
the convergence of the CKRF when the modified first Chern class is either negative or vanishes (cf. Cao\cite{Cao} for comparison in
the smooth KRF settings). We will also study the extension of Perelman's functional to this flow.
\end{rmk}

 At the level of potential functions, Theorem \ref{short time existence of CKRF: metric version with a single divisor} is implied by the following theorem on parabolic Monge-Ampere equations  which we use  all efforts to  prove.
\begin{thm}\label{short time existence of CKRF:potential version, single divisor} Assumptions as in Theorem 1.2. Suppose $0<T<\infty$ and  $f\in C^{2+\acute{\alpha},1+\frac{\acute{\alpha}}{2},\beta}[0,T]$. Suppose $\mu$ is a constant. Then there exists an $0<T_0\leq T$  such that  the conical K\"ahler-Ricci flow equation \begin{displaymath}
 \left \{
\begin{array}{ccr}\label{Def of potential CKRF equation}
 & \frac{\partial \phi}{\partial t}=\log\frac{(\omega_0+\sqrt{-1}\partial\bar{\partial}\phi)^n}{\omega_0^n}+\mu\phi+ f\\
  & \phi=0   \ \textrm{when} \  t=0\\
\end{array} \right.
\end{displaymath}
admits a solution $\phi \in C^{2+\alpha,1+\frac{\alpha}{2},\beta}[0,T_0]$. Furthermore, for any $0<\widehat{\alpha}\leq \alpha$, $\phi$ is the unique solution  in  $C^{2+\widehat{\alpha},1+\frac{\widehat{\alpha}}{2},\beta}[0,T_0]$.
\end{thm}
The crucial ingredients  to establish Theorem \ref{short time existence of CKRF:potential version, single divisor} are the following Schauder type estimates.
\begin{thm}\label{Schauder estimate of the linear equation: single divisor} Suppose 
$0<T<\infty$ and $\{ a(t),\ t\in [0,T]\} $ is a $C^{\alpha,\frac{\alpha}{2},\beta}[0,T]$  family of $(\alpha, \beta)$ metrics (See Definition \ref{Def of a C 2 alpha, alpha over 2 family of  metrics}).  Then there exists a constant $C_{a}$ (depending on $a(t)$) with the following property.  For any   $v\in C^{\alpha,\frac{\alpha}{2},\beta}[0,T]$, the parabolic  equation
 \begin{equation}\label{heat equation with metric g}
\frac{\partial u}{\partial t}=\Delta_{a(t)}u+v,\ u(0,x)=0 \end{equation}
 admits a unique solution  $u $ in $C^{2+\alpha,1+\frac{\alpha}{2},\beta}[0,T]$. Moreover,
  the following estimates hold. $$|u|_{2+\alpha,1+\frac{\alpha}{2},\beta,{M}\times[0,T]}\leq C_{a} |v|_{\alpha,\frac{\alpha}{2},\beta,{M}\times[0,T]};$$
 $$|u|_{2,\alpha,\beta,{M}\times[0,T]}\leq C_{a} |v|_{\alpha,\beta,{M}\times[0,T]}.$$
\end{thm}
The notion of $C^{\alpha,\frac{\alpha}{2},\beta}[0,T]$, $ C^{2+\alpha,1+\frac{\alpha}{2},\beta}[0,T]$ and the other relevant norms will be clearified in
Definition \ref{Def of Schauder norms}.
\begin{rmk}Theorem \ref{Schauder estimate of the linear equation: single divisor} is actually with respect to the real setting. Hopefully  it  can be applied to study real conical flows (such as Yamabe-flow) as well.  
\end{rmk}
\begin{rmk}To prove the parabolic Schauder estimates, we follow Donaldson's  idea  in his  proof of the elliptic Schauder estimates. Aside from Donaldson's work, Jefferes-Mazzeo-Rubinstein  also proved a version of Schauder estimate for elliptic equations, by appealing to the ``edge calculus machinery", while in this article  we do everything by hand.   \\
\end{rmk}
Theorem \ref{Schauder estimate of the linear equation: single divisor} depends on the following estimates on the heat kernel for the standard cone metric.
\begin{thm}\label{all the properties of the h.k1} There exists a constant C as in Definition
\ref{Dependence of the constant C} with the following property. Suppose $\mathfrak{F}$ is a differential operator in $\mathfrak{P}$ (defined  in (\ref{Def of mathfrak P})). Let $|\mathfrak{F}|$ be the order of the differential operator $\mathfrak{F}$. Then  the following estimates on the heat kernel $H(x,y,t)$ of the standard cone metric $g_{E,\beta}$ (see Definition \ref{standard cone metric}) hold.
\begin{enumerate}
\item Suppose $|\mathfrak{F}|=1$. Suppose $t=1$, $x\in A_{\frac{1}{\sqrt{s}}}$, $y\in A_{\frac{10}{\sqrt{s}}}$, $0<s\leq \frac{1}{10000}$, then
  $$|\nabla_xH(x,y,1)|\leq Ce^{-\frac{1}{s}}.$$
  \item If $|\mathfrak{F}| \leq 2$, then for any $x,\ y$,  we have
  $$\int_{0}^{\infty}| \mathfrak{F} H(x,y,\tau)|d\tau\leq C|x-y|^{-(m+|\mathfrak{F}|)}.$$
  \item
Suppose $P(x)\geq \frac{|x-y|}{100^{100}}.\;$ If $|\mathfrak{F}|=3,$ then
we have
$$\int_{0}^{\infty}|\mathfrak{F} H(x,y,\tau)|d\tau\leq C|y-x|^{-(m+3)}.$$
\item
For any $x,\ y$, 
we have
$$\int_{0}^{\infty}|\nabla_{x}\frac{\partial}{\partial t} H(x,y,\tau)|d\tau\leq C|y-x|^{-(m+3)}.$$

  \item If $|\mathfrak{F}| \leq 4$, for any $x,\ y$ and $0<\alpha\leq 1$, we have
   $$\int_{\mathbb{R}^2\times \mathbb{R}^{m}}|\sup_{1\leq t\leq 2} \mathfrak{F} H(x,y,t)||x-y|^{\alpha}dy\leq C.$$
\item For every second order spatial derivative operator $\mathfrak{D}\in \mathfrak{\mathfrak{T}}$, suppose $\rho= \min(\frac{1}{\beta}-1,\ 1)$, $|y|=1$, and $|u_1|,\ |u_2|< \frac{1}{8}$, we have
$$\int_{0}^{\infty}|\mathfrak{D}H(u_1,y,\tau)-\mathfrak{D}H(u_2,y,\tau)|d\tau\leq C|u_1-u_2|^{\rho}.$$
\item Suppose $\rho= \min(\frac{1}{\beta}-1,\ 1)$, $|y|=1$, and $|u_1|,\ |u_2|< \frac{1}{8}$, we have
$$\int_{0}^{\infty}|\nabla_{x}H(u_1,y,\tau)-\nabla_{x}H(u_2,y,\tau)|d\tau\leq C|u_1-u_2|^{\rho}.$$
\end{enumerate}
 \end{thm}
\begin{proof}{of Theorem \ref{all the properties of the h.k1}:} It's exactly  a summarization of Lemma \ref{integral of time derivative of hk w.r.t time},
\ref{integral of gradient of hk w.r.t time},
 \ref{integral of gradient of time derivative of the hk}, \ref{time integral of the bilinear derivatives of the heat kernel},
 \ref{integral of third derivative of h.k w.r.t time: when r big},  Lemma \ref{Holder estimate of the integral of spatial derivative of h.k with respect to time}, Lemma   \ref{Bound on the gradient of h.k when r small and t equals 1},  and Proposition \ref{proposition on integral of the product the heat kernel and distance function to the alpha}.
\end{proof}
Here the Properties (1-5) are satisfied by the heat kernel of the Euclidean space while the proofs are quite subtle. Property 6 and 7 are the only two properties which involve the exponent $\beta$. The difficulty of the proof is that the conical heat kernel is far less explicit than the Euclidean heat kernel or heat kernel in a smooth manifold.  Thus,  we are forced to deal with the representation formulas given in proposition \ref{Heat kernel formula} and Theorem \ref{Carslaw-Wang-Chen representation for hk}. These representation formulas are due to Carslaw; and we  reformulate them to apply to our case.  In particular,   the piecewise continuous (with respect to angle) representation formula in Theorem \ref{Carslaw-Wang-Chen representation for hk} is crucial. These representation formulas  involve  Bessel functions of the second kind, and some complicated integration of trignometric functions, and hyperbolic functions over angles. \\

   Our proof is  inspired by Donaldson's work \cite{Don}.  There are some important
  difference in the technical level since we have to deal with a more complex situations. Restricted to the elliptic case, the function space we consider is exactly the same
  as Donaldson's $C^{2,\alpha,\beta}$ space in \cite{Don}.  Our heat kernel estimates, restricted to the elliptic case, contain  a little bit  more information on the Green's function than the estimates in \cite{Don} (the Green function estimates in \cite{Don} is sufficient for elliptic Schauder estimates). Our proof also gives another approach to prove the estimates in \cite{Don} (see Section \ref{Local estimates for the heat kernel}).  The technical difference between our work and Donaldson's work on the Green's function in \cite{Don} is the following.  In  \cite{Don},  Donaldson works from the heat kernel formula which  involves  Bessel functions
of the first kind and  uses a clever change  varible in his proof. In our parabolic case,    we employ  two  representation formulas of the heat kernel which are different from the one applied by Donaldson. One involves Bessel functions of the second kind and Weber's formula, the other one involves a contour integral which is essentially due to Carslaw. These different representation formulas are necessary ( especially Carslaw's formula) in our parabolic case.  Since we also need to understand the H\"older estimates over the time direction,  we have to prove more global and more explicit properties of the heat kernel.  While Donaldson shows us how to obtain Schauder estimates by working on  Bessel functions,
these new ``additions"  present difficult challenge for the proof our estimates. In summary, while ideas behind the heavy calculation are important, the sheer power of tensor manipulations and analysis of these Bessel functions are crucially needed. \\

 For related work  on the heat kernel of spaces with more genenral conical singularities, we  refer the interested readers to the work by Mooers in \cite{Mooers} and by 
 Mollers in \cite{Mollers}.\\
 
The next diagram (Figure 1) is to indicate the relations of the lemmas, corollaries, theorems,  and propositions in this article. Actually the whole article definitely can be summarized by Figure 1, aside from a few  minor entanglements.
\begin{center}
\begin{tikzpicture}[->,>=stealth',shorten >=1pt,auto,node distance=3cm,
  thick,main node/.style={rectangle,draw,font=\sffamily\small\bfseries}]

  \node[main node] (3)  {T\ref{Schauder estimate of the linear equation: single divisor}} ;
   \node[main node] (2) [left of=3] {T\ref{short time existence of CKRF:potential version, single divisor}};
\node[main node] (1)[left of=2] {T\ref{short time existence of CKRF: metric version with a single divisor}};
\node[main node] (4) [below  of=3] {T \ref{Schauder estimate: constant coefficient}};
\node[main node] (9) [right  of=4] {P \ref{Solvability of the operator with small oscillation in the flat cone}};
\node[main node] (8) [left of=4] {Cor \ref{Schauder estimate for the weak solution which is compactly supported in balls}};
  \node[main node] (6) [below left of=4] {P\ref{Schauder estimate for the weak solution which is compactly supported}, };
   \node[main node] (7) [ below right of=4] {L\ref{parabolic intepolations},L\ref{parabolic intepolations 2},L\ref{parabolic timewise intepolations},
   L\ref{C0 estimate}};
 \node[main node] (11) [left of=6] {P\ref{Holder estimate of the spatial  derivative of the singular integral},
 P\ref{Holder estimate of the time derivative of the singular integral},P\ref{time holder estimate:spatial derivative},P\ref{time holder estimate:time derivative}};
     \node[main node] (14) [ below right of=11] {T\ref{all the properties of the h.k1}};
       \node[main node] (15) [ below right of=14] {L\ref{Bound on the gradient of h.k when r small and t equals 1},P\ref{proposition on integral of the product the heat kernel and distance function to the alpha}};
\node[main node] (17) [ below  left of=14] {L\ref{integral of time derivative of hk w.r.t time},L\ref{integral of gradient of hk w.r.t time},L\ref{integral of gradient of time derivative of the hk},L\ref{integral of third derivative of h.k w.r.t time: when r big},L\ref{time integral of the bilinear derivatives of the heat kernel}};
  \node[main node] (18)[right   of=15] { L\ref{Holder estimate of the integral of spatial derivative of h.k with respect to time}};
    \node[main node] (19)[right   of=18] { L\ref{Estimating the sum of G w.r.t k}};
      \node[main node] (26) [ below   of=17] {T\ref{Decay estimates for E}};
   \node[main node] (23) [ right of=26] {L\ref{asymptotic estimate of radius derivative of Heat kernel},L\ref{asymptotic estimate of theta derivative of Heat kernel},L \ref{asymptotic estimate of time derivative of Heat kernel}};
   \node[main node] (25) [  right of=23] {L\ref{estimating the holder continuity of the derivative of I}};
   \node[main node] (28) [  right of=25] {L\ref{Bessel properties: when r is small}};
  \node[main node] (27) [ below right of=26] {L\ref{lemma on E's z-derivative bounds},L\ref{Bound on the angle derivative}};
   \path[every node/.style={font=\sffamily\small}]
   (9) edge node [left] {} (3)
   (6) edge node [left] {} (8)
   (8) edge node [left] {} (4)
    (4) edge node [left] {} (3)
    (3) edge node [left] {} (2)
    (2) edge node [left] {} (1)
    (7) edge node [left] {} (4)
    (6) edge node [left] {} (4)
     (11) edge node [left] {} (6)
    (19) edge node [left] {} (18)
    (25) edge node [left] {} (18)
    (27) edge node [left] {} (26)
    (26) edge node [left] {} (23)
    (26) edge node [left] {} (17)
    (28) edge node [left] {} (19)
    (14) edge node [left] {} (11)
    (17) edge node [left] {} (14)
    (18) edge node [left] {} (14)
    (15) edge node [left] {} (14)
    (23) edge node [left] {} (15);
         \end{tikzpicture}
         \end{center}

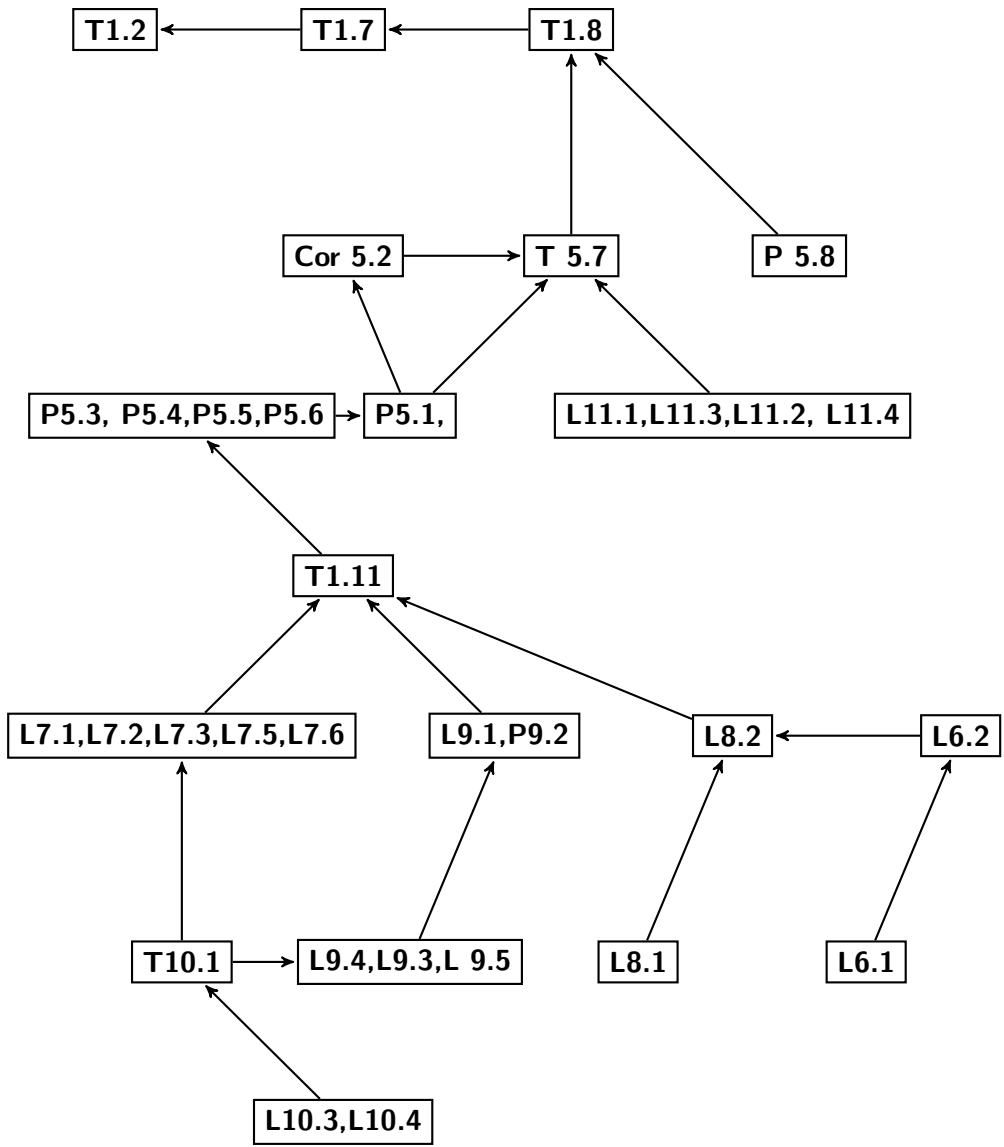
Figure 1: L means lemma, P means proposition, T means theorem, Cor means corollary. Arrows means imply.
\\

Organization of this article: Most of the organization is illustrated in Figure 1. With respect to topics, this article can be divided into three parts:
 \begin{enumerate}
   \item  proof of Theorem \ref{short time existence of CKRF: metric version with a single divisor} assuming Theorem \ref{Schauder estimate of the linear equation: single divisor};
   \item proof of Theorem \ref{Schauder estimate of the linear equation: single divisor} assuming Theorem \ref{all the properties of the h.k1};
   \item  proof of  Theorem \ref{all the properties of the h.k1}.
 \end{enumerate}
 After defining  the relevant norms explicitly and stating several more general theorems in Section  \ref{Setting up of the main problems}, we carry out part 1   in section \ref{Proof of the short time existence assuming the Schauder estimate}; part 2  in Sections \ref{Representation formulas for the heat kernel},  \ref{Holder estimate of the singular integrals and proof of the main Schauder estimates}, \ref{Appendix: some lower order estimates}. Both part 1 and part 2 are more or less standard, except some points where we should adapt polar coordinates into our framework. Part 3 is the most crucial part of this paper and it is also the most ``non-regular" part.  This is carried out in Section
\ref{Notes on Donaldson's work}, \ref{Local estimates for the heat kernel}, \ref{Behaviors of the heat kernel  near the singular set}, \ref{Asymptotic behaviors of  the heat kernel}, \ref{Decay estimates for the heat kernel}.\\

   Acknowledgements: The second author would like to thank Professor Simon Donaldson for conversation on the theory of Bessel functions and his interest in this work. The second author is deeply grateful to  Professor Xianzhe Dai,  Professor Guofang Wei, and Prof Rugang Ye for offering and sponsoring the second author's first math job,  for their interest in this work,   for their continuous support and encouragements.  Both authors would like to thank   Haozhao Li, Weiyong He, Song Sun, and Kai Zheng very much for carefully reading  the earlier versions of this paper. 

\section{Setting up of the main problems and more general statements.\label{Setting up of the main problems}}

   First we would like to mention that, though some operators are not defined in the singular set, we still consider the domain of the functions to be the whole manifold $M$ (instead of $M\setminus D$). The reason is that the functions and currents in our consideration are global, unless otherwise stated. \\
   
        Let $m = 2n$ be a  nonnegative integer.  The local geometry of the conical metrics near the singularity hypersurfaces are like the following.  Let $g_{E,\beta}$ be the standard cone  metric on $\mathbb{R}^2\times \mathbb{R}^{m}, $ i.e.,
          \begin{equation}\label{standard cone metric} g_{E,\beta}=dr^2+\beta^2r^2d\theta^2+g_{ \mathbb{R}^{m}}.\;
          \end{equation} Here $m$ is an even integer. In the subsequent work we will mainly use $\mathbb{R}^{m}$ instead of $C^{n}$, as our Schauder estimates are with respect to the  real setting.  \\
     
   Now following Donaldson, let $z_j=s_j+\sqrt{-1}s_{\bar{j}},j=1,2 \cdots n $ ( $\bar{j}=j+ n\; $). We consider a basis of $(1,0)$ vectors as
   \begin{equation}\label{mathfrak a}\mathfrak{a}=\frac{1}{\sqrt{2}}(\frac{\partial }{\partial r}-\frac{\sqrt{-1}}{\beta r}\frac{\partial }{\partial \theta}), \frac{\partial }{\partial z_j}, j=1...\frac{m}{2}.
   \end{equation}
   Set $\xi=z^{\beta}=re^{i\beta\theta}$, notice that
   \[\frac{\partial^2}{\partial \xi \partial\bar{\xi}}=\frac{1}{4}\Delta_{\beta}=\frac{1}{4}[\frac{\partial^2  }{\partial r^2}+r^{-1}\frac{\partial  }{\partial r}+\frac{1}{\beta^2}r^{-2}\frac{\partial^2  }{\partial \theta^2}].\]
  One can write  the Laplacian of $g_{E,\beta}$ as
 $$ \Delta_{E,\beta}=\Delta_{\beta}+\Sigma_{i=1}^{m}\frac{\partial^2}{\partial s_i^2}. $$
 When there is no confusion,  we will denote $ \Delta_{E,\beta}$ simply as $\Delta$. Moreover,
   $\sqrt{-1}\partial \bar{\partial}$ operator is well defined under the basis
   $$\frac{\partial^2}{\partial \xi \partial\bar{\xi}}, \mathfrak{a}\frac{\partial}{\partial\bar{z_i}},\bar{\mathfrak{a}}\frac{\partial}{\partial z_i},\frac{\partial^2}{\partial z_i \partial\bar{z_j}}, 1\leq i,j\leq n.$$
   
   To state the main properties of the heat kernel, we need to introduce some linear, differential operators (up to 4th orders). First, 
 we write $\nabla_x$ as the gradient operator in $\mathbb{R}^2\times \mathbb{R}^{m}$ with respect to a natural basis \[\frac{\partial}{\partial r}, \frac{1}{\beta r}\frac{\partial}{\partial \theta},\frac{\partial}{\partial s_i} (1\leq i\leq m).\] Secondly, we can introduce a set of real, second order operators
 $\mathfrak{T}$  as:
\begin{equation}\label{Def of mathfrak T}\mathfrak{T}=\{\frac{\partial^2}{\partial r\partial s_i},\ i\leq m;\ \frac{\partial^2}{\partial s_i \partial s_j},\ i,j\leq m;\ \frac{1}{r}\frac{\partial^2}{\partial s_i\partial \theta},\ i\leq m\}.\end{equation}
Finally, the set of operators  $\mathfrak{P}$ we need to study in this paper as
  
  \begin{equation}\label{Def of mathfrak P}\mathfrak{P}=\{{\nabla_x},{\partial \over \partial t}, \mathfrak{D}, \nabla_x\nabla_{y};
 \nabla_x \mathfrak{D},\nabla_x \frac{\partial}{\partial t}; \frac{\partial}{\partial t} \mathfrak{D}, {\partial^2 \over {\partial t^2}}|  \mathfrak{D} \in  \mathfrak{T}\}.\end{equation}
  
 With the $\sqrt{-1}\partial \bar{\partial}$ operator in mind, we  define the Sobolev space $W^{2,2}(M\times[0,T])$ by the following norm. 
\begin{eqnarray}\label{W 2 2 norm}& &|\widehat{u}|_{W^{2,2}(M\times[0,T]) }
 \\&\triangleq & |\frac{\partial \widehat{u}}{\partial t}|_{L^{2}(M\times[0,T])}+
 |\sqrt{-1}\partial \bar{\partial}\widehat{u}|_{L^{2}(M\times[0,T])}
 +|\nabla \widehat{u}|^{(1)}_{L^{2}(M\times[0,T])}+|\widehat{u}|_{L^{2}(M\times[0,T])}.\nonumber
  \end{eqnarray}
This Sobolev norm has a good application in Proposition \ref{Solvability of the operator with small oscillation in the flat cone}.\\
 
 Let  $\Omega $ be a domain in  $\mathbb{R}^2\times \mathbb{R}^{m}$. For any point $x\in \Omega$ we denote $d_x$ to be  the distance from $x$ to the boundary $\partial \Omega$. Notice that  the points in the singular set $\{0\}\times R^m \cap \ int(\Omega)$  are never considered to be boundary points, and this is quite crucial in this article.  For any two  points $x,y\in \Omega$, we set

  \[ d_{x,y}=\min(d_x,d_y),\ d_x\ \textrm{is the spatial distance from}\ x\ \textrm{to}\ \partial{\Omega}.\]
  Similarly, we set
 \begin{eqnarray*} & & l_{(x,t_1),(y,t_2)}=\min(l_{(x,t_1)},l_{(y,t_2)}),
 \end{eqnarray*}  
 where the $l_{(x,t)}$ is defined as 
 \[l_{(x,t)}=\min\{d_x,\ |t-T_1|^{\frac{1}{2}}\}.\]
For any $[T_1, T_2]\subset \mathbb{R}^+$ and
  for any function $\widehat{u}$ over $\Omega\times [T_1, T_2]$, we introduce the usual semi norms as
 \[\begin{array}{lcl}
  [\widehat{u}]_{0,\Omega \times [T_1,T_2] } & = & \displaystyle \sup_{\Omega \times [T_1,T_2]} |\widehat{u}(x,t)|,\\
    {[\widehat{u}]}_{\alpha,\Omega \times [T_1,T_2] }& =&  \displaystyle \sup_{(x,t),(y,t)\in\Omega \times [T_1,T_2]} \frac{|\widehat{u}(x,t)-\widehat{u}(y,t)|}{|x-y|^{\alpha}},\\
   {[\widehat{u}]}_{\alpha,\frac{\alpha}{2},\Omega \times [T_1,T_2] } &= & \displaystyle \sup_{(x,t_1),(y,t_2)\in\Omega \times [T_1,T_2]} \frac{|\widehat{u}(x,t_1)-\widehat{u}(y,t_2)|}{|x-y|^{\alpha}+|t_1-t_2|^{\frac{\alpha}{2}}}.\end{array}\]

   Now, we define weighted Schauder norms (with spatial distance as weights) as:

 \[\begin{array}{lcl} [\widehat{u}]^{(k)}_{0,\Omega \times [T_1,T_2] } & = &  \displaystyle\sup_{\Omega \times [T_1,T_2]} d_x^{k}|\widehat{u}(x,t)|,\\
    {[\widehat{u}]}^{(k)}_{\alpha,\Omega \times [T_1,T_2] } & = &  \displaystyle \sup_{(x,t),(y,t)\in\Omega \times [T_1,T_2]} d_{x,y}^{k+\alpha}\frac{|\widehat{u}(x,t)-\widehat{u}(y,t)|}{|x-y|^{\alpha}},\\
    {[\widehat{u}]}^{(k)}_{\alpha,\frac{\alpha}{2},\Omega \times [T_1,T_2] } &= &  \displaystyle \sup_{(x,t_1),(y,t_2)\in\Omega \times [T_1,T_2]} d_{x,y}^{k+\alpha}\frac{|\widehat{u}(x,t_1)-\widehat{u}(y,t_2)|}{|x-y|^{\alpha}+|t_1-t_2|^{\frac{\alpha}{2}}}.\end{array}\]
We also define 
parabolic weighted Schauder norms (with parabolic  distance as weights) as:

 \[\begin{array}{lcl} [\widehat{u}]^{[k]}_{0,\Omega \times [T_1,T_2] } & = &  \displaystyle\sup_{\Omega \times [T_1,T_2]} l_{x,t}^{k}|\widehat{u}(x,t)|,\\
    
    {[\widehat{u}]}^{[k]}_{\alpha,\frac{\alpha}{2},\Omega \times [T_1,T_2] } &= &  \displaystyle \sup_{(x,t_1),(y,t_2)\in\Omega \times [T_1,T_2]} l_{(x,t_1);(y,t_2)}^{k+\alpha}\frac{|\widehat{u}(x,t_1)-\widehat{u}(y,t_2)|}{|x-y|^{\alpha}+|t_1-t_2|^{\frac{\alpha}{2}}}.\end{array}\]
We define the  norm $|\cdot|^{*}_{2,\alpha,\Omega \times [T_1,T_2] }$  as 
 \begin{eqnarray*}& &|\widehat{u}|^{*}_{2,\alpha,\Omega \times [T_1,T_2] }
 \\&\triangleq& [\frac{\partial \widehat{u}}{\partial t}]^{(2)}_{\alpha,\Omega \times [T_1,T_2] }+[i\partial\bar{\partial}\widehat{u}]^{(2)}_{\alpha,\Omega \times [T_1,T_2]}+|i\partial\bar{\partial}\widehat{u}|^{(2)}_{0,\Omega \times [T_1,T_2]}
 \\& &+|\frac{\partial \widehat{u}}{\partial t}|^{(2)}_{0,\Omega \times [T_1,T_2] }+|\nabla \widehat{u}|^{(1)}_{0,\Omega \times [T_1,T_2]}+|\widehat{u}|_{0,\Omega \times [T_1,T_2]}.
  \end{eqnarray*}

   Finally,  we define 
  the space $\widehat{C}^{2+\alpha,1+\frac{\alpha}{2},\beta}_{\star}\{\Omega \times [T_1,T_2] \}$ (with spatial distance as weights) by  the norm
 \begin{eqnarray*}& &|\widehat{u}|^{*}_{2+\alpha,1+\frac{\alpha}{2},\Omega \times [T_1,T_2] }
  \\&\triangleq& [i\partial\bar{\partial}\widehat{u}]^{(2)}_{\alpha,\frac{\alpha}{2},\Omega \times [T_1,T_2] }+[\frac{\partial \widehat{u}}{\partial t}]^{(2)}_{\alpha,\frac{\alpha}{2},\Omega \times [T_1,T_2] }
 \\& &+[\frac{\partial \widehat{u}}{\partial t}]^{(2)}_{0,\Omega \times [T_1,T_2]}
 +[i\partial\bar{\partial}\widehat{u}]^{(2)}_{0,\Omega \times [T_1,T_2]}+|\nabla \widehat{u}|^{(1)}_{0,\Omega \times [T_1,T_2]}+ |\widehat{u}|_{0,\Omega \times [T_1,T_2]};
 \end{eqnarray*}
and the space 
$\widehat{C}^{2+\alpha,1+\frac{\alpha}{2},\beta}_{[\star]}\{\Omega \times [T_1,T_2] \}$ (with parabolic distance as weights) by the norm
 \begin{eqnarray*}& &|\widehat{u}|^{[*]}_{2+\alpha,1+\frac{\alpha}{2},\Omega \times [T_1,T_2] }
  \\&\triangleq& [i\partial\bar{\partial}\widehat{u}]^{[2]}_{\alpha,\frac{\alpha}{2},\Omega \times [T_1,T_2] }+[\frac{\partial \widehat{u}}{\partial t}]^{[2]}_{\alpha,\frac{\alpha}{2},\Omega \times [T_1,T_2] }
 \\& &+[\frac{\partial \widehat{u}}{\partial t}]^{[2]}_{0,\Omega \times [T_1,T_2]}
 +[i\partial\bar{\partial}\widehat{u}]^{[2]}_{0,\Omega \times [T_1,T_2]}+|\nabla \widehat{u}|^{[1]}_{0,\Omega \times [T_1,T_2]}+ |\widehat{u}|_{0,\Omega \times [T_1,T_2]}.
 \end{eqnarray*}
 Actually in most of the cases we will consider the polydisks centered at points in the divisor. Namely we will frequently appeal to  the following definition.
 \begin{Def}\label{Def of polydisk}Given $p$ in the singular set $(0)\times \mathbb{R}^m$, we consider $A_R(p)$ as 
 \[A_R(p)= D_{R}(p)\times \widehat{B}_{R}(p),\]
 where $D_{R}(p)$ is the disk of radius $R$ in the $\mathbb{R}^2$ component and
 $\widehat{B}_{R}(p)$ is the ball  of radius $R$ in the $\mathbb{R}^m$ component. When 
 the center is $0$ we abbreviate $A_R(0)$ as $A_R$. The advantage of using such domains can be seen, for example, when we  estimate $Z_6$ in the proof of proposition \ref{Holder estimate of the time derivative of the singular integral}.
 \end{Def}
  We wish to point out that, unlike in the paper \cite{Don},  we allow the hypersurface $D$  is the disjoint union of smooth irreducible divisors. Set
  \[\Lambda=\displaystyle\Sigma_{j=1}^{l}2\pi(1-\beta_i)D_i,\]
  where  each $D_i$ is a connected divisor. For any $\epsilon$ small,  denote the tubular neighborhood $T_{\epsilon}(D)$ with radius $\epsilon$ of the hypersurface $D$.  Set $\nu >0$ to be small enough so that the following decomposition holds
  \[
  T_{2\nu}(D)\subset\bigcup_{i=1}^{l}\bigcup_{k=1}^{l_{D_i}} N_k(D_i)
  \]
   (such that  in each $N_k(D_i)$,
$D_i$ is the zero locus of some holomorphic function $z_0$ in $N_k(D_i)$.  Suppose $\iota_{\beta_i}$ is the map defined by
\[
\begin{array}{lcl}
\iota_{\beta_i}: N_k(D_i) & \rightarrow & \mathbb{R}^2 \times \mathbb{C}^{n}\\
\qquad (z_0,z_1......z_n) & \rightarrow & (|z_0|^{\beta_i-1}z_0,z_1,....z_n). \end{array}\]
Using this map,  we define
\begin{equation}\displaystyle \label{Definition of maximal spatial C 2 alpha norm}|u|_{2,\alpha,\Lambda,M \times [T_1,T_2]}=\Sigma_{i=1}^{l}\Sigma_{k=1}^{l_{D_i}}|\widehat{u}|^{\ast}_{2,\alpha,\iota_{\beta_i}[N_k(D_i)]\times [T_1,T_2]}+|u|^{\ast}_{2,\alpha,[M\setminus{T_{\nu}(D)}]\times [T_1,T_2]}
\end{equation} and
 \begin{eqnarray}\label{Definition of time spatial C 2 alpha norm}& &|u|_{2+\alpha,\ 1+\frac{\alpha}{2},\Lambda,M \times [T_1,T_2]}\nonumber
 \\&=&\Sigma_{i=1}^{l}\Sigma_{k=1}^{l_{D_i}}|\widehat{u}|^{\ast}_{2+\alpha,\ 1+\frac{\alpha}{2},\iota[N_k(D)]\times [T_1,T_2]]}+|u|^{\ast}_{2+\alpha,\ 1+\frac{\alpha}{2},[M\setminus{T_{\nu}(D)}]\times [T_1,T_2]}.
 \end{eqnarray}
\begin{Def}\label{Def of Schauder norms}
  We define the space
$C^{2,\alpha,\Lambda}[T_1,T_2]$ by the norm $|\ \cdot\ |_{2,\alpha,\Lambda,M \times [T_1,T_2]}$ and the space $C^{2+\alpha,1+\frac{\alpha}{2},\Lambda}[T_1,T_2]$ by $|\ \cdot\ |_{2+\alpha,\ 1+\frac{\alpha}{2},\Lambda,M \times [T_1,T_2]}$.\\

When $\beta_i = \beta$ for every $i$, then we  simply replace $\Lambda$ by $\beta $ to denote the space and norms $$C^{2+\alpha,1+\frac{\alpha}{2},\Lambda}[T_1,T_2]\ (C^{2,\alpha,\Lambda}[T_1,T_2]),\ |\ \cdot\ |_{2+\alpha,1+\frac{\alpha}{2},\Lambda,M\times [T_1,T_2]}\ (|\ \cdot\ |_{2,\alpha,\Lambda,M \times [T_1,T_2]}),$$
as
$$C^{2+\alpha,1+\frac{\alpha}{2},\beta}[T_1,T_2]\ (C^{2,\alpha,\beta}[T_1,T_2]),\ |\ \cdot\ |_{2+\alpha,1+\frac{\alpha}{2},\beta,M\times [T_1,T_2]}\ (|\ \cdot\ |_{2,\alpha,\beta,M \times [T_1,T_2]}).$$
The  other norms concerning $\Lambda$ should have $\Lambda$ replaced by $\beta$ also.
\\

A time-independent function $u$ is said to be in $C^{2,\alpha,\beta}$ $(C^{2,\alpha,\Lambda})$, if $u$ is in $C^{2,\alpha,\beta}[T_1,T_2]$ $(C^{2,\alpha,\Lambda}[T_1,T_2]$. Since $u$ does not depend on time,  the definition does not depend on $T_1, T_2$ at all.
\end{Def}
\begin{rmk}Similar to norms in (\ref{Definition of maximal spatial C 2 alpha norm}) and (\ref{Definition of time spatial C 2 alpha norm}), we can use the transformations $\iota_{\beta_i}$ to define weaker norms
$$|\ \cdot\ |_{1,\alpha,\Lambda,M \times [T_1,T_2]};\ \ \ |\ \cdot\ |_{1+\alpha, \frac{1}{2}+\frac{\alpha}{2},\Lambda,M \times [T_1,T_2]}$$
and
$$|\ \cdot\ |_{\alpha,\Lambda,M \times [T_1,T_2]};\ \ \ |\ \cdot\ |_{\alpha, \frac{\alpha}{2},\Lambda,M \times [T_1,T_2]}.$$
\end{rmk}
\begin{rmk}
When $u$ is a time independent function over $M$,  we may define the norm $|\ u\ |_{2,\alpha,\Lambda,M}$ as in \cite{Don}.
 To be more precise, the norm $|\ u\ |_{2,\alpha,\Lambda,M}$ is also equivalent to
$|\ u\ |_{2,\alpha,\Lambda,M\times[T_1, T_2]}$ if we view $u$ to be a space time function though $u$ doesn't depend on time. \\
\end{rmk}

\begin{Def}\label{Def of a C 2 alpha, alpha over 2 family of  metrics}
 A time-dependent family  $\{ a(t),\ t\in [0,T]\} $ of $(\alpha, \beta)$ metrics is said to be  $C^{\alpha,\frac{\alpha}{2},\beta}[0,T]$  if there exist 
 a $\phi(t)\in C^{2+\alpha,1+\frac{\alpha}{2},\beta}(M\times [0,T])$  and a smooth metric $\acute{\omega}$ such that 
 \[a(t)=\acute{\omega}+ i \partial \bar \partial \phi(t).
 \]
\end{Def}

In the general case,   we should consider the reference metric $\omega_0$ as\footnote{Note that $\delta$ must be small enough.}
$$\omega_{\Lambda}=\omega^{'}+\delta i \partial \bar \partial [|S_1|_1^{2\beta_1}\times |S_2|_2^{2\beta_2}\times......|S_i|_i^{2\beta_i}......|S_{l}|_{l}^{2\beta_{l}}].$$
and define the $(\alpha, \Lambda)$ metrics as in definition \ref{definition of (alpha,beta) conical metric:single divisor.} (with $\beta$ replaced by $\Lambda$).  Then, more general theorems hold (comparing with Theorems \ref{short time existence of CKRF: metric version with a single divisor}, \ref{short time existence of CKRF:potential version, single divisor}, \ref{Schauder estimate of the linear equation: single divisor}).
\begin{Def}\label{Dependence of the constant C} Since we are dealing with many estimates in this article, it's necessary to illustrate a principle of the dependency of $C$ in every theorem, proposition, corollary and lemma. Namely, we make the following important convention in this article:\\
      
      \textit{Without further notice, the  "C" in each estimate  means a constant depending on the dimension $m+2$, the angle $\beta$, the $\alpha$ (and $\acute{\alpha}$ if any) in the same estimate or in the corresponding theorem (proposition, corollary, lemma). We add subindex to the "C" if it depends on more factors than the above three. Moreover, the "C" in different places might be different.}
\end{Def}

\begin{thm}\label{short time existence of CKRF: metric version with multiple divisors and multiple angles} Suppose  $g_0$ is an $(\acute{\alpha},\Lambda)$ conical K\"ahler metric in $(M, \Sigma_{j=1}^{l}(1-\beta_j)D_j)\;$ where $\alpha' \in  (0, \min\{\min_{1\leq j\leq l}[\frac{1}{\beta_j}-1],1\}).\;$
For any $\alpha \in (0, \alpha') ,$  and constant number $\mu$, there exists an $T_0>0$  (which depends on $g_0$) such that the conical K\"ahler-Ricci flow equation
  \begin{equation*}
{{\partial \omega_g}\over {\partial t}}  = \beta\omega_g -   Ric(g)+ 2\pi \Sigma_{j=1}^{l}(1-\beta_j)D_j,\ \omega_g=\omega_{g_0}\ \textrm{when}\ t=0
\end{equation*}
 admits a solution $g(t)$, $t\in [0,T_0]$ which is smooth (both in space and time) in $(M\setminus D) \times [0,T_0]$. Moreover we have \begin{itemize}
       \item  for every $t\in [0,T_0]$, $g(t)$ is  an $(\alpha,\Lambda)$ conical metric in $(M, \Sigma_{j=1}^{l}(1-\beta_j)D_j));\;$
       \item $g(t)$ is a $C^{\alpha,\frac{\alpha}{2},\Lambda}[0,T_0]$-family of conical  metrics (see Definition \ref{Def of a C 2 alpha, alpha over 2 family of  metrics}).
       \item  for any $0<\widehat{\alpha}\leq \alpha$, $g(t)$ is the unique solution of (\ref{Definition of CKRF}) in the class of $C^{\widehat{\alpha},\frac{\widehat{\alpha}}{2},\Lambda}[0,T]$-family of metrics.  
     \end{itemize}
\end{thm}

\begin{thm}\label{Schauder estimate of the linear equation: multiple divisor}Suppose 
$0<T<\infty$ and $\{ a(t),\ t\in [0,T]\} $ is a $C^{\alpha,\frac{\alpha}{2},\Lambda}[0,T]$-family of $(\alpha, \Lambda)$ metrics (See Definition \ref{Def of a C 2 alpha, alpha over 2 family of  metrics}).  Then there exists a constant $C_{a}$ (depending on $a(t)$) with the following property.  For any   $v\in C^{\alpha,\frac{\alpha}{2},\Lambda}[0,T]$, the parabolic  equation
 \begin{equation}
\frac{\partial u}{\partial t}=\Delta_{a(t)}u+v,\ u(0,x)=0 \end{equation}
 admits a unique solution  $u $ in $C^{2+\alpha,1+\frac{\alpha}{2},\Lambda}[0,T]$. Moreover,
  the following estimates hold. $$|u|_{2+\alpha,1+\frac{\alpha}{2},\Lambda,{M}\times[0,T]}\leq C_{a} |v|_{\alpha,\frac{\alpha}{2},\Lambda,{M}\times[0,T]};$$
 $$|u|_{2,\alpha,\Lambda,{M}\times[0,T]}\leq C_{a} |v|_{\alpha,\Lambda,{M}\times[0,T]}.$$
\end{thm}
\section{Proof of Theorem \ref{short time existence of CKRF: metric version with a single divisor} and Theorem \ref{short time existence of CKRF:potential version, single divisor}
assuming Theorem  \ref{Schauder estimate of the linear equation: single divisor}. \label{Proof of the short time existence assuming the Schauder estimate}}
In this section we assume Theorem  \ref{Schauder estimate of the linear equation: single divisor} to be right and prove the short time existence of CKRF.
  The proof
is actually regular and no special conical structure is involved. For the sake of being rigorous we shall include the proof here.
Since Theorem  \ref{Schauder estimate of the linear equation: single divisor} is an independent part,  we  left it to later discussions.

\begin{proof}{of Theorem \ref{short time existence of CKRF:potential version, single divisor}:}\ \  Actually Theorem \ref{short time existence of CKRF:potential version, single divisor} is  a direct conclusion of Theorem \ref{Schauder estimate of the linear equation: single divisor}  and the iteration argument.
Since we lack an exact reference on the result we want, we include the crucial detail here. We still denote $C_0^{2+\alpha,1+\frac{\alpha}{2},\beta}[0,T]$ as the subspace of functions with
zero initial value of $C^{2+\alpha,1+\frac{\alpha}{2},\beta}[0,T]$. For brevity, we only prove in detail the short time existence result assuming  $\mu=0$.  The result for all $\mu$  follows from exactly the same method. \\

  We consider the following two equations
 \begin{eqnarray}\label{CKRF equation with mu being 0}
 \frac{\partial \phi}{\partial t}=\log\frac{(\omega_0+\sqrt{-1}\partial\bar{\partial}\phi)^n}{\omega_0^n}+f;
   \phi=0   \ \textrm{when} \  t=0.
   \end{eqnarray}
   \begin{eqnarray}\label{Parabolic  equation with respect to the initial metric}
\frac{\partial u}{\partial t}=\Delta_{\omega_0}u + f,\
u=0   \ \textrm{when} \  t=0.
 \end{eqnarray}
Substracting (\ref{Parabolic  equation with respect to the initial metric}) from (\ref{CKRF equation with mu being 0}) we end up with the followin equation.  
\begin{eqnarray}\label{Substraction trick}
 \frac{\partial v}{\partial t}=\log\frac{(\omega_u+\sqrt{-1}\partial\bar{\partial}v)^n}{\omega_u^n}+F(u);
   v=0   \ \textrm{when} \  t=0.
   \end{eqnarray}
   In the above equation, we define that 
   \begin{equation*}\omega_u=(\omega_0+\sqrt{-1}\partial\bar{\partial}u)^n,\ v=\phi-u,\
   F(u)=\log \frac{\omega_u^n}{\omega_0^n}-\Delta_{\omega_0}u.
   \end{equation*}
   Notice that when $T$ is small enough, $\omega_u$ is still a uniformly elliptic family of metric.
   A crucial advantage of this trick of substraction is that 
   \begin{equation*}F(u)=0 \ \textrm{when} \  t=0,
   \end{equation*}
  which implies that the linearized operator is a contraction for short time. \\ 

 To continue, we define $L=\frac{\partial }{\partial t}-\Delta_{\omega_u} - F(u).\;$ Then  this defines
 an  operator
 \[
 L: C_0^{2+\alpha,1+\frac{\alpha}{2},\beta}[0,T] \rightarrow C^{\alpha, \frac{\alpha}{2},\beta}[0,T].
 \]
 By Theorem \ref{Schauder estimate of the linear equation: single divisor}, this is an invertible operator with $|L^{-1}|$ uniformly bounded.
Set  $$E(v)=\log\frac{(\omega_u+i\partial\bar{\partial}v)^n}{\omega_u^n}-\Delta_{\omega_u}.$$
Consider the operator
\[
  \Phi =  L^{-1} \circ E: C_0^{2+\alpha,1+\frac{\alpha}{2},\beta}[0,T]  \rightarrow C_0^{2+\alpha,1+\frac{\alpha}{2},\beta}[0,T] .
\]
 We shall  show in the following that $\Phi$ is a contracting map  in a small ball $B_{\epsilon_1}$ (centered at the zero function) in $C_0^{2+\alpha,1+\frac{\alpha}{2},\beta}[0,T] .\;$  For any $ v_1, v_2 \in C_0^{2+\alpha,1+\frac{\alpha}{2},\beta}[0,T]$, there exist $u_1, u_2 \in C_0^{2+\alpha,1+\frac{\alpha}{2},\beta}[0,T]$
 such that
 $$Lu_1=E(v_1)\qquad {\rm and}\;\; Lu_2=E(v_2)$$
 or equivalently
 \[
 \Phi(v_1) = u_1 \qquad {\rm and}\;\; \Phi(v_2) = u_2.\;
 \]
 Thus we obtain the equation
 \[
 L(u_1-u_2)  = E(v_1) - E(v_2).
 \]
 Note
that  $E$ is a quadratic operator.   Suppose  $|v_1|,|v_2|\leq \epsilon_2$ and $\epsilon_2$ is small enough,  we have that $$|E(v_1)-E(v_2)|_{\alpha,\frac{\alpha}{2},\beta,{M}\times[0,T]}\leq \epsilon_3|v_1-v_2|_{2+\alpha,1+\frac{\alpha}{2},\beta,{M}\times[0,T]}$$

 By Theorem \ref{Schauder estimate of the linear equation: single divisor},  we have
 \[\begin{array}{lcl}
|u_1-u_2|_{2+\alpha,1+\frac{\alpha}{2},\beta,{M}\times[0,T]}& \leq &C_1|E(v_1)-E(v_2)|_{2+\alpha,1+\frac{\alpha}{2},\beta,{M}\times[0,T]}\\
& \leq & C_1 \cdot \epsilon_3\cdot |v_1-v_2|_{2+\alpha,1+\frac{\alpha}{2},\beta,{M}\times[0,T]}.
\end{array}
\]

Then making $\epsilon_2$ sufficiently small, $\epsilon_3$ can be  small enough too.  Thus,  we have:\begin{eqnarray*}& &|\Phi(v_1) -\Phi(v_2)|_{2+\alpha,1+\frac{\alpha}{2},\beta,{M}\times[0,T]} 
\\&=&  |u_1-u_2|_{2+\alpha,1+\frac{\alpha}{2},\beta,{M}\times[0,T]}\leq \frac{1}{2}|v_1-v_2|_{2+\alpha,1+\frac{\alpha}{2},\beta,{M}\times[0,T]}.
\end{eqnarray*}
Then, $\Phi$ is a contraction map in $B_{\epsilon_2}$. To show that there is a fixed point, it suffices to show the first iteration is small (provided the time is short).
We have the following crucial lemma about the first iteration.
\begin{clm}\label{first iteration is small}Let $v_0\equiv 0$. For any $\epsilon_4>0$, there exists a  $T$ small enough such that $$|\Phi(v_0)|_{2+\alpha,1+\frac{\alpha}{2},\beta,{M}\times[0,T]}\leq \epsilon_4.$$
\end{clm}
To prove claim \ref{first iteration is small}, it suffices  to use the Schauder estimate over time.  This is achieved as follows. By Theorem \ref{Schauder estimate of the linear equation: single divisor} we know that $F(u)\in C^{\acute{\alpha},\frac{\acute{\alpha}}{2},\beta}[0,T]$.  Notice  by our choice we have that  $\alpha<\acute{\alpha}$. Then using the crucial fact that $F(u)=0$ when 
 $t=0$,  we trivially get 
 \begin{equation*}|F(u)|_{0,M\times [0,T]}\leq T^{\frac{\acute{\alpha}}{2}}[F(u)]_{\acute{\alpha},\frac{\acute{\alpha}}{2},\beta,M\times [0,T]}.
 \end{equation*}   
 Denote  $\Phi(v_0)$ as $v_1$, notice that 
 \[\frac{\partial v_1}{\partial t}-\Delta_{\omega_u}v_1 - F(u)=0.\] Then by applying 
  Theorem \ref{Schauder estimate of the linear equation: single divisor}, we have when $T$ is small enough that 
\begin{eqnarray*}& &|v_1|_{2+\alpha,1+\frac{\alpha}{2},\beta,{M}\times[0,T]}
\\& \leq & C|F(u)|_{\alpha,\frac{\alpha}{2},\beta,M\times [0,T]}
\\& = & C\{|F(u)|_{0,M\times [0,T]}+[F(u)]_{\alpha,\frac{\alpha}{2},\beta,M\times [0,T]}\}
 \\&  \leq& CT^{\frac{\acute{\alpha}}{2}}[F(u)]_{\acute{\alpha},\frac{\acute{\alpha}}{2},\beta,M\times [0,T]}+C T^{\frac{\acute{\alpha}-\alpha}{2}}|F(u)|_{\acute{\alpha},\frac{\acute{\alpha}}{2},\beta,M\times [0,T]}\leq \epsilon_4.
\end{eqnarray*}
 Then claim \ref{first iteration is small} is proved.\\

 Now we are ready to
define the iteration process. Set $v_0 = 0$ and inductively, we define

\[
  v_{k+1} = \Phi( v_{k}), \qquad {\rm where}\;\; k=0,1,2,\cdots.
\]

Then combining claim \ref{first iteration is small} we have for $T$ small that $|v_k|_{2+\alpha,1+\frac{\alpha}{2},\beta,{M}\times[0,T]}\leq \epsilon_3$ (then $\Phi$ is a contraction over the sequence). Hence
\[
\begin{array}{lcl}
 |v_{k+1} - v_{k}|_{2+\alpha,1+\frac{\alpha}{2},\beta,{M}\times[0,T]} & = & |\Phi(v_k) -\Phi(v_{k-1})|_{2+\alpha,1+\frac{\alpha}{2},\beta,{M}\times[0,T]}
 \\ & \leq & {1\over 2}  |v_{k} - v_{k-1}|_{2+\alpha,1+\frac{\alpha}{2},\beta,{M}\times[0,T]}\\
 & \leq & {1\over 2^k}   |v_{1} - v_{0}|_{2+\alpha,1+\frac{\alpha}{2},\beta,{M}\times[0,T]}\leq {\epsilon_{4}\over 2^k}.
\end{array}
\]
Then it is easy to see that the sequence $\{v_k\}$ is a Cauchy sequence in  $C_0^{2+\alpha,1+\frac{\alpha}{2},\beta}[0,T]. \;$ Thus,
 $\{v_k\}$  converge to a fixed point $v_\infty \in  C_0^{2+\alpha,1+\frac{\alpha}{2},\beta}[0,T]$ such that
\[
    \Phi(v_\infty) =  v_\infty.
\]
In other words,
\[
{{\partial v_\infty}\over {\partial t}} = \log\frac{(\omega_u+i\partial\bar{\partial}v_\infty)^n}{\omega_u^n} +F(u).
\]
Therefore $\phi=v_{\infty}-u$ is a solution to (\ref{CKRF equation with mu being 0}).\\

\end{proof}
\begin{proof}{of Theorems \ref{short time existence of CKRF: metric version with a single divisor}, \ref{short time existence of CKRF: metric version with multiple divisors and multiple angles}:} Again we only say a few words on  Theorem \ref{short time existence of CKRF: metric version with a single divisor}. Theorem \ref{short time existence of CKRF: metric version with multiple divisors and multiple angles} is similar.

 Now,   recall 
\[
f = h_{\omega_{D}}+\log\frac{\omega_0^n}{\omega_{D}^n}.
\]
Since $\omega$ is $(\acute{\alpha},\beta)$, then using \ref{Regularity of h omega 0} we  see that $f \in C^{\acute{\alpha},\beta}$.
Thus by applying Theorem \ref{short time existence of CKRF:potential version, single divisor},  the proof is complete.
\end{proof}
\section{Representation formulas for the heat kernel.\label{Representation formulas for the heat kernel}}
In this section, we give a representation formula for the heat Kernel with respect to the conical model metric $g_{E,\beta}\;$
over $\mathbb{R}^2 \times \mathbb{R}^m.\;$ \\

As in \cite{Wa}, we denote $I_{\mu}(z)$ as the Bessel function of  second kind with order $\mu$ and $J_{\mu}(z)$ as the Bessel function of  first kind with order $\mu$. \\

  First we recall the Weber's formula
\begin{equation}\label{Weber's formula}\frac{1}{2 t}e^{-\frac{r^2+{r^{'}}^2}{4t}}I_{\mu}(\frac{rr^{'}}{2t})=\int_{0}^{\infty}e^{-\lambda^2 t}J_{\mu}(\lambda r)J_{\mu}(\lambda r^{'})\lambda d\lambda,\end{equation}

First, for the sake of completeness we give a proof of the heat kernel formula when $m=0.\;$. We believe this proof is well known to  experts.

\begin{prop}
\label{Heat kernel formula}The heat kernel of  $\{R^2,\ dr^2+\beta^2r^2d\theta^2\}$ is
\begin{eqnarray*}& &\widehat{H}(r,\theta, r^{'},\theta^{'},t)
\\&=&\frac{1}{2\pi\beta}\int^{\infty}_{0}e^{-\lambda^2 t}J_{0}(\lambda r)J_{0}(\lambda r^{'})\lambda d\lambda
\\& & + \frac{1}{2\pi\beta}\displaystyle\Sigma_{k\geq 1} 2\cos(k(\theta-\theta^{'}))\int^{\infty}_{0}e^{-\lambda^2 t}J_{\frac{k}{\beta}}(\lambda r)J_{\frac{k}{\beta}}(\lambda r^{'})\lambda d\lambda
\\&=&\frac{1}{2\pi\beta}\frac{1}{2 t}e^{-\frac{r^2+{r^{'}}^2}{4t}}I_{0}(\frac{rr^{'}}{2t}) + \frac{1}{2\pi\beta}\frac{1}{2 t}e^{-\frac{r^2+{r^{'}}^2}{4t}}\displaystyle \Sigma_{k\geq 1} 2\cos (k(\theta-\theta^{'})) I_{\frac{k}{\beta}}(\frac{rr^{'}}{2t}).
\end{eqnarray*}
\end{prop}
\begin{proof}{of Proposition \ref{Heat kernel formula}:} First, we prove by direct computation that
 \[(\Delta_{\beta}-\frac{\partial}{\partial t})\widehat{H}=0.\]
 Secondly, we want to prove that
$$\lim_{t\rightarrow 0}\beta\int_{-\pi}^{\pi}\int_{0}^{+\infty}\widehat{H}(r,\theta,r^{'},\theta^{'},t)f(r^{'},\theta^{'})r^{'}dr^{'}d\theta^{'}=f(r,\theta).$$
We claim that the following holds.
\begin{clm}\label{Bessel representation of a function}Suppose  $v>-\frac{1}{2}$ and $r>0$,  we have that
$$\displaystyle \lim_{t\rightarrow 0}\int_{0}^{+\infty}\frac{1}{2 t}e^{-\frac{r^2+{r^{'}}^2}{4t}}I_{v}(\frac{rr^{'}}{2t})f(r^{'},\theta^{'})r^{'}dr^{'}=f(r,\theta^{'}).$$
\end{clm}
\begin{proof}
  Suppose $r>0$,  set  $u=\frac{r^{'}-r}{2\sqrt{t}}.\;$ We have
\begin{eqnarray*}& &\lim_{t\rightarrow 0}\int_{0}^{+\infty}\frac{1}{2 t}e^{-\frac{r^2+{r^{'}}^2}{4t}}I_{v}(\frac{rr^{'}}{2t})f(r^{'},\theta^{'})r^{'}dr^{'}
\\&=&\lim_{t\rightarrow 0}\int_{-\infty}^{+\infty}\frac{1}{\sqrt{t}}e^{-u^2}I_{v}(\frac{r^2+2ru\sqrt{t}}{2t})e^{-\frac{r^2+2ru\sqrt{t}}{2t}}f(r+2u\sqrt{t},\theta^{'})(r+2u\sqrt{t})du
\end{eqnarray*}
According to formula $(1)$ of $I_{v}(x)$ in 7.25 of \cite{Wa}, we have
\begin{eqnarray*}& &\lim_{t\rightarrow 0}\frac{1}{\sqrt{t}}I_{v}(\frac{r^2+2ru\sqrt{t}}{2t})e^{-\frac{r^2+2ru\sqrt{t}}{2t}}(r+2u\sqrt{t})
\\&=&\lim_{t\rightarrow 0}\frac{1}{\sqrt{\pi}\sqrt{\frac{r^2+2ru\sqrt{t}}{t}}}\ast\frac{\int_{0}^{\frac{r^2+2ru\sqrt{t}}{t}}w^{v-\frac{1}{2}}e^{-w}dw}{\Gamma(v+\frac{1}{2})}\ast\frac{r+2u\sqrt{t}}{\sqrt{t}}
\\&=&\frac{1}{\sqrt{\pi}}.
\end{eqnarray*}
It follows that
 \begin{eqnarray*}& &\lim_{t\rightarrow 0}\int_{0}^{+\infty}\frac{1}{2 t}e^{-\frac{r^2+{r^{'}}^2}{4t}}I_{v}(\frac{rr^{'}}{2t})f(r^{'},\theta^{'})r^{'}dr^{'}
\\&=&\frac{1}{\sqrt{\pi}}\int_{-\infty}^{+\infty}e^{-u^2}duf(r,\theta^{'})=f(r,\theta^{'}).
\end{eqnarray*}

Then claim \ref{Bessel representation of a function} is proved.
\end{proof}

Now we are ready to prove proposition \ref{Heat kernel formula}. From Fourier Series theory we trivially have
\begin{eqnarray*}& &\lim_{t\rightarrow 0}\beta\int_{-\pi}^{\pi}\int_{0}^{+\infty}H(r,r^{'},\theta,\theta^{'},t)f(r^{'},\theta^{'})r^{'}dr^{'}d\theta^{'}
\\&=&\frac{1}{2\pi}\int_{-\pi}^{\pi}f(r,\theta^{'})d\theta^{'}+\frac{1}{2\pi}\Sigma_{k\geq 1} 2\int_{-\pi}^{\pi}f(r,\theta^{'}) \cos(k(\theta-\theta^{'}))d\theta^{'}
\\&=&f(r,\theta).
\end{eqnarray*}
\end{proof}
\begin{rmk}One should notice that when $r=0$, the heat kernel does not tend to the $\delta-$function anymore when $t\rightarrow 0$. Nevertheless, this does not affect Lemma \ref{representation of the derivative of singular integrals} to be true since we assume $r\neq 0$ in it.
\end{rmk}
Thus it's easy to get the expression for the heat kernel $H(x,y,t)$  of   our model conical metric $g_{E,\beta}: $
\begin{eqnarray}\label{Higher dim heat kernel formula}H(x,y,t) & = & H(r,\theta, r^{'},\theta^{'},R,t)\nonumber
\\&=&\frac{1}{\beta(4\pi t)^{\frac{m}{2}+1}}e^{-\frac{r^2+{r^{'}}^2+R^2}{4t}}[I_{0}(\frac{rr^{'}}{2t})+\Sigma_{k\geq 1} 2\cos(k(\theta-\theta^{'}))I_{\frac{k}{\beta}}(\frac{rr^{'}}{2t})]\nonumber
\\&=&\frac{2t}{\beta(4\pi t)^{\frac{m}{2}+1}}e^{-\frac{R^2}{4t}}[\int^{\infty}_{0}e^{-\lambda^2 t}J_{0}(\lambda r)J_{0}(\lambda r^{'})\lambda d\lambda \nonumber
\\&  +& \displaystyle \Sigma_{k\geq 1} 2\cos (k(\theta-\theta^{'}))\int^{\infty}_{0}e^{-\lambda^2 t}J_{\frac{k}{\beta}}(\lambda r)J_{\frac{k}{\beta}}(\lambda r^{'})\lambda d\lambda].
\end{eqnarray}
Here  $x=(r,\theta,\widehat{x}),\ y=(r^{'},\theta^{'},\widehat{y})$. $R=|\widehat{x}-\widehat{y}|$ is the tangential distance parallel to  $\{0\}\times R^m$.
\begin{center}
\begin{tikzpicture}[domain=-2:2]
\draw[->] (-5.2,0) -- (5.2,0) node[below] {$Re$};
\draw[->] (0,-5.2) -- (0,5.2) node[above] {$Im$};
\draw[very thin,color=gray] (-5,-5) grid (5,5);
\draw[color=blue,reverse directed,left reverse directed]  plot (\x,{1+(\x)^2})  node[right] {$\mathfrak{A}$};
\draw[color=blue,directed,left directed]  plot (\x,{-1-(\x)^2}) node[below] {$\mathfrak{A}$};
\node[right=0] at (0,0) {O};
\node[left=0] at (-3,0) {$-\pi$};
\node[right=0] at (3,0) {$\pi$};
\draw[color=blue,left  directed, directed] (-3,-5) -- (-3,5) node[above] {$L_1,\mathfrak{A}$};
\draw[color=blue, directed, left directed] (3,5) -- (3,-5) node[below] {$L_2,\mathfrak{A}$};
 \end{tikzpicture}
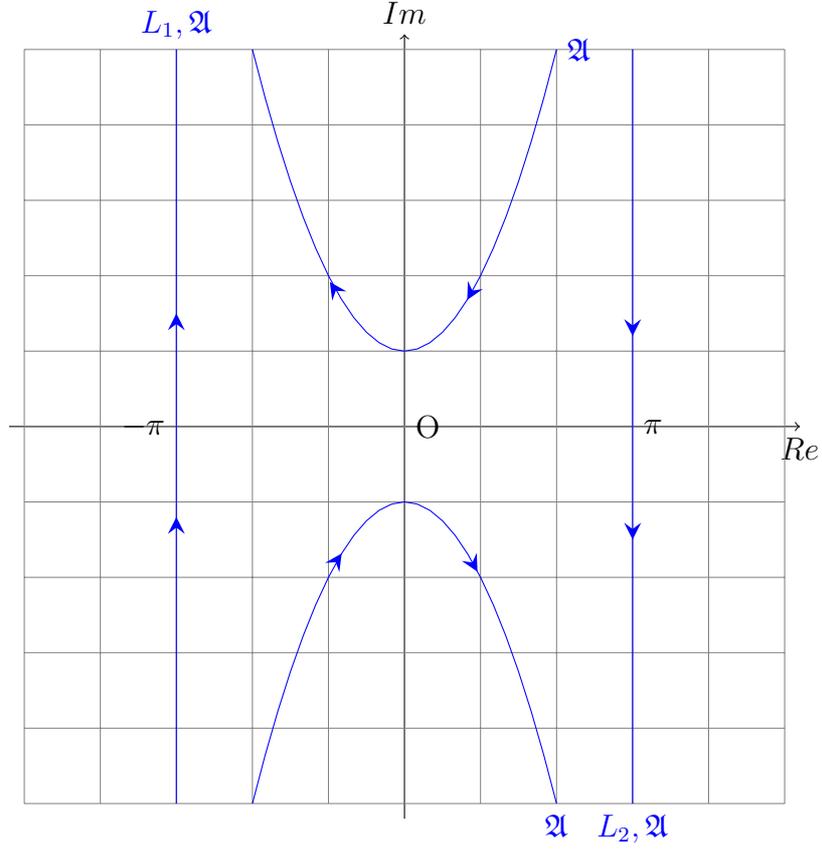
\[\textrm{Figure 2 for the contour}\ \mathfrak{A} \]
\end{center}

\begin{thm}\label{Carslaw-Wang-Chen representation for hk}(Reformulation of Carslaw's results in \cite{Car}) The heat kernel of the flat  conical K\"ahler metric $g_{E,\beta}$ is
\[ H(r,\theta,{\acute{r}},\theta^{'},R,t)=\frac{1}{2\pi \beta}\frac{1}{(4\pi t)^{\frac{m}{2}+1}}e^{-\frac{r^2+{{\acute{r}}}^2+R^2}{4t}}P[\frac{r{\acute{r}}}{2t},\beta(\theta^{'}-\theta)].\]
In the above formula, $P$ is defined as
 \begin{eqnarray} P(z,v) & \triangleq & \int_{\mathfrak{A}} e^{z\cos a}\frac{1}{1-e^{-\frac{i(a-v)}{\beta}}}da. \nonumber
 \end{eqnarray}
 P can also be expressed as 
  \begin{eqnarray} \label{P's Expression formula}& & P(z,v)\nonumber
  \\&=&2\pi[I_0(z)+\displaystyle \Sigma_{k\geq 1}2I_{\frac{k}{\beta}}(z)\cos \frac{kv}{\beta}]\nonumber.
\\&=& 2\pi \beta \displaystyle \Sigma_{k,\ -\pi <v+2k\beta \pi<\pi}e^{z\cos(v+2k\beta \pi)}
+ E(z,v)\nonumber
\\& &+\pi \beta \displaystyle \Sigma_{k,v+2k\beta \pi=\pm\pi}e^{-z}.
\end{eqnarray}

Here $$ E(z,v)=\int_{0}^{\infty}e^{-z\cosh y}\frac{2\sin\frac{\pi}{\beta}[\cos\frac{\pi}{\beta}-\cos\frac{v}{\beta}\cosh\frac{y}{\beta}]}{[\cosh\frac{y}{\beta}-\cos\frac{v-\pi}{\beta}][\cosh\frac{y}{\beta}-\cos\frac{v+\pi}{\beta}]}dy,$$
and $\mathfrak{A} $ is the path in Figure 2 above the theorem (which is the same as FIG 1 of \cite{Car} with argument translated by $\theta$). \\
\end{thm}
\begin{proof}{of Theorem \ref{Carslaw-Wang-Chen representation for hk}:}\
The first expression of the heat kernel is   exactly from the results  in Section 2 iof \cite{Car}. The second expression in the first line of formula (\ref{P's Expression formula}) (which concerns the Bessel function of the second kind $I$) is also exactly from \cite{Car}. Thus we only have to prove the expression in the second line of formula (\ref{P's Expression formula}).   Working from  the formula
\begin{equation}\label{Carslaw's representation formula }P(z,v)=\int_{\mathfrak{A}} e^{z\cos a}\frac{1}{1-e^{-\frac{i(a-v)}{\beta}}}da,\end{equation}
the singularities of this integrand are
\[ a_k=v+2k\beta\pi,\  -\pi\leq v+2k\beta\pi\leq \pi.\]
We can of course deform $\mathfrak{A}$ continuously away from these singularities to get a integral over the $L_1$ and $L_2$ in Figure 2, together with the residue integrals around the singularities $a_k$. At those $a_k$ ($a_k=v+2k\beta$) such that $-\pi<a_k<\pi$, the residues are  $2\pi\beta$. At those $a_k$  such that  $a_k=v+2k\beta\pi =\pi$ or $-\pi$, the contribution of the half circle with radius
tending to $0$ is
$$\lim_{\epsilon\rightarrow 0}\int_A e^{z\cos a}\frac{1}{1-e^{-\frac{i(a-v)}{\beta}}}da=\pi \beta e^{z\cos (\pm \pi)}=\pi \beta e^{-z}.$$

 To prove the formula for $E(z,v)$ in formula \ref{P's Expression formula}, it suffices to  notice that
 over $L_1$, $a=-\pi+iy,\ y\in (-\infty,+\infty)$;  over $L_2$, $a=\pi+iy,\ y\in (+\infty,-\infty)$. Thus,

 \begin{eqnarray*}& &\int_{L_1+L_2}e^{z[\cos a]}\frac{1}{1-e^{-\frac{i(a-v)}{\beta}}}da
 \\&=&\int_{-\infty}^{\infty}e^{-z\cosh y}\frac{i}{1-e^{\frac{iv}{\beta}}e^{\frac{y}{\beta}}e^{\frac{i\pi}{\beta}}}dy
 -\int_{-\infty}^{\infty}e^{-z\cosh y}\frac{i}{1-e^{\frac{iv}{\beta}}e^{\frac{y}{\beta}}e^{-\frac{i\pi}{\beta}}}dy
 \\&=&\int_{0}^{\infty}e^{-z\cosh y}F(v,y)dy,
 \end{eqnarray*}
 where $$F(v,y)=\frac{i}{1-e^{\frac{iv}{\beta}}e^{\frac{y}{\beta}}e^{\frac{i\pi}{\beta}}}+\frac{i}{1-e^{\frac{iv}{\beta}}e^{-\frac{y}{\beta}}e^{\frac{i\pi}{\beta}}}
 -\frac{i}{1-e^{\frac{iv}{\beta}}e^{\frac{y}{\beta}}e^{-\frac{i\pi}{\beta}}}-\frac{i}{1-e^{\frac{iv}{\beta}}e^{-\frac{y}{\beta}}e^{-\frac{i\pi}{\beta}}}.$$
 Then it's easy to see that
$$F(v,y)=\frac{2\sin\frac{\pi}{\beta}[\cos\frac{\pi}{\beta}-\cos\frac{v}{\beta}\cosh\frac{y}{\beta}]}{[\cosh\frac{y}{\beta}-\cos\frac{v-\pi}{\beta}][\cosh\frac{y}{\beta}-\cos\frac{v+\pi}{\beta}]}. $$

Thus,
\[
\int_{L_1+L_2}e^{z[\cos a]}\frac{1}{1-e^{-\frac{i(a-v)}{\beta}}}da = E(z,v).
\]
Thus, after expanding we have that
\[
P(z,v) = 2\pi \beta \displaystyle \Sigma_{k,\ -\pi <v+2k\beta \pi<\pi}e^{z\cos(v+2k\beta \pi)}
+ \pi \beta \displaystyle \Sigma_{k,v+2k\beta \pi=\pm\pi}e^{-z} +E(z,v).
\]
The proof is complete.
\end{proof}
 Next we state the representation formulas for the second derivatives of the convolution of the heat kernel with a function, which are solutions to the parabolic equations with $f$ as right hand side .
\begin{lem}\label{representation of the derivative of singular integrals}Suppose $f$ is supported in $A_{10}$, then  \begin{eqnarray*}& &\frac{\partial}{\partial t}(H\ast f)(x)
\\&=&f(x,t)+\int_{0}^{t}\int_{A_{10}}\frac{\partial}{\partial t}[H(x,y,t-\tau)][f(y,\tau)-f(x,\tau)]dyd\tau
\\& &-\int_{0}^{t}f(x,\tau)d\tau\int_{\partial A_{10}}<\nabla_{y}H(x,y,t-\tau),n(y)>ds_y.
\end{eqnarray*}

Furthermore, for any $1\leq i\leq m$ the following formulas are true.
\begin{eqnarray*}\frac{\partial^2}{\partial r\partial s_i}(H\ast f)
&=&\int_{0}^{t}\int_{A_{10}}\frac{\partial^2}{\partial r\partial s_i}[H(x,y,t-\tau)][f(y,\tau)-f(x,\tau)]dyd\tau
\\& &-\int_{0}^{t}f(x,\tau)d\tau\int_{\partial A_{10}}\frac{\partial}{\partial r}H(x,y,t-\tau)n_i(y)ds_y;
\end{eqnarray*}
\begin{eqnarray*}\frac{1}{r}\frac{\partial^2}{\partial \theta\partial s_i}(H\ast f)
&=&\int_{0}^{t}\int_{A_{10}}\frac{1}{r}\frac{\partial^2}{\partial \theta\partial s_i}[H(x,y,t-\tau)][f(y,\tau)-f(x,\tau)]dyd\tau
\\& &-\int_{0}^{t}f(x,\tau)d\tau\int_{\partial A_{10}}\frac{1}{r}\frac{\partial}{\partial \theta}H(x,y,t-\tau)n_i(y)ds_y.
\end{eqnarray*}
$n$ stands for the inward normal of $\partial A_{10}$.
\end{lem}
\begin{proof}{of Lemma \ref{representation of the derivative of singular integrals}:} We only  prove the representation formula for $\frac{\partial^2}{\partial r\partial s_i}(H\ast f)$, the formula  for $\frac{\partial}{\partial t}(H\ast f)$ and $\frac{1}{r}\frac{\partial^2}{\partial r \partial s_i}(H\ast f)$ are the same. It suffices to show that
$$\lim_{h\rightarrow 0}\frac{\partial^2}{\partial r\partial s_i}\int_{0}^{t}\int_{A_{10}}\frac{\partial^2}{\partial r\partial s_i}[H(x,y,t-\tau)\eta_h(y)]f(y,\tau)dyd\tau$$
exists  and equals the right hand side, provided $\eta_h$ is properly chosen. \\

The crucial point is that, we can choose a good $\eta_h$ away from the singularity. Suppose $x$ is not on the singular hypersurface $\{0\}\times \mathbb{R}^{m}$, there is a ball $B_{h}(x)$ of radius $h$ which does not only miss $\{0\}\times \mathbb{R}^{m}$, but also has argument range less than $\beta\pi$. Then by the choice of $B_{h}(x)$, $d(x,y)$ is smooth in $y$ in $ B_{h}(x)$. We choose a smooth function $\psi$ such that $\psi(u)=0$ when $u\leq \frac{1}{2}$; $\psi(u)=1$ when $u\geq 1$. Then by letting $\eta_h(y)=\psi(\frac{d(x,y)}{h})$, the proof is complete, as in \cite{GT}.
\end{proof}

The next lemma is obvious.
\begin{prop}\label{Integral formula for compact supported weak solution}Suppose $u$ is compactly supported  in $A_{10}$. Suppose $u\in L^2[0,T;H_0^{1}(A_{10})]$ and  $\frac{\partial u}{\partial t}\in L^2[0,T;H^{-1}(A_{10})]$. Suppose $u$ solves $(\Delta -\frac{\partial}{\partial t})u=f$ in the weak sense over $A_{10}\times[0,T]$ and $u$ has zero initial value.
 Suppose $f\in C^{\alpha,\beta}(A_{10}\times[0,T])$. Then $$u=-(H\ast f)(x,t)
=-\int_{0}^{t}\int_{A_{10}}H(x,y,t-\tau)f(y,\tau)dyd\tau.$$
\end{prop}

\section{H\"older estimate of the singular integrals and proof of the main Schauder estimates in Theorem \ref{Schauder estimate of the linear equation: single divisor}.\label{Holder estimate of the singular integrals and proof of the main Schauder estimates}}
This section is devoted to the proof of Theorem \ref{Schauder estimate of the linear equation: single divisor}, by establishing
proposition \ref{Holder estimate of the spatial  derivative of the singular integral},
\ref{Holder estimate of the time derivative of the singular integral},
\ref{time holder estimate:spatial derivative},  \ref{time holder estimate:time derivative}, \ref{Schauder estimate for the weak solution which is compactly supported},  Theorem \ref{Schauder estimate: constant coefficient}.  All of these are achieved from the properties of
the heat kernel in Theorem \ref{all the properties of the h.k1} (which will be proved in the later part of this article). However, for the sake of being precise we won't quote directly  Theorem \ref{all the properties of the h.k1}, but will appeal to the lemmas which Theorem \ref{all the properties of the h.k1} is consisted of. Notice that  the estimate in this section are more or less standard, since we are equipped with Theorem \ref{all the properties of the h.k1}. 

\begin{prop}\label{Schauder estimate for the weak solution which is compactly supported}Suppose $u$ is compactly supported  in $A_{1}$. Suppose $u\in L^2[0,T;H_0^{1}(A_{1})]$ and  $\frac{\partial u}{\partial t}\in L^2[0,T;H^{-1}(A_{1})]$. Suppose $u$ solves $(\Delta -\frac{\partial}{\partial t})u=f$ in the weak sense over $A_{1}\times[0,T]$ and $u$ has zero initial value. Suppose $f\in C^{\alpha,\frac{\alpha}{2},\beta}(A_{1}\times[0,T])$.
  Then there exists a constant C as in Definition 
\ref{Dependence of the constant C} such that  the following estimates hold.$$[\sqrt{-1}\partial \bar{\partial} u]_{\alpha,\frac{\alpha}{2},\beta,A_1\times [0,T]}\leq C[f]_{\alpha,\frac{\alpha}{2},\beta,A_{1}\times [0,T]}$$
  and
  $$[\sqrt{-1}\frac{\partial}{\partial t} u]_{\alpha,\frac{\alpha}{2},\beta,A_1\times [0,T]}\leq C[f]_{\alpha,\frac{\alpha}{2},\beta,A_{1}\times [0,T]}.$$
\end{prop}
\begin{proof}{of proposition \ref{Schauder estimate for the weak solution which is compactly supported}:}, By rescaling,
this is  a direct corollary of
proposition \ref{Holder estimate of the spatial  derivative of the singular integral},
 \ref{Holder estimate of the time derivative of the singular integral},
 \ref{time holder estimate:spatial derivative},  \ref{time holder estimate:time derivative}, and the representation formula in proposition \ref{Integral formula for compact supported weak solution}. Just notice two points. First, we can take $u$ to be supported in $A_1$ so we don't have to shrink domain. Second, by the same ideas in \cite{Don},
  all the terms in
 $\sqrt{-1}\partial \bar{\partial} u$ except the term  $\Delta_{\beta}u $ are estimated by proposition \ref{Holder estimate of the spatial  derivative of the singular integral},
 \ref{Holder estimate of the time derivative of the singular integral},
 \ref{time holder estimate:spatial derivative},  \ref{time holder estimate:time derivative}. Then to estimate  $\Delta_{\beta}u$, it suffices
to use the equation $(\Delta -\frac{\partial}{\partial t})u=f$ itself (as $\Delta=\Delta_{\beta}+\Sigma_{i=1}^{m}\frac{\partial^2}{\partial s_i^2}$).
\end{proof}
\begin{cor}\label{Schauder estimate for the weak solution which is compactly supported in balls}Suppose $u$ is compactly supported  in $B_{1}$. Suppose $u\in L^2[0,T;H_0^{1}(B_{1})]$ and  $\frac{\partial u}{\partial t}\in L^2[0,T;H^{-1}(B_{1})]$. Suppose $u$ solves $(\Delta -\frac{\partial}{\partial t})u=f$ in the weak sense over $B_{1}\times[0,T]$ and $u$ has zero initial value. Suppose $f\in C^{\alpha,\frac{\alpha}{2},\beta}(B_{1}\times[0,T])$.
    Then there exists a constant C (as in Definition 
\ref{Dependence of the constant C}) such that  the following estimates hold. $$[\sqrt{-1}\partial \bar{\partial} u]_{\alpha,\frac{\alpha}{2},\beta,B_1\times [0,T]}\leq C[f]_{\alpha,\frac{\alpha}{2},\beta,B_{1}\times [0,T]}$$
  and
  $$[\sqrt{-1}\frac{\partial}{\partial t} u]_{\alpha,\frac{\alpha}{2},\beta,B_1\times [0,T]}\leq C[f]_{\alpha,\frac{\alpha}{2},\beta,B_{1}\times [0,T]}.$$
\end{cor}
\begin{proof}{of Corollary \ref{Schauder estimate for the weak solution which is compactly supported in balls}:}  It's actually straight forward to
deduce this proposition using proposition \ref{Schauder estimate for the weak solution which is compactly supported}. To fully illustrate the equivalence relation between $B_1$ and $A_1$, we shall include the detail here. We divide the situation into two cases in the following.\\

Case 1: When $dist\{B_1(p),\ \{0\}\times \mathbb{R}^{m}\}\geq 1000$, it's obvious since the equation is regular away from the singularity.\\

Case 2: When $dist\{B_1(p),\ \{0\}\times \mathbb{R}^{m}\}\leq  1000$, by translation invariance we can assume $\widehat{p}=0$. Then $B_1\subset A_{1100}$ and we have
\begin{eqnarray*}& &[\sqrt{-1}\partial \bar{\partial} u]_{\alpha,\frac{\alpha}{2},\beta,B_1(p)\times [0,T]}
\\& \leq & [\sqrt{-1}\partial \bar{\partial} u]_{\alpha,\frac{\alpha}{2},\beta,A_{1100}\times [0,T]}
\\& \leq & [f]_{\alpha,\frac{\alpha}{2},\beta,A_{20000}\times [0,T]}
\\& \leq & [f]_{\alpha,\frac{\alpha}{2},\beta,B_{30000}(p)\times [0,T]}.
\end{eqnarray*}

  Then since both $u$ and $f$ are supported in $B_1(p)$, the proof is completed.
\end{proof}

The next result is on estimates of spatial derivatives of the singular integral.\begin{prop}\label{Holder estimate of the spatial  derivative of the singular integral}There exists a $C$ (  as in Definition 
\ref{Dependence of the constant C}) independent of $T$ such that for any $\mathfrak{D}\in \mathfrak{T}$, the following estimate holds.
$$[\mathfrak{D}(H\ast f)]_{\alpha,\beta,A_1\times[0,T]}\leq C[f]_{\alpha,\beta,A_{10}\times[0,T]}.$$
\end{prop}
\begin{proof}{of proposition \ref{Holder estimate of the spatial  derivative of the singular integral}:} The proof is similar to Donaldson's estimates in \cite{Don}. We only prove the statement for
 $\frac{\partial^2}{\partial r\partial s_i}$ with some fixed $i$. The statements for $\frac{\partial^2}{\partial s_i \partial s_j},\ i,j\leq m$ and $\frac{1}{r}\frac{\partial^2}{\partial s_i\partial \theta},\ i\leq m$ are proved in exactly the same way.
 \\

 By rescaling it suffices to assume $|f|_{\alpha,\beta}=1$. Denote $|x_1-x_2|=\delta$.\\
Case1: $x_1=0$ (or $x_2=0$)  , we have
\begin{eqnarray*}& &\frac{\partial^2}{\partial r\partial s_i}(H\ast f)(0,t)-\frac{\partial^2}{\partial r\partial s_i}(H\ast f)(x_2,t)
\\&=&I_1+I_2+I_3+I_4+I_5+I_6,
\end{eqnarray*}
where \begin{eqnarray*}
  & &I_1=\int_{0}^{t}\int_{\{|y|\leq 64\delta\}\cap A_{10}}\frac{\partial^2}{\partial r\partial s_i}[H(0,y,t-\tau)][f(y,\tau)-f(0,\tau)]dyd\tau;
  \\& &I_2=-\int_{0}^{t}\int_{\{|y|\leq 64\delta\}\cap A_{10}}\frac{\partial^2}{\partial r\partial s_i}[H(x_2,y,t-\tau)][f(y,\tau)-f(x_2,\tau)]dyd\tau;
  \\& &I_3=\int_{0}^{t}\int_{\{|y|\geq 64\delta\}\cap A_{10}}[\frac{\partial^2}{\partial r\partial s_i}H(0,y,t-\tau)-\frac{\partial^2}{\partial r\partial s_i}H(x_2,y,t-\tau)]\}\\& &\quad \quad  \ \times[f(y,\tau)-f(0,\tau)]dyd\tau;
   \\& &I_4=\int_{0}^{t}[f(x_2,\tau)-f(0,\tau)]d\tau\int_{\{|y|\geq 64\delta\}\cap A_{10}}\frac{\partial^2}{\partial r\partial s_i}[H(x_2,y,t-\tau)]dy;
   \\& &I_5=-\int_{0}^{t}[f(0,\tau)-f(x_2,\tau)]d\tau\int_{\partial A_{10}}[\frac{\partial}{\partial r}H(0,y,t-\tau)]n_{i}(y)ds_y;
    \\& &I_6=-\int_{0}^{t}f(x_2,\tau)d\tau\int_{\partial A_{10}}[\frac{\partial}{\partial r}H(0,y,t-\tau)-\frac{\partial}{\partial r}H(x_2,y,t-\tau)]n_{i}(y)ds_y.
         \end{eqnarray*}
          
On $I_1$, by Item 1 in Lemma \ref{integral of time derivative of hk w.r.t time}, we have  $\int_{0}^{+\infty}|\frac{\partial^2}{\partial r\partial s_i}H(0,y,t)|dt\leq C|y|^{-(m+2)}$. Then
$$|I_1|\leq C\int_{\{|y|\leq 64\delta\}\cap A_{10}}|y|^{-(m+2)+\alpha}dy\leq C\delta^{\alpha}.$$

On $I_2$, by item 2 in Lemma  \ref{integral of time derivative of hk w.r.t time} and rescaling, we have $$\int_{0}^{+\infty}|\frac{\partial^2}{\partial r\partial s_i}H(x_2,y,t)|dt\leq C|y-x_2|^{-(m+2)}.$$ Thus
$$|I_2|\leq C\int_{\{|y|\leq 64\delta\}\cap A_{10}}|y-x_2|^{-(m+2)+\alpha}dy\leq C\delta^{\alpha}.$$

On $I_3$, by Lemma \ref{integral of gradient of time derivative of the hk} and rescaling, we have $$\int_{0}^{+\infty}|\frac{\partial^2}{\partial r\partial s_i}H(x_2,y,t)-\frac{\partial^2}{\partial r\partial s_i}H(0,y,t)|dt\leq C\delta|y|^{-(m+3)}.$$ Then
$$|I_3|\leq C\delta\int_{\{|y|\geq 64\delta\}\cap A_{10}}|y|^{-(m+2)-1+\alpha}dy\leq C\delta^{\alpha}.$$

On $I_4$, by Lemma \ref{integral of gradient of hk w.r.t time}  we have $\int_{0}^{\infty}|\frac{\partial}{\partial r}[H(x_2,y,\tau)]|d\tau\leq C\delta^{-(m+1)}$ when $|x_2|=\delta$ and $|y|=64\delta$. With the conclusion of Lemma  \ref{integral of gradient of hk w.r.t time} in mind, we compute that
\begin{eqnarray*}& &|I_4|
\\&\leq&\int_{0}^{t}|f(x_2,\tau)-f(0,\tau)|d\tau\int_{|y|=64\delta}|\frac{\partial}{\partial r}[H(x_2,y,t-\tau)|ds_y
\\& &+\int_{0}^{t}|f(x_2,\tau)-f(0,\tau)|d\tau\int_{\partial A_{10}}|\frac{\partial}{\partial r}[H(x_2,y,t-\tau)|ds_y
\\&\leq&C\delta^{\alpha}[\int_{|y|=64\delta}\int_{0}^{\infty}|\frac{\partial}{\partial r}[H(x_2,y,\tau)]|d\tau ds_y
+\int_{\partial A_{10}}\int_{0}^{\infty}|\frac{\partial}{\partial r}[H(x_2,y,\tau)]|d\tau ds_y]\\&\leq&C\delta^{\alpha}.
\end{eqnarray*}

On $I_5$, we have
\begin{eqnarray*}& &|\int_{0}^{t}[f(0,\tau)-f(x_2,\tau)]d\tau\int_{\partial A_{10}}<\frac{\partial}{\partial r}H(0,y,t-\tau),n_{i}(y)>ds_y|
\\&\leq& \delta^{\alpha} \int_{0}^{t}d\tau\int_{\partial A_{10}}|\frac{\partial}{\partial r}H(0,y,t-\tau),n_{i}(y)|ds_y.
\end{eqnarray*}
Still by Lemma \ref{integral of gradient of hk w.r.t time}, we have over ${\partial A_{10}}$ that
$$\int_{0}^{t}|\frac{\partial}{\partial r}H(0,y,t-\tau),n_{i}(y)|d\tau\leq C,$$
then $|I_5|\leq C\delta^{\alpha}$.\\

The estimate of $I_6$ is the reason we use the polydisk $A_{10}$ but not balls. Recall from Definition  that
$$\partial A_{10}=[\partial D_{10}\times \widehat{B}_{10}]\cup [ D_{10}\times \partial\widehat{B}_{10}].$$
Over $\partial D_{10}\times \widehat{B}_{10}$, we have $n_{i}(y)=0$. Therefore the second inequality in Lemma \ref{Holder estimate of the integral of spatial derivative of h.k with respect to time} directly implies that
\begin{eqnarray*}& &\int_{0}^{t}d\tau\int_{\partial A_{10}}|\frac{\partial}{\partial r}H(0,y,t-\tau)-\frac{\partial}{\partial r}H(x_2,y,t-\tau),n_{i}(y)|ds_y
\\&\leq& \int_{D_{10}\times \partial\widehat{B}_{10}}\int_{0}^{+\infty}|\frac{\partial}{\partial r}H(0,y,t-\tau)-\frac{\partial}{\partial r}H(x_2,y,t-\tau)| d\tau ds_y
\\&\leq& C\delta^{\rho},
\end{eqnarray*}
which means $|I_6|\leq C\delta^{\rho}$. Thus the proof is completed in this case.\\

Case 2: We assume $\min\{P(x_1), P(x_2)\}=l\geq 3\delta$, where $P(x)$ is norm of the projection of $x$ to the  $\mathbb{R}^2$ component. This case is easier to handle since both points are far away from the singularity. Without loss of generality we can assume
$\widehat{x_1}=0$ and then $|x_1|=l \delta$. Then we split the difference as
\begin{eqnarray*}& &\frac{\partial^2}{\partial r\partial s_i}(H\ast f)(x_1,t)-\frac{\partial^2}{\partial r\partial s_i}(H\ast f)(x_2,t)
\\&=&IV_1+IV_2+IV_3+IV_4+IV_5+IV_6+IV_7+IV_8,
\end{eqnarray*}
where
\begin{eqnarray*}
  & &IV_1=\int_{0}^{t}\int_{\{|y|\leq 2l\delta,|y-x_1|\leq 64\delta\}\cap A_{10}}\frac{\partial^2}{\partial r\partial s_i}[H(x_1,y,t-\tau)][f(y,\tau)-f(x_1,\tau)]dyd\tau;
 \\& &IV_2=-\int_{0}^{t}\int_{\{|y|\leq 2l\delta,|y-x_1|\leq 64\delta\}\cap A_{10}}\frac{\partial^2}{\partial r\partial s_i}[H(x_2,y,t-\tau)][f(y,\tau)-f(x_2,\tau)]dyd\tau;
  \\& &IV_3=\int_{0}^{t}\int_{\{|y|\leq 2l\delta,|y-x_1|\geq 64\delta\}\cap A_{10}}[\frac{\partial^2}{\partial r\partial s_i}H(x_1,y,t-\tau)-\frac{\partial^2}{\partial r\partial s_i}H(x_2,y,t-\tau)]\}
  \\& &\quad \quad \quad \ \times[f(y,\tau)-f(x_1,\tau)]dyd\tau;
  \\& &IV_4=\int_{0}^{t}[f(x_2,\tau)-f(x_1,\tau)]d\tau\int_{\{|y|\leq 2l\delta,|y-x_1|\geq 64\delta\}\cap A_{10}}\frac{\partial^2}{\partial r\partial s_i}[H(x_2,y,t-\tau)]dy;
  \\& &IV_5=\int_{0}^{t}\int_{\{|y|\geq 2l\delta\}\cap A_{10}}[\frac{\partial^2}{\partial r\partial s_i}H(x_1,y,t-\tau)-\frac{\partial^2}{\partial r\partial s_i}H(x_2,y,t-\tau)]\}
  \\& & \quad \quad \quad \ \times [f(y,\tau)-f(x_1,\tau)]dyd\tau;
  \\& &IV_6=\int_{0}^{t}[f(x_2,\tau)-f(x_1,\tau)]d\tau\int_{\{|y|\geq 2l\delta\}\cap A_{10}}\frac{\partial^2}{\partial r\partial s_i}[H(x_2,y,t-\tau)]dy;
 \\& &IV_7=\int_{0}^{t}[f(x_1,\tau)-f(x_2,\tau)]d\tau\int_{\partial A_{10}}<\frac{\partial}{\partial r}H(x_1,y,t-\tau),n_{i}(y)>ds_y;
  \\& &IV_8=\int_{0}^{t}f(x_2,\tau)d\tau\int_{\partial A_{10}}<\frac{\partial}{\partial r}H(x_1,y,t-\tau)-\frac{\partial}{\partial r}H(x_2,y,t-\tau),n_{i}(y)>ds_y.
\end{eqnarray*}
  Mostly we will proceed as in Case 1, except on some special items.\\
  
   On $IV_1$, from the same analysis as in $I_2$ in Case 1, we have
\begin{eqnarray*}& &|IV_1|\leq \int_{0}^{t}\int_{\{|y-x_1|\leq 64\delta\}\cap A_{10}}\frac{\partial^2}{\partial r\partial s_i}[H(x_1,y,t-\tau)][f(y,\tau)-f(x_1,\tau)]dyd\tau
\\&\leq&C\int_{\{|y-x_1|\leq 64\delta\}\cap A_{10}}|y-x_1|^{-(m+2)+\alpha}dy\leq C\delta^{\alpha}.
\end{eqnarray*}

Second we compute \begin{eqnarray*}& &|IV_2|\leq \int_{0}^{t}\int_{\{|y-x_2|\leq 65\delta\}\cap A_{10}}\frac{\partial^2}{\partial r\partial s_i}[H(x_2,y,t-\tau)][f(y,\tau)-f(x_2,\tau)]dyd\tau
\\&\leq&C\int_{\{|y-x_2|\leq 65\delta\}\cap A_{10}}|y-x_2|^{-(m+2)+\alpha}dy\leq C\delta^{\alpha}.
\end{eqnarray*}

On $IV_3$, notice that when $|y|\leq 2l\delta$ and $|y-x|\geq 63\delta$, by Lemma 
\ref{integral of third derivative of h.k w.r.t time: when r big},  we have \[\int_{0}^{\infty}|\nabla_x\frac{\partial^2}{\partial r\partial s_i}H(x,y,\tau)|d\tau\leq C|y-x|^{-(m+3)}\] for any
$x$ in the line segment $\overline{x_1x_2}$. Thus
\begin{eqnarray*}& &|IV_3|
\\&\leq&C\delta\int_{\{|y-x_1|\geq 64\delta\}\cap A_{10}}|y-x_1|^{-(m+3)+\alpha}dy\leq C\delta^{\alpha}.
\end{eqnarray*}

On $IV_4$, similarly  we have
\begin{eqnarray*}& &|IV_4|
\\&\leq& C\delta^{\alpha}\int_{0}^{t}\int_{|y|=2l\delta}|\frac{\partial}{\partial r}[H(x_2,y,t-\tau)]|d\tau ds(y)\\& &+C\delta^{\alpha}\int_{0}^{t}\int_{|y-x_1|=64\delta}|\frac{\partial}{\partial r}[H(x_2,y,t-\tau)]|d\tau ds(y)
\\& &+C\delta^{\alpha}\int_{0}^{t}\int_{\partial A_{10}}|\frac{\partial}{\partial r}[H(x_2,y,t-\tau)]|d\tau ds(y).
\end{eqnarray*}
By  Lemma \ref{integral of gradient of hk w.r.t time} we have  \[\int_{0}^{\infty}|\frac{\partial}{\partial r}[H(x_2,y,\tau)]|d\tau\leq C|y-x_2|^{-(m+1)}\] when $|y-x_1|=64\delta$. Therefore when $|y|=2l\delta$, since
$l\geq 2$, we have
$$\int_{0}^{\infty}|\frac{\partial}{\partial r}[H(x_2,y,\tau)]|d\tau\leq C|y-x_2|^{-(m+1)}\leq C\delta^{-(m+1)},$$
thus \[\int_{|y|=2l\delta}\int_{0}^{\infty}|\frac{\partial}{\partial r}[H(x_2,y,\tau)]|d\tau ds(y)\leq C.\]
When $|y-x_1|=64\delta$, we also have
$$\int_{0}^{\infty}|\frac{\partial}{\partial r}[H(x_2,y,\tau)]|d\tau\leq C|y-x_2|^{-(m+1)}\leq C\delta^{-(m+1)},$$
thus \[\int_{|y-x_1|=64\delta}\int_{0}^{\infty}|\frac{\partial}{\partial r}[H(x_2,y,\tau)]|d\tau ds(y)\leq C.\] According to Lemma \ref{integral of gradient of hk w.r.t time} and the assumption that $|x_2|\leq 1$, we have
$$\int_{0}^{t}\int_{\partial A_{10}}|\frac{\partial}{\partial r}[H(x_2,y,t-\tau)]|d\tau ds(y)\leq C.$$
Then $|IV_4|\leq C\delta^{\alpha}$.\\

 $IV_6$ can be estimated easily in the way $IV_4$ was estimated.
$IV_5$ can be estimated exactly as $I_3$ in Case 1. The estimates of $IV_7$ and $IV_8$ are exactly  the same as the estimates of
$I_5$ and $I_6$. Thus case 2 is all set.\\

Finally, given any $x_1,\ x_2$ such that $|x_1-x_2|=\delta$, we always have $$|\frac{\partial^2}{\partial r\partial s_i}(H\ast f)(x_1,t)-\frac{\partial^2}{\partial r\partial s_i}(H\ast f)(x_2,t)|\leq C\delta^{\alpha}$$
The reason is that by translation invariance  we can assume $\widehat{x_1}=0$ and $P(x_1)\leq P(x_2)$. If $P(x_1)\leq 3\delta$, we have:
\begin{eqnarray*}& &|\frac{\partial^2}{\partial r\partial s_i}(H\ast f)(x_1,t)-\frac{\partial^2}{\partial r\partial s_i}(H\ast f)(x_2,t)|
\\&\leq& |\frac{\partial^2}{\partial r\partial s_i}(H\ast f)(x_1,t)-\frac{\partial^2}{\partial r\partial s_i}(H\ast f)(0,t)|
\\& &+|\frac{\partial^2}{\partial r\partial s_i}(H\ast f)(x_2,t)-\frac{\partial^2}{\partial r\partial s_i}(H\ast f)(0,t)|
\\&\leq& C\delta^{\alpha}
\end{eqnarray*}
from Case 1. If $P(x_1)\geq 3\delta$ we are in Case 2. Thus the proof is completed.
\end{proof}

\begin{prop}\label{Holder estimate of the time derivative of the singular integral}There exists a $C$ (as in Definition 
\ref{Dependence of the constant C}) independent of $T$ such that
$$[\frac{\partial}{\partial t}(H\ast f)]_{\alpha,\beta,A_1\times[0,T]}\leq C[f]_{\alpha,\beta,A_{10}\times[0,T]}.$$
\end{prop}
\begin{proof}{of proposition \ref{Holder estimate of the time  derivative of the singular integral}:}\ The framework of this lemma is simpler than that of Lemma \ref{Holder estimate of the spatial  derivative of the singular integral}. We will repeatedly use the assumptions $|x_1|,|x_2|\leq 1$. By rescaling it suffices to assume $|f|_{\alpha,\beta}=1$. Denote $|x_1-x_2|=\delta$. We compute
\begin{eqnarray*}& &\frac{\partial}{\partial t}(H\ast f)(x_1,t)-\frac{\partial}{\partial t}(H\ast f)(x_2,t)
\\&=&f(x_1,t)-f(x_2,t)+Z_1+Z_2+Z_3+Z_4+Z_5+Z_6.
\end{eqnarray*}
\begin{eqnarray*}& 
  Z_1&=\int_{0}^{t}\int_{\{|y-x_1|\leq 64\delta\}\cap A_{10}}\frac{\partial}{\partial t}[H(x_1,y,t-\tau)][f(y,\tau)-f(x_1,\tau)]dyd\tau;
 \\& Z_2&=-\int_{0}^{t}\int_{\{|y-x_1|\leq 64\delta\}\cap A_{10}}\frac{\partial}{\partial t}[H(x_2,y,t-\tau)][f(y,\tau)-f(x_2,\tau)]dyd\tau;
  \\& Z_3&=\int_{0}^{t}\int_{\{|y-x_1|\geq 64\delta\}\cap A_{10}}[\frac{\partial}{\partial t}H(x_1,y,t-\tau)-\frac{\partial}{\partial t}H(x_2,y,t-\tau)]\}
  \\& &[f(y,\tau)-f(x_1,\tau)]dyd\tau;
 \\& Z_4&=\int_{0}^{t}[f(x_2,\tau)-f(x_1,\tau)]d\tau\int_{\{|y-x_1|\geq 64\delta\}\cap A_{10}}\frac{\partial}{\partial t}[H(x_2,y,t-\tau)]dy;
 \\& Z_5&=-\int_{0}^{t}[f(x_1,\tau)-f(x_2,\tau)]d\tau\int_{\partial A_{10}}<\nabla_yH(x_1,y,t-\tau),n(y)>ds_y;
  \\& Z_6&=-\int_{0}^{t}f(x_2,\tau)d\tau\int_{\partial A_{10}}<\nabla_yH(x_1,y,t-\tau)-\nabla_yH(x_2,y,t-\tau),n(y)>ds_y.
\end{eqnarray*}
On $Z_1$, from  Lemma \ref{integral of time derivative of hk w.r.t time}, we have  $\int_{0}^{+\infty}|\frac{\partial}{\partial t}H(x_1,y,t)|dt\leq C|y-x_1|^{-(m+2)}$. Then
$$|Z_1|\leq C\int_{\{|y-x_1|\leq 64\delta\}\cap A_{10}}|y|^{-(m+2)+\alpha}dy\leq C\delta^{\alpha}.$$

On $Z_2$, again from Lemma \ref{integral of time derivative of hk w.r.t time} and rescaling, we have \[\int_{0}^{+\infty}|\frac{\partial}{\partial t}H(x_2,y,t)|dt\leq C|y-x_2|^{-(m+2)}.\] Thus
$$|Z_2|\leq C\int_{\{|y-x_1|\leq 64\delta\}\cap A_{10}}|y-x_2|^{-(m+2)+\alpha}dy\leq C\delta^{\alpha}.$$
On $Z_3$, from Lemma \ref{integral of gradient of time derivative of the hk} and rescaling, we have \[\int_{0}^{+\infty}|\frac{\partial}{\partial t}H(x_2,y,t)-\frac{\partial}{\partial t}H(x_1,y,t)|dt\leq C\delta|y-x_1|^{-(m+3)}.\] Then
$$|Z_3|\leq C\delta\int_{\{|y-x_1|\geq 64\delta\}\cap A_{10}}|y-x_1|^{-(m+2)-1+\alpha}dy\leq C\delta^{\alpha}.$$
On $Z_4$, since we have $\int_{0}^{\infty}|\nabla_yH(x_2,y,\tau)|d\tau\leq C\delta^{-(m+1)}$ for $|x_2|=\delta$ and $|y-x_1|=64\delta$, then
\begin{eqnarray*}& &|Z_4|
\\&=&|\int_{0}^{t}[f(x_2,\tau)-f(x_1,\tau)]d\tau\int_{\{|y-x_1|\geq 64\delta\}\cap A_{10}}\Delta[H(x_2,y,t-\tau)]dy|
\\&\leq&|\int_{0}^{t}[f(x_2,\tau)-f(x_1,\tau)]d\tau\int_{\partial[\{|y-x_1|=64\delta\}\cap A_{10}]}
\\& &<\nabla_y[H(x_2,y,t-\tau)],n(y)>ds_y|
\\&\leq&C\delta^{\alpha}[\int_{|y-x_1|=64\delta}\int_{0}^{\infty}|\nabla_y[H(x_2,y,\tau)]|d\tau ds_y
\\& &+\int_{\partial A_{10}}\int_{0}^{\infty}|\nabla_y[H(x_2,y,\tau)]|d\tau ds_y]
\\&\leq&C\delta^{\alpha}.
\end{eqnarray*}
On $Z_5$, we have
\begin{eqnarray*}& &|\int_{0}^{t}[f(x_1,\tau)-f(x_2,\tau)]d\tau\int_{\partial A_{10}}<\nabla_yH(x_1,y,t-\tau),n(y)>ds_y|
\\&\leq& \delta^{\alpha} \int_{0}^{t}d\tau\int_{\partial A_{10}}|\nabla_yH(x_1,y,t-\tau)|ds_y.
\end{eqnarray*}
Still by Lemma \ref{integral of gradient of hk w.r.t time}, we have over ${\partial A_{10}}$ that
$$\int_{0}^{t}|\nabla_yH(x_1,y,t-\tau))|d\tau\leq C,$$
then $|Z_5|\leq C\delta^{\alpha}$.\\

The estimate of $Z_6$ is the reason we use the polydisk $A_{10}$ but not balls. Again recall from the definitions  that
$$\partial A_{10}=[\partial D_{10}\times \widehat{B}_{10}]\cup [ D_{10}\times \partial\widehat{B}_{10}].$$
Over $\partial D_{10}\times \widehat{B}_{10}$, we have $n(y)=\frac{\partial}{\partial \acute{r}}$; Over $ D_{10}\times \partial\widehat{B}_{10}$, we have $n(y)=\frac{\partial}{\partial \bar{R}}$. The easy but crucial observation is that the normal vectors are all orthogonal to $\frac{\partial}{\partial \acute{\theta}}$. Therefore Lemma \ref{time integral of the bilinear derivatives of the heat kernel} directly implies that
\begin{eqnarray*}& &\int_{0}^{t}d\tau\int_{\partial A_{10}}|\nabla_yH(x_1,y,t-\tau)-\nabla_yH(x_2,y,t-\tau),n(y)|ds_y
\\&\leq& \delta\int_{0}^{t}d\tau\int_{\partial D_{10}\times \widehat{B}_{10}}|\nabla_{\bar{x}}\frac{\partial}{\partial \acute{r}}H[\bar{x}(y,t-\tau),y,t-\tau]|ds_y
\\& &+ \delta\int_{0}^{t}d\tau\int_{ (D_{10}\setminus\{0\})\times \partial\widehat{B}_{10}}|\nabla_{\bar{x}}\frac{\partial}{\partial \bar{R}}H[\bar{x}(y,t-\tau),y,t-\tau]|ds_y
\\&\leq& C\delta,
\end{eqnarray*}
which means $|Z_6|\leq C\delta$. Thus the proof is done  in this case.\\
\end{proof}
 In the following we shall study the timewise H\"older estimate of the spatial second derivatives and time derivative of the solution. The main conclusions in section \ref{Asymptotic behaviors of  the heat kernel}  play  important roles here.\\

\begin{prop}\label{time holder estimate:spatial derivative}There exists a $C$ (as in Definition 
\ref{Dependence of the constant C}) with the following property. For all second order differential operators $\mathfrak{D}\in \mathfrak{T}$, points $x\in A_1$ and $t$ such that $0\leq t<t+\delta\leq T$,  the following potential estimate holds.
$$|\mathfrak{D}(H\ast f)(x,t+\delta)-\mathfrak{D}(H\ast f)(x,t)|\leq C\delta^{\frac{\alpha}{2}}[f]_{\alpha,\beta,A_{10}\times[0,T]} .$$
\end{prop}
\begin{proof}{of Proposition \ref{time holder estimate:spatial derivative}:}\ It suffices to  assume $\delta\leq \frac{1}{10000}$ in the H\"older estimates.
 The case when $\delta\geq \frac{1}{10000}$ reduces to lower order estimates in section \ref{Appendix: some lower order estimates} which are a lot easier than the estimates we try to prove here.
  The necessary lower order estimates are stated in section \ref{Appendix: some lower order estimates}.
  In the estimate of $Q_4$ below, we applied Lemma \ref{Bound on the gradient of h.k when r small and t equals 1}
   which requires the $\delta$ to be small.\\

  By rescaling invariance of the heat kernel it suffices to assume $|f|_{\alpha,\beta}=1$. It suffices to only consider   $\mathfrak{D}=\frac{\partial^2}{ \partial r\partial s_i}$, since  the other derivatives are really similar. We compute
\begin{eqnarray*}& &(\frac{\partial^2}{ \partial r\partial s_i}H\ast f)(x,t+\delta)-(\frac{\partial^2}{ \partial r\partial s_i}H\ast f)(x,t)
\\&=&Q_1+Q_2+Q_3+Q_4+Q_5,
\end{eqnarray*}
where

\begin{eqnarray*}
&Q_1&= \int_{t-\delta}^{t+\delta}d\tau\int_{A_{10}}\frac{\partial^2}{ \partial r\partial s_i}H(x,y,t+\delta-\tau)[f(y,\tau)-f(x,\tau)]dy
\\&Q_2& =-\int_{t-\delta}^{t}d\tau\int_{A_{10}}\frac{\partial^2}{ \partial r\partial s_i}H(x,y,t-\tau)][f(y,\tau)-f(x,\tau)]dy
\\&Q_3& =\int_{0}^{t-\delta}d\tau\int_{A_{10}}[\frac{\partial^2}{ \partial r\partial s_i}H(x,y,t+\delta-\tau)-\frac{\partial^2}{ \partial r\partial s_i}H(x,y,t-\tau)]
\\&\times&[f(y,\tau)-f(x,\tau)]dy
\\&Q_4&=-\int_{t}^{t+\delta}f(x,\tau) d\tau\int_{\partial A_{10}}\frac{\partial}{ \partial r}H(x,y,t+\delta-\tau)n_i(y)ds_y
\\&Q_5&=-\int_{0}^{t}f(x,\tau) d\tau\int_{\partial A_{10}}[\frac{\partial}{ \partial r}H(x,y,t+\delta-\tau)-\frac{\partial}{ \partial r}H(x,y,t-\tau)]n_i(y)ds_y
\end{eqnarray*}
Next, using the estimates in Proposition \ref{proposition on integral of the product the heat kernel and distance function to the alpha}, we estimate that
$$\int_{\mathbb{R}^2\times \mathbb{R}^m}|\frac{\partial^2}{ \partial r\partial s_i}H(x,y,t)||y-x|^{\alpha}dy\leq Ct^{\frac{\alpha}{2}-1},$$ then we estimate
\begin{eqnarray*}& &|Q_1|
\\&\leq&  \int_{0}^{2\delta}dt\int_{\mathbb{R}^2\times \mathbb{R}^m}|\frac{\partial^2}{ \partial r\partial s_i}H(x,y,t)||y-x|^{\alpha}dy
\\&\leq& C \int_{0}^{2\delta}t^{\frac{\alpha}{2}-1} dt
\\&\leq& C\delta^{\frac{\alpha}{2}}.
\end{eqnarray*}
Similarly we have $|Q_2|\leq C\delta^{\frac{\alpha}{2}}$.\\

For $Q_3$, we compute
\begin{eqnarray*}& &|Q_3|
\\&\leq &
 \int_{-\infty}^{t-\delta}d\tau\int_{A_{10}}\ |\frac{\partial^2}{ \partial r\partial s_i}H(x,y,t+\delta-\tau)-\frac{\partial^2}{ \partial r\partial s_i}H(x,y,t-\tau)|\ |y-x|^{\alpha}dy
 \\&=&
 \int_{\delta}^{\infty}t^{\frac{\alpha}{2}-1}dt\int_{\frac{A_{10}}{t}}\ |\frac{\partial^2}{ \partial r\partial s_i}H(\frac{x}{\sqrt{t}},u,1+\frac{\delta}{t})-\frac{\partial^2}{ \partial r\partial s_i}H(\frac{x}{\sqrt{t}},u,1)|\ |u-\frac{x}{\sqrt{t}}|^{\alpha}du
  \\&\leq &
 C\delta\int_{\delta}^{\infty}t^{\frac{\alpha}{2}-2}dt\int_{R^2\times R^m}\sup_{1\leq s\leq 2}|\frac{\partial}{\partial s}\frac{\partial^2}{ \partial r\partial s_i}H(\frac{x}{\sqrt{t}},u,s)||u-\frac{x}{\sqrt{t}}|^{\alpha}du
\\ &\leq&C\delta\int^{\infty}_{\delta}t^{\frac{\alpha}{2}-2}dt
 \\&\leq&C\delta^{\frac{\alpha}{2}},
\end{eqnarray*}
where we applied the following estimate in Proposition \ref{proposition on integral of the product the heat kernel and distance function to the alpha}:  $$\int_{\mathbb{R}^2\times \mathbb{R}^m}\sup_{1\leq s\leq 2}|\frac{\partial}{\partial s} \frac{\partial^2}{ \partial r\partial s_i}H(\frac{x}{\sqrt{t}},u,s)||u-\frac{x}{\sqrt{t}}|^{\alpha}du\leq C.$$

On $Q_4$, we have to involve Lemma \ref{Bound on the gradient of h.k when r small and t equals 1}. We compute
\begin{eqnarray*}& &|Q_4|
\\&=&
 \int_{0}^{\delta}dt\int_{\partial A_{10}}|\frac{\partial}{ \partial r}H(x,y,t)|ds_y
 \\&=&\int_{0}^{\delta}t^{-1}dt\int_{\frac{\partial A_{10}}{t}}|\frac{\partial}{ \partial r}H(\frac{x}{t},u,1)|ds_u
\leq \int_{0}^{\delta}t^{-1}\cdot t^{-\frac{m+1}{2}}e^{-\frac{1}{t}}dt
 \\&\leq&C\delta.
\end{eqnarray*}

On $Q_5$, from Lemma \ref{integral of gradient of time derivative of the hk}, we have
\begin{eqnarray*}& &|Q_5|
\\&\leq &
 \delta\int_{0}^{t}d\tau\int_{\partial A_{10}}|\frac{\partial^2}{ \partial r\partial t}H(x,y,\bar{t}(x,y)-\tau)|ds_y
 \\&\leq&C\delta,
\end{eqnarray*}
where $t \leq \bar{t}(x,y)\leq t+\delta$.

\end{proof}
\begin{prop}\label{time holder estimate:time derivative}There exists a $C$ (as in Definition 
\ref{Dependence of the constant C}) with the following properties. For all points $x\in A_1$ and $t_1,t_2$ in  $[0,\ T]$, the following potential estimate holds. $$|(\frac{\partial H}{\partial t}\ast f)(x,t_1)-(\frac{\partial H}{\partial t}\ast f)(x,t_2)|\leq C[f]_{\alpha,\frac{\alpha}{2},\beta,A_{10}\times[0,T]}|t_1-t_2|^{\frac{\alpha}{2}} .$$
\end{prop}
\begin{proof}{of Proposition \ref{time holder estimate:time derivative}:}\ The proof is by the same way as in Lemma \ref{time holder estimate:spatial derivative}. We  apply the following decomposition
after rescaling with respect time and applying proposition \ref{proposition on integral of the product the heat kernel and distance function to the alpha}.
\begin{eqnarray*}& &\frac{\partial H}{\partial t}\ast f(x,t+\delta)-\frac{\partial H}{\partial t}\ast f(x,t)
\\&=& (\frac{\partial H}{\partial t}H\ast f)(x,t+\delta)-(\frac{\partial H}{\partial t}H\ast f)(x,t)
\\&=& \int_{t-\delta}^{t+\delta}d\tau\int_{A_{10}}\frac{\partial H}{\partial t}H(x,y,t+\delta-\tau)[f(y,\tau)-f(x,\tau)]dy
\\& &- \int_{t-\delta}^{t}d\tau\int_{A_{10}}\frac{\partial H}{\partial t}H(x,y,t-\tau)][f(y,\tau)-f(x,\tau)]dy
\\& &+ \int_{-\infty}^{t-\delta}d\tau\int_{A_{10}}[\frac{\partial H}{\partial t}H(x,y,t+\delta-\tau)-\frac{\partial H}{\partial t}H(x,y,t-\tau)]
\\& &\times[f(y,\tau)-f(x,\tau)]dy
\\& &-\int_{t}^{t+\delta}f(x,\tau) d\tau\int_{\partial A_{10}}<\nabla_yH(x,y,t+\delta-\tau),\ n(y)>ds_y
\\& &-\int_{0}^{t}f(x,\tau) d\tau\int_{\partial A_{10}}<\nabla_y[H(x,y,t+\delta-\tau)-H(x,y,t-\tau)],n_i(y)>ds_y
\\& &+f(x,t+\delta)- f(x,t).
\end{eqnarray*}

\end{proof}

The next theorem on interior estimates is important. As we mentioned earlier, the singular set is never considered as boundary. 

\begin{thm}\label{Schauder estimate: constant coefficient} There exists a $C(T)$ (as in Definition 
\ref{Dependence of the constant C} and depending on $T$)  with the following properties.  For any open set $\Omega \in \mathbb{R}^{m+2}$,   suppose $u\in C_0^{2+\alpha,1+\frac{\alpha}{2},\beta}[0,T]$. Then $u$ satisfies the following estimates 
\begin{eqnarray*}& &|u|^{(\star)}_{2+\alpha,1+\frac{\alpha}{2},\beta,{\Omega}\times[0,T]}\leq C(T)|(\Delta-\frac{\partial }{\partial t})u|^{(2)}_{\alpha,\frac{\alpha}{2},\beta,{\Omega}\times[0,T]};
\\& & |u|^{(\star)}_{2,\alpha,\beta,{\Omega}\times[0,T]}\leq C(T) |(\Delta-\frac{\partial }{\partial t})u|^{(2)}_{\alpha,\beta,{\Omega}\times[0,T]};
\\& &|u|^{[\star]}_{2+\alpha,1+\frac{\alpha}{2},\beta,{\Omega}\times[0,T]}\leq  C(T)  |(\Delta-\frac{\partial }{\partial t})u|^{[2]}_{\alpha,\frac{\alpha}{2},\beta,{\Omega}\times[0,T]}.
\end{eqnarray*}
\end{thm}
\begin{proof}{of Theorem \ref{Schauder estimate: constant coefficient}:} 
Using Proposition \ref{Schauder estimate for the weak solution which is compactly supported}, multiplication by a cutoff function, and  the lower order estimates in the appendix, it's routine to deduce the  estimates in Theorem \ref{Schauder estimate: constant coefficient}. We should apply Proposition \ref{Schauder estimate for the weak solution which is compactly supported} near the singularities and apply Proposition \ref{Schauder estimate for the weak solution which is compactly supported in balls} away from the singularities. Nevertheless, we would like to mention that
 the interpolation tricks we implicitly employed here are not the same ones as in \cite{GT}, although the tricks are extremely  simple and natural. 
 Notice that away from the singularity we should apply Corollary \ref{Schauder estimate for the weak solution which is compactly supported in balls}.
\end{proof}

To prove the part of Theorem \ref{Schauder estimate of the linear equation: single divisor} on existence of classical solutions, we need the following existence result 
over $E_T$. The power of the space  $W^{2,2}_{0}(E_T)$ is displayed here. The following result is definitely true when $\beta=1$ (when there is no singularity), and we believe this is more than well known to experts. 
\begin{prop}\label{Solvability of the operator with small oscillation in the flat cone} The parabolic operator  $\Delta_{\omega_D}-\frac{\partial }{\partial t}$ is surjective map among the following function spaces \[C_{0}^{2+\alpha,1+\frac{\alpha}{2},\beta}[0,T]:\rightarrow\ C^{\alpha,\frac{\alpha}{2},\beta}[0,T].\] 
\end{prop}
\begin{proof}{of Proposition \ref{Solvability of the operator with small oscillation in the flat cone}:} Though the aprori estimates in Theorem \ref{Schauder estimate: constant coefficient} hold, it's still necessary to show  the existence of a classical solution. Since Proposition \ref{Solvability of the operator with small oscillation in the flat cone} has its own independent interest and needs the estimates in $W^{2,2}(M\times[0,T])$ (which involve a bunch of other techniques), we refer the reader to the note \cite{WYQ}. 
\end{proof}

Now we begin addressing   the final step of the  main Schauder estimates and existence results in Theorem \ref{Schauder estimate of the linear equation: single divisor}, by using the  results  established earlier in this section.
\begin{proof}{ of 
Theorem \ref{Schauder estimate of the linear equation: single divisor}: }\ 
 First we show the aprori estimates hold. Namely, suppose $u\in C^{2+\alpha,1+\frac{\alpha}{2},\beta}[0,T]$, we prove that  the estimates in 
 Theorem \ref{Schauder estimate of the linear equation: single divisor} hold. These go  exactly the same as in \cite{GT}. Notice that away from the singularity we have 
 no problem at all, then
what we need to prove more than  Theorem  \ref{Schauder estimate: constant coefficient}
is that we should show the Schauder estimate for the space-time dependent operator $\Delta_{a(t)}-\frac{\partial}{\partial t}$, not only for
the constant coefficient operator $\Delta-\frac{\partial}{\partial t}$ in  Theorem  \ref{Schauder estimate: constant coefficient}.
\\

 We  include the crucial detail on how to deal with the space-time dependent operator. We would like to point out that
though the intepolation inequalities which are hiding behind should be proved using some different method, these estimates here are absolutely regular. Let $h$ be small enough with respect to $g$, by  Theorem  \ref{Schauder estimate: constant coefficient} and the discussion in section 4.2 of \cite{Don},  there exists a $D$-diagonal constant coeffificient operator $\Delta_{a_0}$ such that 
\begin{eqnarray*}& &|i\partial \bar{\partial}u|^{[2]}_{\alpha,\frac{\alpha}{2},\beta,A_h\times [t,t+h^2]}
\\&\leq &C|(\Delta_{a_0}-\frac{\partial}{\partial t})u|^{[2]}_{\alpha,\frac{\alpha}{2},\beta,A_h\times [t,t+h^2]}+C|u|_{0,A_h\times [t,t+h^2]}
\\&\leq& C|(\Delta_{a(t)}-\frac{\partial}{\partial t})u|^{[2]}_{\alpha,\frac{\alpha}{2},\beta,A_h\times [t,t+h^2]}+C|(\Delta_{a(t)}-\Delta_{a_0})u|^{[2]}_{\alpha,\frac{\alpha}{2},\beta,A_h\times [t,t+h^2]}
\\& &+C|u|_{0,A_h\times [t,t+h^2]}.
\end{eqnarray*}
By further shrinking $h$ to be small enough  with respect to $a(t)$, we have  that
\begin{eqnarray*}& &C|(\Delta_{a(t)}-\Delta_{a_0})u|^{[2]}_{\alpha,\frac{\alpha}{2},\beta,A_h\times [t,t+h^2]}
\\&\leq& \frac{1}{4}[i\partial \bar{\partial}u]^{[2]}_{\alpha,\frac{\alpha}{2},\beta,A_h\times [t,t+h^2]}+Ch^{\alpha}[i\partial \bar{\partial}u]^{[2]}_{0,A_h\times [t,t+h^2]}.
\end{eqnarray*}
Then by making $h$ even smaller, we arrive at
\begin{equation}\label{Top order Schauder estimates for time invariant metric with parabolic weights}[i\partial \bar{\partial}u]^{[2]}_{\alpha,\frac{\alpha}{2},\beta,A_h\times [t,t+h^2]}\leq C|(\Delta_{a(t)}-\frac{\partial}{\partial t})u|^{[2]}_{\alpha,\frac{\alpha}{2},\beta,A_h\times [t,t+h^2]}+C|u|_{0,A_h\times [t,t+h^2]}.
\end{equation}

As the above, by making $h$ smaller if necessary, we get 
\begin{equation}\label{Top order Schauder estimates for time invariant metric with spatial  weights}|i\partial \bar{\partial}u]^{(2)}_{\alpha,\frac{\alpha}{2},\beta,A_h\times [0,h^2|}\leq C|(\Delta_{a(t)}-\frac{\partial}{\partial t})u|^{(2)}_{\alpha,\frac{\alpha}{2},\beta,A_h\times [0,h^2]}+C|u|_{0,A_h\times [0,h^2]}.\end{equation}

Combing (\ref{Top order Schauder estimates for time invariant metric with parabolic weights}), (\ref{Top order Schauder estimates for time invariant metric with spatial weights}),  Corollary \ref{Schauder estimate for the weak solution which is compactly supported in balls}, and the $C^0$ estimate in Lemma \ref{C0 estimate},  we get
$$|i\partial \bar{\partial}u|_{\alpha,\frac{\alpha}{2},\beta,{M}\times[0,T]}\leq C_{a} |v|_{\alpha,\frac{\alpha}{2},\beta,{M}\times[0,T]}.$$
The rest of the story are all of lower order, which is easier and regular.  See in \cite{GT}, \cite{Lieberman},\cite{LSU}. 

 Second, the existence of a classical solution follows from deforming the solution of the parabolic equation of  the operator $\Delta_{\omega_D}-\frac{\partial}{\partial t}$ ($\omega_D$ as in (\ref{Model conical metric defined by Donaldson})). The invertibleness of $\Delta_{\omega_D}-\frac{\partial}{\partial t}$ is guaranteed by Proposition \ref{Solvability of the operator with small oscillation in the flat cone}. The existence of the deformation is directly implied by the aprori estimates established above and a generalization of Theorem 5.2 in \cite{GT} (on existence of continuous deformation of solutions).  
\\

The proof of Theorem \ref{Schauder estimate of the linear equation: single divisor} is now complete.
\end{proof}

\section{Heat kernel and  Bessel functions of the second kind.\label{Notes on Donaldson's work}}
In this section, we use  Bessel functions of the second kind to give a shorter proof of  the estimates on integral of Bessel functions used in \cite{Don}.  The results here are slightly more general than those in \cite{Don}.  First, a lemma on Bessel function of the second kind.

\begin{lem}\label{Bessel properties: when r is small}For any $\epsilon>0$, nonnegative integer $N$, and exponents $\mu_1,\mu_2\geq 0$, such that
$\frac{N}{\beta}-\mu_1\geq 0$, there is a constant $C(\epsilon,N,\mu_1,\mu_2)$ ( as in Definition \ref{Dependence of the constant C} and  might tend to $\infty$ as $\epsilon$ goes to $0$) such that the following estimate on the weighted summation of Bessel functions holds.
$$\Sigma_{k=N}^{\infty}k^{\mu_2}I_{\frac{k}{\beta}-\mu_1}(z)\leq C(\epsilon,N,\mu_1,\mu_2)z^{\frac{N}{\beta}-\mu_1}e^{\frac{3z}{2}}e^{\epsilon z}.$$

\end{lem}
\begin{proof}{of Lemma \ref{Bessel properties: when r is small}:}\ The estimate directly follows from the property of Gamma functions. We only prove the case when $N=1,\ \mu_2=2,\ \mu_1=0$ i.e $\Sigma_{k=1}^{\infty}k^2I_{\frac{k}{\beta}}(z)\leq C(\epsilon)z^{\frac{1}{\beta}}e^{\frac{3z}{2}}e^{\epsilon z}$. The rest follows similarly.\\

Again from section (7.25) of \cite{Wa} we have 
\begin{equation}\label{7.25 of Waston}I_{\frac{k}{\beta}}(z)\leq \frac{\sqrt{\pi}z^{\frac{k}{\beta}}e^z}{\Gamma(\frac{k}{\beta}+\frac{1}{2})2^{\frac{k}{\beta}}}.
\end{equation} 
 We also have the following obvious inequality.
\begin{equation}\label{z to a power times e to the z}z^{\frac{k-1}{\beta}}e^{-(\frac{1}{2}+\epsilon)z}\leq (\frac{1}{\frac{1}{2}+\epsilon})^{(\frac{k-1}{\beta})}(\frac{k-1}{\beta})^{(\frac{k-1}{\beta})}e^{-\frac{k-1}{\beta}}.
\end{equation}
From section (12.31) of \cite{WaWhi} we get
\begin{equation}\label{12.31 of WaWhi}\log\Gamma(\frac{k}{\beta}+\frac{1}{2})\geq \frac{k}{\beta}\log(\frac{k}{\beta}+\frac{1}{2})-\frac{k}{\beta}+C.
\end{equation}

Using (\ref{7.25 of Waston}) and (\ref{z to a power times e to the z}), we compute
\begin{eqnarray}\label{01}& &\Sigma_{k=1}^{\infty}k^2I_{\frac{k}{\beta}}(z)\nonumber
\\&\leq& \Sigma_{k=1}^{\infty}\frac{\sqrt{\pi}k^2z^{\frac{k}{\beta}}e^z}{\Gamma(\frac{k}{\beta}+\frac{1}{2})2^{\frac{k}{\beta}}}
= \sqrt{\pi}z^{\frac{1}{\beta}}e^{\frac{3z}{2}}e^{\epsilon z}\Sigma_{k=1}^{\infty}\frac{k^2z^{\frac{k-1}{\beta}}e^{-(\frac{1}{2}+\epsilon)z}}{\Gamma(\frac{k}{\beta}+\frac{1}{2})2^{\frac{k}{\beta}}}\nonumber
\\&\leq&\sqrt{\pi}z^{\frac{1}{\beta}}e^{\frac{3z}{2}}e^{\epsilon z}\Sigma_{k=1}^{\infty}\frac{k^2(\frac{k-1}{\beta})^{(\frac{k-1}{\beta})}e^{-\frac{k-1}{\beta}}(\frac{2}{1+2\epsilon})^{(\frac{k-1}{\beta})}}{\Gamma(\frac{k}{\beta}+\frac{1}{2})2^{\frac{k}{\beta}}}\nonumber
\\&\leq& Cz^{\frac{1}{\beta}}e^{\frac{3z}{2}}e^{\epsilon z}\Sigma_{k=1}^{\infty}\frac{(\frac{k}{\beta})^{(\frac{k}{\beta}+1)}e^{-\frac{k}{\beta}}}{\Gamma(\frac{k}{\beta}+\frac{1}{2})(1+2\epsilon)^{\frac{k}{\beta}}}
.\end{eqnarray}
To estimate the last series above, we take logarithm of the general terms and apply (\ref{12.31 of WaWhi}) to get
\begin{eqnarray}\label{02}& &\log[\frac{(\frac{k}{\beta})^{(\frac{k}{\beta}+1)}e^{-\frac{k}{\beta}}}{\Gamma(\frac{k}{\beta}+\frac{1}{2})(1+2\epsilon)^{\frac{k}{\beta}}}]\nonumber
\\&\leq &(\frac{k}{\beta}+1)\log(\frac{k}{\beta})-\frac{k}{\beta}-\frac{k}{\beta}\log(\frac{k}{\beta}+\frac{1}{2})+\frac{k}{\beta}+C-\frac{k}{\beta}\log (1+2\epsilon)\nonumber
\\&\leq &\log(\frac{k}{\beta})+C-\frac{k}{\beta}\log(1+2\epsilon)
\leq C(\epsilon)-\frac{k}{2\beta}\log (1+2\epsilon).
\end{eqnarray}
Combining (\ref{01}) and (\ref{02}), we have
\begin{eqnarray*}& &\Sigma_{k=1}^{\infty}k^2I_{\frac{k}{\beta}}(z)
\\&\leq&Cz^{\frac{1}{\beta}}e^{\frac{3z}{2}}e^{\epsilon z}\Sigma_{k=1}^{\infty}\frac{(\frac{k}{\beta})^{(\frac{k}{\beta}+1)}e^{-\frac{k}{\beta}}}{\Gamma(\frac{k}{\beta}+\frac{1}{2})(1+2\epsilon)^{\frac{k}{\beta}}}
\leq C(\epsilon)z^{\frac{1}{\beta}}e^{\frac{3z}{2}}e^{\epsilon z}\Sigma_{k=1}^{\infty}(1+2\epsilon)^{-\frac{k}{2\beta}}
\\&\leq& C(\epsilon)z^{\frac{1}{\beta}}e^{\frac{3z}{2}}e^{\epsilon z}.
\end{eqnarray*}

\end{proof}

Now, similar to  \cite{Don}, for any  $v$ and $\mu\geq 0$,  we define  $G_{\mu,v}(r,r^{'},R)$ as $$G_{\mu,v,h}(r,r^{'},R)=\int_{0}^{\infty}\int^{\infty}_{0}(2\pi t)^{1-v }(\frac{rr^{'}}{2t})^{h}e^{-\lambda^2 t-\frac{R^2}{4t}}J_{\mu}(\lambda r)J_{\mu}(\lambda r^{'})\lambda d\lambda dt.$$

Using formula (\ref{Weber's formula}), which is also referable  in \cite{Wa}, we have another expression
\begin{equation}\label{G's formula by Weber's formual}G_{\mu,v,h}(r,r^{'},R)=\pi\int^{\infty}_{0}(2\pi t)^{-v}(\frac{rr^{'}}{2t})^{h}e^{-\frac{r^2+{r^{'}}^2+R^2}{4t}}I_{\mu}(\frac{rr^{'}}{2t}) dt.
\end{equation}

\begin{lem}\label{Estimating the sum of G w.r.t k} Let $\mu_2>0, \;v>0.\;$ The following holds:
\begin{itemize}
\item Suppose $rr^{'}<\frac{9R^2}{10},\ R>0,\ N\geq \beta \mu_1, h\geq \mu_1-\frac{N}{\beta}\;$  (h could be negative). Then
$$\Sigma_{k\geq \beta\mu_1}k^{\mu_2}|G_{\frac{k}{\beta}-\mu_1,v,h}(r,r^{'},R)|\leq C(v,\mu_1,\mu_2,h)R^{2-2v}.$$
\item Suppose $r< \frac{4r^{'}}{11},\ r^{'}>0.\;$ Then
$$\Sigma_{k\geq \beta\mu_1}k^{\mu_2}|G_{\frac{k}{\beta}-\mu_1,v,h}(r,r^{'},R)|\leq C(v,\mu_1,\mu_2,h){(r^{'})}^{2-2v}.$$
\end{itemize}

\end{lem}
\begin{proof}{of Lemma \ref{Estimating the sum of G w.r.t k}:}\ These estimates are straight forward consequences of Lemma \ref{Bessel properties: when r is small}. By Lemma \ref{Bessel properties: when r is small} and the assumption that 
$h-\mu_1+\frac{N}{\beta}\geq 0$, we have
\begin{eqnarray}\label{03}& &\Sigma_{k\geq N}k^{\mu_2}(\frac{rr^{'}}{2t})^{h}|G_{\frac{k}{\beta}-\mu_1,v,h}(r,r^{'},R)|\nonumber
\\&= &C\int_{0}^{+\infty}t^{-v}e^{-\frac{r^2+{r^{'}}^2+R^2}{4t}}(\frac{rr^{'}}{2t})^{h-\mu_1+\frac{N}{\beta}}\Sigma_{k\geq N}^{\infty}k^{\mu_2}I_{\frac{k}{\beta}-\mu_1}(\frac{r\acute{r}}{2t})dt\nonumber
\\&\leq &C(v,\mu_1,\mu_2,h,\epsilon)\int_{0}^{+\infty}t^{-v}e^{-\frac{r^2+{r^{'}}^2+R^2}{4t}}e^{\frac{3r\acute{r}}{4t}}e^{\frac{\epsilon r\acute{r}}{4t}}dt.
\end{eqnarray}
We consider the term $e^{-\frac{r^2+{r^{'}}^2+R^2}{4t}}e^{\frac{3r\acute{r}}{4t}}e^{\frac{\epsilon r\acute{r}}{4t}}$ and try to bound it from above in both cases.\\

When $R>0$, and $rr^{'}<\frac{9R^2}{10}$, without loss of generality we can assume $R=1$. Let $\epsilon=10^{-100}$ and using $rr^{'}<\frac{9}{10}$, we compute
\begin{eqnarray*}& &e^{-\frac{r^2+{r^{'}}^2+R^2}{4t}}e^{\frac{3r\acute{r}}{4t}}e^{\frac{\epsilon r\acute{r}}{4t}}
\\&= &e^{-\frac{(r-r^{'})^2}{4t}}e^{-\frac{1-(1+\epsilon) r\acute{r}}{4t}}
\leq e^{-\frac{1}{400t}}.
\end{eqnarray*}
Then in this case we have
\begin{eqnarray*}& &\Sigma_{k\geq N}k^{\mu_2}|G_{\frac{k}{\beta}-\mu_1,v,h}(r,r^{'},R)|
\\&\leq &C(v,\mu_1,\mu_2,h)\int_{0}^{+\infty}t^{-v}e^{-\frac{1}{400t}}dt
\\&\leq &C(v,\mu_1,\mu_2,h).
\end{eqnarray*}
The statement in the first case for all $R>0$ is thus obtained  by rescaling.\\ 

In the case when $\acute{r}>0$ and $r< \frac{4r^{'}}{11}$, we can also assume $\acute{r}=1$. Then working from (\ref{03}), using 
$r<\frac{4}{11}$ and $\epsilon=10^{-100}$, we have
\begin{eqnarray*}& &e^{-\frac{r^2+{r^{'}}^2+R^2}{4t}}e^{\frac{3r\acute{r}}{4t}}e^{\frac{\epsilon r\acute{r}}{4t}}
\\&\leq  &e^{-\frac{r^2-(3+\epsilon)r+1}{4t}}
\leq e^{-\frac{1}{4\times 10^{10000}t}},
\end{eqnarray*}
thus in the second case we simply compute from (\ref{03}) that 
\begin{eqnarray*}& &\Sigma_{k\geq N}k^{\mu_2}|G_{\frac{k}{\beta}-\mu_1,v,h}(r,r^{'},R)|
\\&\leq &C(v,\mu_1,\mu_2,h)\int_{0}^{+\infty}t^{-v}e^{-\frac{1}{4\times 10^{10000}t}}dt
\\&\leq &C(v,\mu_1,\mu_2,h).
\end{eqnarray*}

The proofs in  both cases in Lemma \ref{Estimating the sum of G w.r.t k}   are completed now.
\end{proof}



\section{Local estimates for the heat kernel.\label{Local estimates for the heat kernel}}
   In this section we prove Lemma \ref{integral of time derivative of hk w.r.t time},\ref{integral of gradient of hk w.r.t time},L\ref{integral of gradient of time derivative of the hk}, \ref{integral of third derivative of h.k w.r.t time: when r big}, \ref{time integral of the bilinear derivatives of the heat kernel}. These lemmas are part of Theorem \ref{all the properties of the h.k1}. 

\begin{lem}\label{integral of time derivative of hk w.r.t time} For any $x$ and $y$ and $\mathfrak{D}\in \mathfrak{T}$, we have
$$\left\{ \begin{array}{lcl} \int_{0}^{\infty}|\frac{\partial}{\partial \tau}H(x,y,\tau)|d\tau & \leq & C|x-y|^{-(m+2)},\\
\int_{0}^{\infty}|\mathfrak{D}H(x,y,\tau)|d\tau &\leq & C|x-y|^{-(m+2)}.\end{array}\right. $$
In particular, we have
 $$
 \left\{\begin{array}{lcl} \int_{0}^{\infty}|\frac{\partial}{\partial \tau}H(0,y,\tau)|d\tau & \leq & C|y|^{-(m+2)},\\
 \int_{0}^{\infty}|\mathfrak{D}H(0,y,\tau)|d\tau & \leq & C|y|^{-(m+2)}. \end{array}\right. $$
\end{lem}
\begin{proof}{of Lemma \ref{integral of time derivative of hk w.r.t time}:}
\ We only  prove the estimates for the time derivative, the estimate for the spatial derivatives are the same. The estimate  $\int_{0}^{\infty}|\frac{\partial}{\partial t}H(0,y,\tau)|d\tau\leq C|y|^{-(m+2)}$ has a much shorter proof than the general estimate. To be precise, it suffices to assume $|y|=1$, assume $r^{'}=0$ and $\widehat{y}=0 $. Then we have
\begin{eqnarray*}& &H(0,y,t-\tau)
\\&=&\frac{1}{\beta(4\pi t)^{\frac{m}{2}+1}}e^{-\frac{r^2+{r^{'}}^2+R^2}{4t}}[I_{0}(\frac{rr^{'}}{2t})+\Sigma_{k\geq 1} 2cosk(\theta-\theta^{'})I_{\frac{k}{\beta}}(\frac{rr^{'}}{2t})]
\\&=&\frac{1}{\beta(4\pi)^{\frac{m}{2}+1}}\frac{1}{t^{\frac{m}{2}+1}}e^{-\frac{1}{4t}}.
\end{eqnarray*}
 Thus
\begin{eqnarray*}& &\int_{0}^{\infty}|\frac{\partial}{\partial t}H(0,y,t-\tau)|d\tau
\\&=&\frac{1}{\beta(4\pi)^{\frac{m}{2}+1}}\int_{0}^{\infty}|\frac{1}{4t^{\frac{m}{2}+3}}e^{-\frac{1}{4t}}-(\frac{m}{2}+1)\frac{1}{t^{\frac{m}{2}+2}}e^{-\frac{1}{4t}}|d\tau
\\&\leq&C(m,\beta).
\end{eqnarray*}

The general estimate for all $x$ and $y$ is involves more properties of the Bessel functions. We can  assume $|x-y|=1$ and   the  theorem in general follows from rescaling.
Denote $v=\beta(\theta-\theta^{'})$. We still assume that  $ v\neq \pi$ or $-\pi$ mod $2\beta\pi$ so the last term in (\ref{P's Expression formula}) vainishes. The conclusion for all $ v$ follows from continuity of the estimates with respect to $v$.
\\

We first compute from (\ref{Higher dim heat kernel formula}) that 
\begin{eqnarray*}& &\frac{\partial}{\partial t}H(x,y,t)
\\&=&\frac{C}{t^{\frac{m}{2}+1}}[\frac{R^2+(r^{'}-r)^2+2r\acute{r}(1-\cos v)}{4t^2}-(\frac{m}{2}+1)\frac{1}{t}]e^{-\frac{R^2+(r^{'}-r)^2+2r\acute{r}(1-\cos v)}{4t}}
\\& &+\frac{C}{t^{\frac{m}{2}+1}}\Sigma_{k\neq 0,\ -\pi <v+2k\beta \pi<\pi}[\frac{R^2+(r^{'}-r)^2+2r\acute{r}(1-\cos (v+2k\beta\pi))}{4t^2}
\\& &-(\frac{m}{2}+1)\frac{1}{t}]
\ast e^{-\frac{R^2+(r^{'}-r)^2+2r\acute{r}(1-\cos( v+2k\beta\pi))}{4t}}
\\& &+\frac{C}{t^{\frac{m}{2}+1}}[\frac{R^2+(r^{'})^2+r^2}{4t^2}-(\frac{m}{2}+1)\frac{1}{t}]E(\frac{r\acute{r}}{2t},v)\ast e^{-\frac{R^2+(r^{'})^2+r^2}{4t}}
\\& &+\frac{C}{t^{\frac{m}{2}+1}}\frac{r\acute{r}}{2t^2}\frac{\partial E}{\partial z}(\frac{r\acute{r}}{2t},v)e^{-\frac{R^2+(r^{'})^2+r^2}{4t}}.
\end{eqnarray*}
It is easy to see from $|x-y|=1$ and (\ref{Def of  a}) we have the following two estimates:
 $$R^2+(r^{'}-r)^2+2r\acute{r}\geq \frac{1}{2};$$
 \begin{eqnarray*}& &R^2+(r^{'}-r)^2+2r\acute{r}(1-\cos( v+2k\beta\pi))
\\&\geq& R^2+(r^{'}-r)^2+2r\acute{r}8a
\\&\geq& R^2+(r^{'}-r)^2+2r\acute{r}(1-\cos v)4a
\\&\geq& a.
\end{eqnarray*}
Then by applying  Theorem \ref{Decay estimates for E} we easily get 
\begin{equation}\frac{\partial}{\partial t}H(x,y,t)|\leq C(1+\frac{1}{t^{\frac{m}{2}+5}})e^{-\frac{a}{1000t}}.
\end{equation} Hence
\begin{eqnarray*}& &\int_{0}^{\infty}|\frac{\partial}{\partial t}H(x,y,t)|dt
 \\&\leq &C\int_{0}^{\infty}(1+\frac{1}{t^{\frac{m}{2}+5}})e^{-\frac{a}{1000t}}dt
 \\&\leq &C.
 \end{eqnarray*}
\end{proof}

Using similar methods as in the proof of Lemma \ref{integral of time derivative of hk w.r.t time}, we deduce  Lemma \ref{integral of gradient of hk w.r.t time} as follows. 
\begin{lem}\label{integral of gradient of hk w.r.t time}
For any $x$ and $y$, we have
  $$\int_{0}^{\infty}|\nabla_y H(x,y,\tau)|d\tau\leq C|x-y|^{-(m+1)}.$$
\end{lem}

Using similar techniques,  we also have control on the following integral of third order derivative of the heat kernel. 
\begin{lem}\label{integral of gradient of time derivative of the hk}For every $x$ and $y$ we have  $$\int_{0}^{\infty}|\nabla_x \frac{\partial}{\partial t}H(x,y,\tau)|d\tau\leq C|y-x|^{-(m+3)}.$$
\end{lem}
\begin{proof}{of Lemma \ref{integral of gradient of time derivative of the hk}:}
This lemma is a directly corollary of Theorem \ref{Decay estimates for E}. However, we still think it's necessary to  give a clear proof here since the lemma  so important.
It suffices to estimate $\int_{0}^{\infty}| \frac{\partial^2}{\partial t\partial r}H(x,y,\tau)|d\tau$, the other derivatives are settled down in the same way. By scaling invariance, we can assume $|x-y|=1$.
Denote $v=\beta(\theta-\theta^{'})$. We still assume that  $ v\neq \pi$ or $-\pi$ mod $2\beta\pi$ so the last term in (\ref{P's Expression formula}) vainishes. The conclusion for all $ v$ follows from continuity of the estimates with respect to $v$.
 The formula says

\begin{eqnarray*}\label{Time-radius derivative of the 2-dim heat kernel}& &\frac{\partial^2H}{\partial t\partial r}\nonumber
\\&=&\frac{1}{(4\pi )^{\frac{m}{2}+1}}e^{-\frac{({r-{\acute{r}}})^2+2r\acute{r}(1-\cos v)+R^2}{4t}}\{\frac{1}{t^{\frac{m}{2}+1}}\frac{({r-{\acute{r}}})+\acute{r}(1-\cos v)}{2 t^2}
\\& &-\frac{1}{t^{\frac{m}{2}+1}}[\frac{({r-{\acute{r}}})^2+2r\acute{r}(1-\cos v)+R^2}{4t^2}-(\frac{m}{2}+1)\frac{1}{t}]\frac{({r-{\acute{r}}})+\acute{r}(1-\cos v)}{2 t}\}
\\& &+\displaystyle \frac{1}{(4\pi )^{\frac{m}{2}+1}}\Sigma_{k\neq 0,|v+2k\beta \pi|<\pi}Ce^{-\frac{({r-{\acute{r}}})^2+2r\acute{r}(1-\cos (v+2k\beta\pi))+R^2}{4t}}\{
\\& &\frac{1}{t^{\frac{m}{2}+1}}\frac{({r-{\acute{r}}})+\acute{r}(1-\cos (v+2k\beta\pi))}{2 t^2}
\end{eqnarray*}
\begin{eqnarray*}
& &-\frac{1}{t^{\frac{m}{2}+1}}[\frac{({r-{\acute{r}}})^2+2r\acute{r}(1-\cos(v+2k\beta\pi))+R^2}{4t^2}-(\frac{m}{2}+1)\frac{1}{t}]
\\& &\times\frac{({r-{\acute{r}}})+\acute{r}(1-\cos (v+2k\beta\pi))}{2 t}\}
\\& &+\frac{1}{(4\pi )^{\frac{m}{2}+1}}\frac{1}{2\pi\beta}e^{-\frac{r^2+{{\acute{r}}}^2+R^2}{4t}}\{[(\frac{m}{2}+2)\frac{r}{2 t^{\frac{m}{2}+3}}
-\frac{r}{2 t^{\frac{m}{2}+2}}\frac{r^2+{{\acute{r}}}^2+R^2}{4t^2}]E(\frac{r\acute{r}}{2t},v)
\\& &+[-(\frac{m}{2}+2)\frac{\acute{r}}{2t^{\frac{m}{2}+3}}+\frac{1}{t^{\frac{m}{2}+1}}\frac{r^2+{{\acute{r}}}^2+R^2}{4t^2}\ast \frac{\acute{r}}{2t}
+\frac{1}{t^{\frac{m}{2}+1}}\ast \frac{r\acute{r}}{2t^2}\ast \frac{\acute{r}}{2t}]\frac{\partial E(\frac{r\acute{r}}{2t},v)}{\partial z}
\\& &-\frac{(r\acute{r})\acute{r}}{4 t^{\frac{m}{2}+4}}\frac{\partial^2 E(\frac{r\acute{r}}{2t},v)}{\partial z^2}\}
.\nonumber
  \end{eqnarray*}
Though there are lots of terms, but all the them are actually estimated by a single way.  Pick
$$\frac{\acute{r}(1-\cos{v})}{ t^{\frac{m}{2}+3}}e^{-\frac{({r-{\acute{r}}})^2+R^2}{4t}} e^{-\frac{r\acute{r}}{2t}(1-\cos{v})}$$
as an example. First we make the following claim. 
\begin{clm}$\acute{r}(1-\cos{v})\leq 4$.\end{clm} This is easy to prove.\\
Case 1: Suppose $\acute{r}\leq 2$, it's obvious;\\
Case 2: Suppose $\acute{r}\geq 2$, since $|x-y|=1$ we have $r\geq \acute{r}-1\geq 1$. Then
$$\acute{r}(1-\cos{v})\leq r\acute{r}(1-\cos{v})\leq |x-y|^2=1.$$
The claim is now proved.\\

 We go on to estimate that 
$$\frac{\acute{r}(1-\cos{v})}{t^{\frac{m}{2}+3}}e^{-\frac{({r-{\acute{r}}})^2+R^2}{4t}} e^{-\frac{r\acute{r}}{2t}(1-\cos{v})}\leq \frac{4}{t^{\frac{m}{2}+3}}e^{-\frac{1}{4t}}.$$
Therefore estimating each term like the above,  combining (\ref{Def of  a}), doing the same thing as in Lemma \ref{Holder estimate of the integral of spatial derivative of h.k with respect to time},
 we get that
 $$|\frac{\partial^2H}{\partial t\partial r}|\leq C(1+\frac{1}{t^{\frac{m}{2}+5}})e^{-\frac{a}{1000t}},$$
where $a=\frac{\sin^2{\frac{\beta \pi}{2}}}{4}$. Then
\begin{eqnarray*}& &\int_{0}^{\infty}|\frac{\partial^2H}{\partial t\partial r}(x,y,t)|dt
\leq C.
 \end{eqnarray*}
\end{proof}

It's easy to deduce the following estimate when $x$ is far from the singularity (comparing to $|x-y|$).
\begin{lem}\label{integral of third derivative of h.k w.r.t time: when r big}Suppose $P(x)\geq \frac{|x-y|}{100^{100}}$, then for any second order operator
$\mathfrak{D}\in\ \mathfrak{T}$, we have that
$$\int_{0}^{\infty}|\nabla_x\mathfrak{D}H(x,y,\tau)|d\tau\leq C|y-x|^{-(m+3)}.$$
\end{lem}
\begin{proof}{of Lemma \ref{integral of third derivative of h.k w.r.t time: when r big}:} It suffices to assume $|x-y|=1$ and thus from the hypothesis we
have that $P(x)\geq \frac{1}{100^{100}}$, which means the $r$ is big. Thus adapting the same proof of Lemma \ref{integral of gradient of time derivative of the hk}
,  use Theorem \ref{Decay estimates for E} and $r\geq \frac{1}{100^{100}}$, the proof is completed. In a word, the case when $r$ big is trivial (the real trouble is when $r$ is small).
\end{proof}
\begin{lem}\label{time integral of the bilinear derivatives of the heat kernel} For any $x$ and $y$, we have
$$\int_{0}^{\infty}|\nabla_x\frac{\partial}{\partial \acute{s_i}}H(x,y,t)|dt\leq C|x-y|^{m+2}$$
and
$$\int_{0}^{\infty}|\nabla_x\frac{\partial}{\partial \acute{r}}H(x,y,t)|dt\leq C|x-y|^{m+2}.$$
Notice that the operators $\frac{\partial}{\partial \acute{s_i}},\ 1\leq i\leq m$ and $\frac{\partial}{\partial \acute{r}}$ are all with respect to the $y$ variable, while $\nabla_x$ is with respect to the $x$ variable.
\end{lem}
\begin{proof}{of Lemma \ref{time integral of the bilinear derivatives of the heat kernel}:} This is again a straight forward corollary of Theorem \ref{Decay estimates for E}, just notice that neither $\nabla_x\frac{\partial}{\partial s_i}$ nor $\nabla_x\frac{\partial}{\partial \acute{r}}$ concerns twice or more derivatives on the $v$ (argument) of the $E(z,v)$ in Theorem \ref{Decay estimates for E}.\\

 To be precise, we again assume $|x-y|=1$ and   $ v\neq \pi$ or $-\pi$ mod $2\beta\pi$ ( so the last term in (\ref{P's Expression formula}) vainishes). We only show the estimate for $\int_{0}^{\infty}|\frac{\partial^2}{\partial r\partial \acute{r}}H(x,y,t)|dt$, the other terms  are similar. We compute
\begin{eqnarray*}& & \frac{\partial^2}{\partial r\partial \acute{r}}H(x,y,t)
\\&=&\frac{1}{(4\pi )^{\frac{m}{2}+1}}[\frac{\cos v}{2t^{\frac{m}{2}+2}}+\frac{[(r-\acute{r})+\acute{r}(1-\cos v)][(\acute{r}-r)+r(1-\cos v)]}{4t^{\frac{m}{2}+3}}]
\\&\times &e^{-\frac{(r-\acute{r})^2+2r\acute{r}(1-\cos v)+R^2}{4t}}
+\frac{1}{(4\pi )^{\frac{m}{2}+1}}\Sigma_{k\neq 0,| v+2k\beta \pi|<\pi}[\frac{\cos ( v+2k\beta \pi)}{2t^{\frac{m}{2}+2}}
\\& &+\frac{[(r-\acute{r})+\acute{r}(1-\cos( v+2k\beta \pi))][(\acute{r}-r)+r(1-\cos ( v+2k\beta \pi)))]}{4t^{\frac{m}{2}+3}}]
\\&\times &e^{-\frac{(r-\acute{r})^2+2r\acute{r}(1-\cos (v+2k\beta \pi))+R^2}{4t}}
\\& &+\frac{1}{(4\pi )^{\frac{m}{2}+1}}\frac{1}{2\pi\beta}\frac{1}{t^{\frac{m}{2}+1}}[(\frac{r\acute{r}}{4t^2})E(\frac{r\acute{r}}{2t}, v)-(\frac{r^2}{4t^2})\frac{\partial}{\partial z}E(\frac{r\acute{r}}{2t}, v)\\& &-(\frac{\acute{r}^2}{4t^2})\frac{\partial}{\partial z}E(\frac{r\acute{r}}{2t}, v)
+(\frac{1}{2t})\frac{\partial}{\partial z}E(\frac{r\acute{r}}{2t}, v)+ (\frac{r\acute{r}}{4t^2})\frac{\partial^2}{\partial z^2}E(\frac{r\acute{r}}{2t}, v)]e^{-\frac{r^2+{{\acute{r}}}^2+R^2}{4t}}.
\end{eqnarray*}
The hypothesis $|x-y|=1$ implies that
\begin{equation}\label{05}
|(r-\acute{r})+\acute{r}(1-\cos v)|\leq 10,\ |(\acute{r}-r)+r(1-\cos v)|\leq 10;
\end{equation}
\begin{eqnarray}\label{06}& &[(r-\acute{r})+\acute{r}(1-\cos( v+2k\beta \pi))][(\acute{r}-r)+r(1-\cos ( v+2k\beta \pi)))]\nonumber
\\&\leq & (r-\acute{r})^2+(r-\acute{r})(r+\acute{r})(1-\cos( v+2k\beta \pi))\nonumber
\\& &+r\acute{r}(1-\cos (v+2k\beta \pi))^2\nonumber
\\&\leq &20+(r-\acute{r})^2
+2r\acute{r}(1-\cos (v+2k\beta \pi)).
\end{eqnarray}

Thus, by the idea  used in  the estimate of $\frac{1}{r}\frac{\partial\widehat{H}}{\partial  v}$ in Lemma \ref{asymptotic estimate of theta derivative of Heat kernel}, estimate (\ref{06}), and $|x-y|=1$,  it's easy to see that
\begin{eqnarray*}& &[(r-\acute{r})+\acute{r}(1-\cos( v+2k\beta \pi))][(\acute{r}-r)+r(1-\cos ( v+2k\beta \pi)))]
\\& & \times e^{-\frac{(r-\acute{r})^2+2r\acute{r}(1-\cos (v+2k\beta \pi))+R^2}{4t}}
\\&\leq &[20+(r-\acute{r})^2
+2r\acute{r}(1-\cos (v+2k\beta \pi))]
\times e^{-\frac{(r-\acute{r})^2+2r\acute{r}(1-\cos (v+2k\beta \pi))+R^2}{4t}}
\\&\leq& C(1+t)e^{-\frac{a}{1000t}},
\end{eqnarray*}
where $a$ is the one in  (\ref{Def of  a}). Applying Theorem \ref{Decay estimates for E}, it's easy to see that
 $$|\frac{\partial^2}{\partial r\partial \acute{r}}H(x,y,t)|\leq C(1+\frac{1}{t^{\frac{m}{2}+5}})e^{-\frac{a}{1000t}}.$$

 Then we easily get
 \begin{eqnarray*}& &\int_{0}^{\infty}|\frac{\partial^2}{\partial r\partial \acute{r}}H(x,y,t)|dt
 \\&\leq &C\int_{0}^{\infty}(1+\frac{1}{t^{\frac{m}{2}+5}})e^{-\frac{a}{1000t}}dt
 \\&\leq &C.
 \end{eqnarray*}
 By rescaling, the proof in the general case is also completed.
\end{proof}
\section{Behaviors of the heat kernel  near the singular set.\label{Behaviors of the heat kernel  near the singular set}}
In this section we  prove Lemma \ref{Holder estimate of the integral of spatial derivative of h.k with respect to time}. This lemma contains  the last two items in Theorem \ref{all the properties of the h.k1}, and  is the only lemma which concerns the exponent $\frac{1}{\beta}-1$.\\

First of all,      the Bessel functions of the second kind  satisfies the following recursive relation (see \cite{Wa}).
\begin{equation}\label{recursion of I_v}I_v^{'}(z)+\frac{vI_v(z)}{z}=I_{v-1}(z).\end{equation}

Two useful corollaries of the above recursive relation which  we will keep applying are
\begin{equation} \left\{\begin{array}{lcll} I_v^{'}(z) & \leq & I_{v-1}(z)  &\ \textrm{when}\ v\geq 1 \ \textrm{and}\ z\geq 0;\\
 \frac{vI_v(z)}{z} & \leq & I_{v-1}(z) & \ \textrm{when}\ v\geq 1 \ \textrm{and}\ z\geq 0.
\end{array} \right.  \label{bounding derivative of Iv by Iv-1}
\end{equation}

Notice that when $v\geq 1$, $\frac{I_v(z)}{z} $ is well defined at $z=0$.\\

The following lemma  is useful in proving Lemma \ref{Holder estimate of the integral of spatial derivative of h.k with respect to time}.
\begin{lem}\label{estimating the holder continuity of the derivative of I}Suppose $v\geq 1$ and  $z_2\geq z_1$, we have
$$|I^{'}_v(z_2)-I^{'}_v(z_1)|\leq \frac{v}{2^{v}\Gamma(v+1)}|z_2^{v-1}-z_1^{v-1}|+|z_2-z_1|(1+\frac{v}{2})I_v(z_2).$$
In particular, if $v<2$ we have
$$|I^{'}_v(z_2)-I^{'}_v(z_1)|\leq \frac{v}{2^{v}\Gamma(v+1)}|z_2-z_1|^{v-1}+|z_2-z_1|(1+\frac{v}{2})I_v(z_2).$$
\end{lem}
\begin{proof}{of Lemma \ref{estimating the holder continuity of the derivative of I}:}\  Suppose $z_2\geq z_1$, from the series representation formula in section 3.7 of \cite{Wa},  we compute
\begin{eqnarray*}& &|I^{'}_v(z_2)-I^{'}_v(z_1)|
\\&=&\Sigma^{\infty}_{j= 0}\frac{(v+2j)z_2^{v+2j-1}}{2^{v+2j}j!(v+j)!}-\Sigma^{\infty}_{j= 0}\frac{(v+2j)z_1^{v+2j-1}}{2^{v+2j}j!(v+j)!}
\\&=&\frac{v}{2^{v}\Gamma(v+1)}(z_2^{v-1}-z_1^{v-1})+\Sigma^{\infty}_{j= 1}\frac{(v+2j)z_2^{v+2j-1}}{2^{v+2j}j!(v+j)!}-\Sigma^{\infty}_{j= 1}\frac{(v+2j)z_1^{v+2j-1}}{2^{v+2j}j!(v+j)!}
\\&\leq&\frac{v}{2^{v}\Gamma(v+1)}(z_2^{v-1}-z_1^{v-1})+(z_2-z_1)\Sigma^{\infty}_{j= 1}\frac{(v+2j)(v+2j-1)z_2^{v+2j-2}}{2^{v+2j}j!(v+j)!}.
\end{eqnarray*}
Using the simple inequality 
$$\frac{(v+2j)(v+2j-1)}{j(v+j)}\leq 4+2v,\ \textrm{when}\ j\geq 1 \textrm{and}\ 
j\geq 1,$$ we have
\begin{eqnarray*}& &\Sigma^{\infty}_{j= 1}\frac{(v+2j)(v+2j-1)z_2^{v+2j-2}}{2^{v+2j}j!(v+j)!}
\\&\leq& (4+2v)\Sigma^{\infty}_{j= 1}\frac{z_2^{v+2j-2}}{2^{v+2j}(j-1)!(v+j-1)!}=(1+\frac{v}{2})I_v(z_2)
\end{eqnarray*}
Then the proof of  the first part is set.\\

 To get the second part it suffices to use  $z_2^{v-1}-z_1^{v-1}\leq (z_2-z_1)^{v-1}$ provided $1\leq v<2$.
\end{proof}
\begin{lem}\label{Holder estimate of the integral of spatial derivative of h.k with respect to time}For every second order spatial derivative operator $\mathfrak{D}\in \mathfrak{T}$, the following holds. Suppose $\rho= \min(\frac{1}{\beta}-1,\ 1)$, $|y|=1$ and $|u_1|,\ |u_2|< \frac{1}{8}$, we have
$$\int_{0}^{\infty}|\mathfrak{D}H(u_1,y,\tau)-\mathfrak{D}H(u_2,y,\tau)|d\tau\leq C|u_1-u_2|^{\rho}$$ and
$$\int_{0}^{\infty}|\nabla_{x}H(u_1,y,\tau)-\nabla_{x}H(u_2,y,\tau)|d\tau\leq C|u_1-u_2|^{\rho}.$$
\end{lem}
\begin{proof}{of Lemma \ref{Holder estimate of the integral of spatial derivative of h.k with respect to time}:} Let  $u_1=(r_1,\theta_1,\widehat{u_1})$,
$u_2=(r_2,\theta_2,\widehat{u_2})$, $y=(r^{'},\theta^{'}, \widehat{y})$ and   $R_1=|\widehat{y}-\widehat{u_1}|,\ R_2=|\widehat{y}-\widehat{u_2}| $.
 We only prove the first inequality in Lemma \ref{Holder estimate of the integral of spatial derivative of h.k with respect to time}, which is a second order estimate. The second inequality is a first order estimate, which is much easier than second order estimates. To prove the first inequality, it suffices to consider the operator $\frac{\partial^2}{\partial r\partial s_i}$, the other operators are similar. \\

From now on we assume $\frac{1}{\beta}-1< 1$, the case when $\frac{1}{\beta}-1\geq 1$ is much easier to handle with the following method and we end up with a Lipschitz estimate i.e, with the bound as $C|u_1-u_2|$. \\

 We split the integral as the following. For all  $k$, write 
 \[H_k(r,\theta,{\acute{r}},\theta^{'},R,t)=\frac{1}{ \beta}\frac{1}{(4\pi t)^{\frac{m}{2}+1}}e^{-\frac{r^2+{{\acute{r}}}^2+R^2}{4t}}I_{\frac{k}{\beta}}(\frac{r\acute{r}}{2t})\cos k(\theta-\acute{\theta}).\] We compute 
\begin{eqnarray*}& &\int_0^{\infty}\frac{\partial^2}{\partial r\partial s_i}H_k(u_1,y,t)-\frac{\partial^2}{\partial r\partial s_i}H_k(u_2,y,t)dt
\\&=&\int_0^{\infty}\frac{C}{t^{\frac{m}{2}+3}}[r_1(s_{i,1}-\acute{s_i})e^{-\frac{r_1^2+{r^{'}}^2+R^2_1}{4t}}I_{\frac{k}{\beta}}(\frac{r_1r^{'}}{2t})cosk(\theta_1-\theta^{'})
\\&-&r_2(s_{i,2}-\acute{s_i})e^{-\frac{r_2^2+{r^{'}}^2+R^2_2}{4t}}I_{\frac{k}{\beta}}(\frac{r_2r^{'}}{2t})cosk(\theta_2-\theta^{'})]dt
\\& &+\int_0^{\infty}\frac{C\acute{r}}{t^{\frac{m}{2}+3}}[(s_{i,1}-\acute{s_i})e^{-\frac{r_1^2+{r^{'}}^2+R^2_1}{4t}}I^{'}_{\frac{k}{\beta}}(\frac{r_1r^{'}}{2t})cosk(\theta_1-\theta^{'})
\\& &-(s_{i,2}-\acute{s_i})e^{-\frac{r_2^2+{r^{'}}^2+R^2_2}{4t}}I^{'}_{\frac{k}{\beta}}(\frac{r_2r^{'}}{2t})cosk(\theta_2-\theta^{'})]dt.
\end{eqnarray*}
Thus \begin{equation}\label{CH11} \int_0^{\infty}|\frac{\partial^2}{\partial r\partial s_i}H(u_1,y,t)-\frac{\partial^2}{\partial r\partial s_i}H(u_2,y,t)|dt\leq C(|II_1|+|II_2|),\end{equation}
where (let $\vartheta_{k}=2$ when $k\geq 2$ and $\vartheta_{1}=1$)
           \begin{eqnarray*}& &II_1
               \\&=&\Sigma_{k=0}^{\infty}\vartheta_{k}[\int_{0}^{\infty}|\frac{1}{t^{\frac{m}{2}+3}}[r_1(s_{i,1}-\acute{s_i})e^{-\frac{r_1^2+{r^{'}}^2+R^2_1}{4t}}I_{\frac{k}{\beta}}(\frac{r_1r^{'}}{2t})cosk(\theta_1-\theta^{'})
\\&-&r_2(s_{i,2}-\acute{s_i})e^{-\frac{r_2^2+{r^{'}}^2+R^2_2}{4t}}I_{\frac{k}{\beta}}(\frac{r_2r^{'}}{2t})cosk(\theta_2-\theta^{'})]|dt],
\end{eqnarray*}
and
\begin{eqnarray*}& &II_2
               \\&=&\Sigma_{k=0}^{\infty}\int_{0}^{\infty}|\frac{\acute{r}}{t^{\frac{m}{2}+3}}[(s_{i,1}-\acute{s_i})e^{-\frac{r_1^2+{r^{'}}^2+R^2_1}{4t}}I^{'}_{\frac{k}{\beta}}(\frac{r_1r^{'}}{2t})cosk(\theta_1-\theta^{'})
\\& &-(s_{i,2}-\acute{s_i})e^{-\frac{r_2^2+{r^{'}}^2+R^2_2}{4t}}I^{'}_{\frac{k}{\beta}}(\frac{r_2r^{'}}{2t})cosk(\theta_2-\theta^{'})]|dt.
\end{eqnarray*}
We only  estimate $|II_2|$,  and  $|II_1|$ can be settled down  in the same  way and are actually easier to deal with.
Let \begin{itemize}
      \item $a_{1,k}=cosk(\theta_1-\theta^{'}),\ a_{2,k}=cosk(\theta_2-\theta^{'})$;
      \item $b_1=(s_{i,1}-\acute{s_i})e^{-\frac{r_1^2+{r^{'}}^2+R^2_1}{4t}},\ b_2=(s_{i,2}-\acute{s_i})e^{-\frac{r_2^2+{r^{'}}^2+R^2_2}{4t}}$;
      \item $c_{1,k}=I^{'}_{\frac{k}{\beta}}(\frac{r_1r^{'}}{2t}),\ c_{2,k}=I^{'}_{\frac{k}{\beta}}(\frac{r_2r^{'}}{2t})$.
    \end{itemize}

    Using the identity \begin{equation}\label{abc trick} a_{1,k}b_1c_{1,k}-a_{2,k}b_2c_{2,k}=(a_{1,k}-a_{2,k})b_1c_{1,k}+(b_1-b_2)a_{2,k}c_{1,k}+a_{2,k}b_2(c_{1,k}-c_{2,k}),
     \end{equation}assuming  $r_1\leq r_2$ without loss of generality, we estimate the contributions of $|(a_{1,k}-a_{2,k})b_1c_{1,k}|$, $|(b_1-b_2)a_{2,k}c_{1,k}|$, $|a_{2,k}b_2(c_{1,k}-c_{2,k})|$ to $|II_2|$ independently in the following. \\

From now on we will repeatedly  use the assumption $|u_1|, |u_2|,|y|\leq 1$. These mean that 
\begin{equation} \label{CH01} r_1,\ r_2,\ \acute{r},\  \widehat{u_1},\  \widehat{u_2}, \ \widehat{y},\ R\  \leq 2.
\end{equation} 

    First, on the contribution of $(a_{1,k}-a_{2,k})b_1c_{1,k}$ to $|II_2|$, we compute
    \begin{eqnarray}\label{CH05}& &|(a_{1,k}-a_{2,k})c_{1,k}| \nonumber
    \\&=&r_1^{\frac{1}{\beta}-1}|cosk(\theta_1-\theta^{'})-cosk(\theta_2-\theta^{'})|[r_1^{1-\frac{1}{\beta}}I^{'}_{\frac{k}{\beta}}(\frac{r_1r^{'}}{2t})].
    \end{eqnarray}
Using  $r_1\leq r_2$ and $\sin ku\leq Ck\sin u$, we have
\begin{eqnarray}\label{CH04}& &r_1^{\frac{1}{\beta}-1}|cosk(\theta_1-\theta^{'})-cosk(\theta_2-\theta^{'})|\nonumber
\\&\leq& Ck\{\sqrt{r_1r_2|\sin^2\frac{|\theta_1-\theta_2|}{2}|}\}^{\frac{1}{\beta}-1}\leq Ck|u_1-u_2|^{\frac{1}{\beta}-1}.
\end{eqnarray}
By Lemma \ref{estimating the holder continuity of the derivative of I},  we have for $k\geq 1$ that

\begin{eqnarray}\label{CH02}r_1^{1-\frac{1}{\beta}}I^{'}_{\frac{k}{\beta}}(\frac{r_1r^{'}}{2t})
\leq r_1^{1-\frac{1}{\beta}}I_{\frac{k}{\beta}-1}(\frac{r_1r^{'}}{2t})
= (\frac{r^{'}}{2t})^{\frac{1}{\beta}-1}\frac{I_{\frac{k}{\beta}-1}(\frac{r_1r^{'}}{2t})}{(\frac{r_1r^{'}}{2t})^{\frac{1}{\beta}-1}}.
\end{eqnarray}
When $k=0$ we have
\begin{eqnarray}\label{CH03}r_1^{1-\frac{1}{\beta}}I^{'}_{0}(\frac{r_1r^{'}}{2t})
= r_1^{1-\frac{1}{\beta}}I_{1}(\frac{r_1r^{'}}{2t})
= (\frac{r^{'}}{2t})^{\frac{1}{\beta}-1}\frac{I_{1}(\frac{r_1r^{'}}{2t})}{(\frac{r_1r^{'}}{2t})^{\frac{1}{\beta}-1}}.
\end{eqnarray}
   Thus for  all $k\geq 2$, by (\ref{CH02}), (\ref{CH03}), (\ref{CH04}),  and (\ref{CH05}), we have
   \begin{eqnarray}\label{CH06} & &|(a_{1,k}-a_{2,k})c_{1,k}|\nonumber
   \\&\leq& Ck(\frac{r^{'}}{2t})^{\frac{1}{\beta}-1}(\frac{r_1r^{'}}{2t})^{1-\frac{1}{\beta}}I_{\frac{k}{\beta}-1}(\frac{r_1r^{'}}{2t})|u_1-u_2|^{\frac{1}{\beta}-1}\nonumber
   \\&\leq& \frac{Ck}{t^{\frac{1}{\beta}-1}}(\frac{r_1r^{'}}{2t})^{1-\frac{1}{\beta}}I_{\frac{k}{\beta}-1}(\frac{r_1r^{'}}{2t})|u_1-u_2|^{\frac{1}{\beta}-1}.
   \end{eqnarray}

 At this stage, the readers should notice from the assumption $r_1\leq r_2$ we have that $b_2\leq b_1$ , which is crucial for our trick to work as in the previous steps. Namely, we have to factor $a_1b_1c_{1,k}-a_2b_2c_{2,k}$ in a right way as in (\ref{abc trick}). Applying  formula (\ref{G's formula by Weber's formual}), Lemma \ref{Estimating the sum of G w.r.t k}, and (\ref{CH06}),   we compute
    \begin{eqnarray}\label{CH07}& &\Sigma_{k=0}^{\infty}\int_{0}^{\infty}\frac{1}{t^{\frac{m}{2}+3}}|(a_{1,k}-a_{2,k})c_{1,k}|dt\nonumber
 \\&\leq&C|u_1-u_2|^{\frac{1}{\beta}-1}\Sigma_{k=1}^{\infty} \int_{0}^{\infty}\frac{k}{t^{\frac{m}{2}+2+\frac{1}{\beta}}} e^{-\frac{r_1^2+{r^{'}}^2+R^2_2}{4t}}(\frac{r_1r^{'}}{2t})^{1-\frac{1}{\beta}}I_{\frac{k}{\beta}-1}(\frac{r_1r^{'}}{2t})dt\nonumber
 \\&=&C|u_1-u_2|^{\frac{1}{\beta}-1}\Sigma_{k=1}^{\infty} kG_{\frac{k}{\beta}-1,\frac{m}{2}+2+\frac{1}{\beta}, 1-\frac{1}{\beta}}(r_1,r^{'},R_2)\nonumber
 \\&\leq &C|u_1-u_2|^{\frac{1}{\beta}-1}.
 \end{eqnarray}

Second, we turn our attention to  the contribution of  $|(b_1-b_2)a_{2,k}c_{1,k}|$ to $|II_2|$.  We compute
\begin{eqnarray*}& &|b_1-b_2|=|(s_{i,1}-\acute{s_i})e^{-\frac{r_1^2+{r^{'}}^2+R^2_1}{4t}}-(s_{i,2}-\acute{s_i})e^{-\frac{r_2^2+{r^{'}}^2+R^2_2}{4t}}|
    \\&=&|(s_{i,1}-\acute{s_i})e^{-\frac{{r^{'}}^2}{4t}}e^{-\frac{R^2_1}{4t}}(e^{-\frac{r^2_1}{4t}}-e^{-\frac{r^2_2}{4t}})
    +(s_{i,1}-\acute{s_i})e^{-\frac{{r^{'}}^2}{4t}}e^{-\frac{r^2_2}{4t}}(e^{-\frac{R^2_1}{4t}}-e^{-\frac{R^2_2}{4t}})
     \\& &+(s_{i,1}-s_{i,2})e^{-\frac{{r^{'}}^2}{4t}}e^{-\frac{r^2_2+R^2_2}{4t}}|
    .\end{eqnarray*}
Then by the mean value theorem and $r_1\leq r_2$ we have
$$|e^{-\frac{r^2_1}{4t}}-e^{-\frac{r^2_2}{4t}}|\leq \frac{100|r_1-r_2|}{t}e^{-\frac{r^2_1}{4t}}$$
and
$$|e^{-\frac{R^2_1}{4t}}-e^{-\frac{R^2_2}{4t}}|\leq \frac{100|R_1-R_2|}{t}e^{-\frac{R^2_{1,2}}{4t}}$$
for some $R_{1,2}$ between $R_1$ and $R_2$.
    Thus using the same way as we estimate the contribution of  $|(a_1-a_2)b_1c_{1,k}|$ to $|II_2|$, we conclude that
     \begin{eqnarray}\label{CH08}& &\Sigma_{k=0}^{\infty}\int_{0}^{\infty}\frac{1}{t^{\frac{m}{2}+3}}|(b_1-b_2)a_{2,k}c_{1,k}|dt
\leq C|u_1-u_2|.
\end{eqnarray}

  Next, we estimate the crucial contribution of $|a_{2,k}b_2(c_{1,k}-c_{2,k})|$ to $|II_2|$. \\

   Using Lemma \ref{estimating the holder continuity of the derivative of I} and the inequality $I_1^{'}(z)\leq I_{0}(z)$, recalling that  $r_2\geq r_1$
   and $\beta>\frac{1}{2}$, we have 

  \begin{eqnarray*}
  & &\quad |c_{1,k}-c_{2,k}|= |I^{'}_{\frac{k}{\beta}}(\frac{r_1r^{'}}{2t})-I^{'}_{\frac{k}{\beta}}(\frac{r_2r^{'}}{2t})|
 \\&\leq&  \left \{
\begin{array}{cc}
 \frac{\frac{k}{\beta}(\frac{k}{\beta}-1)}{2^{\frac{k}{\beta}}\Gamma(\frac{k}{\beta}+1)}(\frac{r_2r^{'}}{2t})^{\frac{k}{\beta}-2}|\frac{r_2r^{'}}{2t}-\frac{r_1r^{'}}{2t}|+|\frac{r_2r^{'}}{2t}-\frac{r_1r^{'}}{2t}|(1+\frac{k}{2\beta})I_{\frac{k}{\beta}}(\frac{r_2r^{'}}{2t}) &\\
             \textrm{when}\  k\geq 2&;     \\
 \frac{\frac{1}{\beta}}{2^{\frac{1}{\beta}}(\frac{1}{\beta})!}|\frac{r_2r^{'}}{2t}-\frac{r_1r^{'}}{2t}|^{\frac{1}{\beta}-1}+|\frac{r_2r^{'}}{2t}-\frac{r_1r^{'}}{2t}|(1+\frac{1}{2\beta})I_{\frac{1}{\beta}}(\frac{r_2r^{'}}{2t}) \ \textrm{when}\  k=1 &\\
     |\frac{r_1r^{'}}{2t}-\frac{r_2r^{'}}{2t}|I_0(\frac{r_2r^{'}}{2t})\  \textrm{when}\ k=0.&
\end{array}
\right .
 \\&=&  \left \{
\begin{array}{cc}
 \frac{\frac{k}{\beta}(\frac{k}{\beta}-1)}{2^{\frac{k}{\beta}}\Gamma(\frac{k}{\beta}+1)}(\frac{r^{'}}{2t})(\frac{r_2r^{'}}{2t})^{\frac{k}{\beta}-2}|r_2-r_1|+\frac{r^{'}}{2t}|r_2-r_1|(1+\frac{k}{2\beta})I_{\frac{k}{\beta}}(\frac{r_2r^{'}}{2t}) &\\
             \textrm{when}\  k\geq 2&;     \\
 \frac{\frac{1}{\beta}}{2^{\frac{1}{\beta}}(\frac{1}{\beta})!}(\frac{r^{'}}{2t})^{\frac{1}{\beta}-1}|r_2-r_1|^{\frac{1}{\beta}-1}+|r_2-r_1|\frac{r^{'}}{2t}(1+\frac{1}{2\beta})I_{\frac{1}{\beta}}(\frac{r_2r^{'}}{2t}) \ \textrm{when}\  k=1 &\\
             \frac{r^{'}}{2t}|r_2-r_1|I_0(\frac{r_2r^{'}}{2t})\  \textrm{when}\ k=0.&
\end{array}
\right .
 \end{eqnarray*}

   When
 $k=0$ we implicitly use  $I_{-1}(x)=I_1(x)$, thanks to the properties of Bessel functions.\\

 We estimate the contribution of $|a_{2,k}b_2(c_{1,k}-c_{2,k})|$ to $|II_2|$ as
 \begin{eqnarray}\label{CH09}& &\Sigma_{k\geq 2}\int_{0}^{\infty}\frac{1}{t^{\frac{m}{2}+3}}|c_{1,k}-c_{2,k}|a_{2,k}b_2dt\nonumber
 \\& &+\int_{0}^{\infty}\frac{1}{t^{\frac{m}{2}+3}}|c_{1,1}-c_{2,1}|a_{2,1}b_2dt+\int_{0}^{\infty}\frac{1}{t^{\frac{m}{2}+3}}|c_{1,0}-c_{2,0}|a_{2,0}b_2dt\nonumber
 \\&\leq &N_1+N_2+N_3+N_4+N_5,
 \end{eqnarray}
 where
 \begin{eqnarray*}& &N_1=\Sigma_{k\geq 2}\int_{0}^{\infty} \frac{\frac{k}{\beta}(\frac{k}{\beta}-1)}{2^{\frac{k}{\beta}}\Gamma(\frac{k}{\beta}+1)}(\frac{r^{'}}{2t})(\frac{r_2r^{'}}{2t})^{\frac{k}{\beta}-2}|r_2-r_1|\frac{|s_{i,2}-\acute{s_i}|}{t^{\frac{m}{2}+3}}e^{-\frac{r_2^2+{r^{'}}^2+R^2_2}{4t}}dt,
 \\& &N_2=\Sigma_{k\geq 2}\int_{0}^{\infty} (1+\frac{k}{2\beta})(\frac{r^{'}}{2t})|r_2-r_1|I_{\frac{k}{\beta}}(\frac{r_2r^{'}}{2t})\frac{|s_{i,2}-\acute{s_i}|}{t^{\frac{m}{2}+3}}e^{-\frac{r_2^2+{r^{'}}^2+R^2_2}{4t}}dt,
 \\& &N_3=\int_{0}^{\infty} \frac{\frac{1}{\beta}}{2^{\frac{1}{\beta}}(\frac{1}{\beta})!}(\frac{r^{'}}{2t})^{\frac{1}{\beta}-1}|r_2-r_1|^{\frac{1}{\beta}-1}\frac{|s_{i,2}-\acute{s_i}|}{t^{\frac{m}{2}+3}}e^{-\frac{r_2^2+{r^{'}}^2+R^2_2}{4t}}dt,
 \\& &N_4=\int_{0}^{\infty} (1+\frac{1}{2\beta})(\frac{r^{'}}{2t})|r_2-r_1|I_{\frac{1}{\beta}}(\frac{r_2r^{'}}{2t})\frac{|s_{i,2}-\acute{s_i}|}{t^{\frac{m}{2}+3}}e^{-\frac{r_2^2+{r^{'}}^2+R^2_2}{4t}},
 \\& &N_5=\int_{0}^{\infty} (\frac{r^{'}}{2t})|r_2-r_1|I_{0}(\frac{r_2r^{'}}{2t})\frac{|s_{i,2}-\acute{s_i}|}{2t^{\frac{m}{2}+3}}e^{-\frac{r_2^2+{r^{'}}^2+R^2_2}{4t}}dt.
 \end{eqnarray*}
Before estimating term by term, we should deal with the term $e^{-\frac{r_2^2+{r^{'}}^2+R^2_2}{4t}}$ first.
\begin{clm}\label{estimating the exponential term} $e^{-\frac{r_2^2+{r^{'}}^2+R^2_2}{4t}}\leq e^{-\frac{1}{8t}}$.
\end{clm}
The see why the claim holds, just notice that we have $|u_1|,|u_2|\leq \frac{1}{8},\ {r^{'}}^2+|\widehat{y}|^2= |y|^2 =1$, then we get
 \begin{eqnarray*}& & {r^{'}}^2+R_2^2={r^{'}}^2+(\widehat{y}-\widehat{u}_2)^2
 \\&\geq& 1-[|\widehat{y}|^2-||\widehat{y}|-|\widehat{u_2}||^2]
  \\&=& 1-|\widehat{u_2}|(|\widehat{y}|+|\widehat{u_2}|)\geq 1-\frac{1}{8}\times\frac{9}{8}>\frac{1}{2}.
 \end{eqnarray*}
 Therefore $e^{-\frac{r_2^2+{r^{'}}^2+R^2_2}{4t}}\leq e^{-\frac{1}{8t}}$.\\

Now we estimate term by term. For $N_1$,  we compute
 \begin{eqnarray*}& & N_1
 \\&=&\Sigma_{k\geq 2}\int_{0}^{\infty} \frac{\frac{k}{\beta}(\frac{k}{\beta}-1)}{2^{\frac{k}{\beta}}\Gamma(\frac{k}{\beta}+1)}(\frac{r^{'}}{2t})(\frac{r_2r^{'}}{2t})^{\frac{k}{\beta}-2}|r_2-r_1|\frac{|s_{i,2}-\acute{s_i}|}{t^{\frac{m}{2}+3}}e^{-\frac{r_2^2+{r^{'}}^2+R^2_2}{4t}}dt
  \\&\leq&|u_1-u_2|\Sigma_{k\geq 2} \frac{\frac{k}{\beta}(\frac{k}{\beta}-1)}{2^{\frac{2k}{\beta}-2}\Gamma(\frac{k}{\beta}+1)8^{\frac{k}{\beta}-2}}\int_{0}^{\infty} \frac{1}{t^{\frac{m}{2}+2+\frac{k}{\beta}}}e^{-\frac{1}{8t}}dt.
 \end{eqnarray*}

 Notice that by the definition of the $Gamma$-function we have
 $$\int_{0}^{\infty} \frac{1}{t^{\frac{m}{2}+2+\frac{k}{\beta}}}e^{-\frac{1}{8t}}dt=8^{\frac{m}{2}+\frac{k}{\beta}+1}\Gamma(\frac{m}{2}+\frac{k}{\beta}+1).$$

 Then we have
  \begin{eqnarray*}& & N_1
 \\&\leq&|u_1-u_2|\Sigma_{k\geq 2} \frac{\frac{k}{\beta}(\frac{k}{\beta}-1)(8^{\frac{m}{2}+\frac{k}{\beta}+1})\Gamma(\frac{m}{2}+\frac{k}{\beta}+1)}{2^{\frac{2k}{\beta}-1}\Gamma(\frac{k}{\beta}+1)8^{\frac{k}{\beta}-2}}
 \\&\leq&|u_1-u_2|\Sigma_{k\geq 2} \frac{(\frac{k}{\beta})^2(8^{\frac{m}{2}+4})(\frac{k}{\beta}+m)^m}{4^{\frac{k}{\beta}}}
  \\&\leq&\ C|u_1-u_2|.
 \end{eqnarray*}
 To deduce  the second inequality above we used the fact that $\frac{\Gamma(\frac{m}{2}+\frac{k}{\beta}+1)}{\Gamma(\frac{k}{\beta}+1)}\leq (\frac{k}{\beta}+m)^m$. Thus we finished the estimate on $N_1$.\\

 For $N_2$, from formula (\ref{G's formula by Weber's formual}),  and Corollary \ref{Estimating the sum of G w.r.t k}, taking $v=8$, $\mu_1=\frac{1}{8}$ and $\mu_2=1$ in Corollary \ref{Estimating the sum of G w.r.t k},  we have that
 \begin{eqnarray*}& & N_2
 \\&\leq&C|u_2-u_1|\Sigma_{k\geq 2} (1+\frac{k}{2\beta})G_{\frac{k}{\beta},\frac{m}{2}+4,0}(r_2,r^{'},R_2)
 \\&\leq& C|u_2-u_1|.
 \end{eqnarray*}
 For $N_3$, actually it's our crucial term since it concerns the H\"older continuity exponent $\frac{1}{\beta}-1$. Nevertheless, from our established results,  $N_3$ can be estimated easily as
  \begin{eqnarray*}& & N_3
 \\&\leq&C|u_1-u_2|^{\frac{1}{\beta}-1}\int_{0}^{\infty} \frac{1}{t^{\frac{m}{2}+2+\frac{1}{\beta}}}e^{-\frac{1}{8t}}dt
  \\&\leq&C|u_1-u_2|^{\frac{1}{\beta}-1}.
 \end{eqnarray*}
 Using the same (and easier) way as we estimated $N_2$, we have that 
$$N_4 +N_5\leq C|u_1-u_2|,$$
 where we applied Claim \ref{estimating the exponential term}. Since $|u_1-u_2|\leq |u_1-u_2|^{\frac{1}{\beta}-1}\leq 1$, using the decomposition (\ref{CH09} ) and  the above estimates for $(N_i,0\leq i\leq 5)$,  we estimate the contribution of $|c_{1,k}-c_{2,k}|a_{2,k}b_2$ to $|II_2|$  as
  \begin{eqnarray}\label{CH10}& &\Sigma_{k=0}^{\infty}\int_{0}^{\infty}\frac{1}{t^{\frac{m}{2}+3}}|c_{1,k}-c_{2,k}|a_{2,k}b_2dt
\leq C|u_1-u_2|^{\frac{1}{\beta}-1}.
\end{eqnarray}

 Thus now we are ready to make the final conclusion on $|II_2|$. Assume $\frac{1}{\beta}-1< 1$, using (\ref{CH07}), (\ref{CH08}), and (\ref{CH10})   on
 the contributions of $|(a_1-a_2)b_1c_{1,k}|$, $|(b_1-b_2)a_2c_{1,k}|$, $|a_2b_2(c_{1,k}-c_{2,k})|$ to $|II_2|$, we have
  \begin{eqnarray}\label{II 2}& &|II_2|\leq C|u_1-u_2|^{\frac{1}{\beta}-1}.
\end{eqnarray}

 Regarding $|II_1|$, the H\"older continuity exponent $\frac{1}{\beta}-1$ doesn't appear in it. Actually employing the same method as we estimate $|II_2|$ (in a much easier way), the following estimate holds.
   \begin{eqnarray}\label{II 1}& &|II_1|\leq C|u_1-u_2|.
\end{eqnarray}

 Thus by (\ref{CH11}), (\ref{II 1}), and (\ref{II 2}), the proof of the lemma is completed.
\end{proof}

\section{Asymptotic behaviors of  the heat kernel.\label{Asymptotic behaviors of  the heat kernel}}
In this section we prove Lemma \ref{Bound on the gradient of h.k when r small and t equals 1} and proposition \ref{proposition on integral of the product the heat kernel and distance function to the alpha}, which are two items in Theorem \ref{all the properties of the h.k1}.  They rely on the  asymptotic estimates for the derivatives of the heat kernel in Lemma \ref{asymptotic estimate of radius derivative of Heat kernel}, \ref{asymptotic estimate of theta derivative of Heat kernel}, \ref{asymptotic estimate of time derivative of Heat kernel},   which   are implied by Theorem  \ref{Decay estimates for E}. \\

From now on, in both section \ref{Asymptotic behaviors of  the heat kernel} and \ref{Decay estimates for the heat kernel} we write $v=\beta({\theta-\acute{\theta}})$. 
Notice that the domain of $v$ is $[-\beta \pi, \beta \pi]$.
\begin{lem}\label{Bound on the gradient of h.k when r small and t equals 1}Suppose $t=1$, $x\in A_{\frac{1}{\sqrt{s}}}$, $y\in A_{\frac{10}{\sqrt{s}}}$, $0<s\leq \frac{1}{10000}$, then the following estimate holds
  $$|\nabla_xH(x,y,1)|\leq Ce^{-\frac{1}{s}}.$$

\end{lem}
\begin{proof}{of Lemma \ref{Bound on the gradient of h.k when r small and t equals 1}:} This is a direct consequence of  the lemmas in section \ref{Asymptotic behaviors of  the heat kernel}. We only consider the estimate for $\frac{1}{r}\frac{\partial}{\partial \theta}H(x,y,1)$, the other directional derivatives are similar.  From  the hypothesis
 $$r^2+|\widehat{x}|^2=\frac{1}{s},\ \acute{r}^2+|\widehat{y}|^2=\frac{100}{s},$$
it's easy to deduce that $$(r-r^{'})^2+R^2\geq \frac{20}{s}.$$
Then from Lemma \ref{asymptotic estimate of radius derivative of Heat kernel} and \ref{asymptotic estimate of theta derivative of Heat kernel}, it's straightforward to see that
\begin{eqnarray*} |\frac{1}{r}\frac{\partial}{\partial \theta}H(x,y,1)|\leq
 Ce^{-\frac{1}{s}}.
\end{eqnarray*}
\end{proof}

Let $\large a=\frac{\sin^2(\frac{\beta\pi}{2})}{4}.$ In the following discussions, we will frequently use the  inequality 
\begin{equation}\label{Holder inequality-simple version}1+x^p\leq C(1+x^q)  \ \textrm{if}\ q\geq p \ \textrm{and}\ 0\leq p\leq q\leq 10,
\end{equation}
and the inequality
\begin{equation}\label{inequality on power of x times exponential of x}x^pe^{-x^q}\leq C_{p,q},\ \textrm{if}\ q>0.
\end{equation}

\begin{prop}\label{proposition on integral of the product the heat kernel and distance function to the alpha}
There exists a constant $C$ with the following property.  For any $x\in \mathbb{R}^2\times \mathbb{R}^m$, any  second order spatial differential operator $\mathfrak{D}\in \mathfrak{T}$,\ and any $\alpha$  such that $0\leq \alpha\leq 1$, the following estimates on the heat kernel hold.
$$\int_{\mathbb{R}^2\times \mathbb{R}^{m}}|\sup_{1\leq t\leq 2}\frac{\partial}{\partial t} \mathfrak{D}H(x,y,t)||x-y|^{\alpha}dy\leq C;$$
$$\int_{\mathbb{R}^2\times \mathbb{R}^{m}}|\sup_{1\leq t\leq 2}\frac{\partial^2}{\partial t^2} H(x,y,t)||x-y|^{\alpha}dy\leq C.$$
\end{prop}
\begin{proof}{of Proposition \ref{proposition on integral of the product the heat kernel and distance function to the alpha}:}\
We still write $v=\beta({\theta-\acute{\theta}})$.
Since \begin{equation}\label{H=H^2 times HE}H=\widehat{H}\times H_E,\ H_E=\frac{1}{(4\pi t)^{\frac{m}{2}}}e^{-\frac{R^2}{4t}},\ R=|\widehat{x}-\widehat{y}|,
\end{equation}

we obviously have
\begin{eqnarray*}& &|\sup_{1\leq t\leq 2}\frac{\partial}{\partial t} \mathfrak{D}H(x,y,t)|
\\&\leq&(\sup_{1\leq t\leq 2}|\frac{1}{r}\frac{\partial\widehat{H}}{\partial \theta}|+\sup_{1\leq t\leq 2}|\frac{1}{r}\frac{\partial^2\widehat{H}}{\partial t \partial \theta}|+\sup_{1\leq t\leq 2}|\frac{\partial\widehat{H}}{\partial r}|+\sup_{1\leq t\leq 2}|\frac{\partial^2\widehat{H}}{\partial t \partial r}|+\sup_{1\leq t\leq 2}|\frac{\partial\widehat{H}}{\partial t}|)
\\& &\times(\sup_{1\leq t\leq 2}|\frac{\partial H_E}{ \partial t}|+\Sigma_{i=1}^{m}(\sup_{1\leq t\leq 2}|\frac{\partial H_E}{ \partial s_i}|+\sup_{1\leq t\leq 2}|\frac{\partial^2 H_E}{\partial t \partial s_i}|)).
\end{eqnarray*}
From (\ref{H=H^2 times HE}) we see that\begin{equation}\label{A01}(\sup_{1\leq t\leq 2}|\frac{\partial H_E}{ \partial t}|+\Sigma_{i=1}^{m}(\sup_{1\leq t\leq 2}|\frac{\partial H_E}{ \partial s_i}|+\sup_{1\leq t\leq 2}|\frac{\partial^2 H_E}{\partial t \partial s_i}|)\leq Ce^{-\frac{R^2}{16}}.\end{equation}

 Without loss of generality we can assume $\theta=0$. Thus from Lemma \ref{asymptotic estimate of radius derivative of Heat kernel}, \ref{asymptotic estimate of theta derivative of Heat kernel}, and \ref{asymptotic estimate of time derivative of Heat kernel}, we have
\begin{eqnarray*}& &(\sup_{1\leq t\leq 2}|\frac{1}{r}\frac{\partial\widehat{H}}{\partial \theta}|+\sup_{1\leq t\leq 2}|\frac{1}{r}\frac{\partial^2\widehat{H}}{\partial t \partial \theta}|+\sup_{1\leq t\leq 2}|\frac{\partial\widehat{H}}{\partial r}|
\\& &+\sup_{1\leq t\leq 2}|\frac{\partial^2\widehat{H}}{\partial t \partial r}|+\sup_{1\leq t\leq 2}|\frac{\partial\widehat{H}}{\partial t}|)
\\&\leq&C\underline{Z},
\end{eqnarray*}
where

\begin{eqnarray*}& &\underline{Z}
\\&=&(1+\frac{1}{r}|r{\acute{r}}\sin v|)e^{-\frac{({r-{\acute{r}}})^2}{16}}e^{-\frac{r{\acute{r}}[1-\cos v]}{8}}
\\& &+
e^{-\frac{({r-{\acute{r}}})^2}{16}}e^{-ar{\acute{r}}}
+e^{-\frac{r^2+{\acute{r}}^2}{16}}.
\end{eqnarray*}
Thus
\begin{equation*}|\sup_{1\leq t\leq 2}\frac{\partial}{\partial t} \mathfrak{D}H(x,y,t)|\leq C\underline{Z}e^{-\frac{R^2}{16}}.\end{equation*}

Then combining \ref{A01},   we estimate
\begin{eqnarray*}& &\int_{\mathbb{R}^2\times \mathbb{R}^{m}}|\sup_{1\leq t\leq 2}\frac{\partial}{\partial t} \mathfrak{D}H(x,y,t)||x-y|^{\alpha}dy
\\&\leq&C\int_{-\beta\pi}^{\beta\pi}\int_{0}^{\infty}\int_{R^m}\underline{Z}e^{-\frac{R^2}{16}}(|{r-{\acute{r}}}|^{\alpha}+r{\acute{r}}[1-\cos v]
+R^{\alpha})\acute{r}d\acute{r}d{v} d\widehat{y}
\\&=&C\int_{-\beta\pi}^{\beta\pi}\int_{0}^{\infty}\underline{Z}(|{r-{\acute{r}}}|^{\alpha}+r{\acute{r}}[1-\cos v]
)\acute{r}d\acute{r}d{v} \int_{R^m}e^{-\frac{R^2}{16}}d\widehat{y}
\\& &+C\int_{-\beta\pi}^{\beta\pi}\int_{0}^{\infty}\underline{Z}\acute{r}d\acute{r}d{v} \int_{R^m}R^{\alpha}e^{-\frac{R^2}{16}}d\widehat{y}.
\end{eqnarray*}
Since it's obvious that
$$\int_{R^m}R^{\alpha}e^{-\frac{R^2}{16}}d\widehat{y}+\int_{R^m}e^{-\frac{R^2}{16}}d\widehat{y}\leq C,$$
 we have
 \begin{eqnarray*}& &\int_{\mathbb{R}^2\times \mathbb{R}^{m}}|\sup_{1\leq t\leq 2}\frac{\partial}{\partial t} \mathfrak{D}H(x,y,t)||x-y|^{\alpha}dy
 \\&\leq &\int_{-\beta\pi}^{\beta\pi}\int_{0}^{\infty}\underline{Z}(1+|{r-{\acute{r}}}|^{\alpha}+(r{\acute{r}}[1-\cos\theta^{'}])^{\frac{\alpha}{2}}
)\acute{r}d\acute{r}d{v}
\\&\leq&U_{1}+U_{2}+U_{3}+U_{4}.
 \end{eqnarray*}
 The four terms in the decomposition are 
 \begin{eqnarray*} &  U_{1}&=\int_{-\beta\pi}^{\beta\pi}\int_{0}^{\infty}e^{-\frac{({r-{\acute{r}}})^2}{16}}
                e^{-\frac{r{\acute{r}}[1-\cos v]}{8}}|r{\acute{r}}\sin v|(1+|{r-{\acute{r}}}|^{\alpha}+(r{\acute{r}}[1-\cos v])^{\frac{\alpha}{2}})
          \\& & \qquad    \times \frac{\acute{r}}{r}d\acute{r}d{v},
\\&  U_{2}&=\int_{-\beta\pi}^{\beta\pi}\int_{0}^{\infty}e^{-\frac{({r-{\acute{r}}})^2}{16}}
                e^{-\frac{r{\acute{r}}[1-\cos{v}]}{8}}(1+|{r-{\acute{r}}}|^{\alpha}+(r{\acute{r}}[1-\cos{v}])^{\frac{\alpha}{2}})
                \acute{r}d\acute{r}d{v},
           \\&  U_{3}&=\int_{-\beta\pi}^{\beta\pi}\int_{0}^{\infty}e^{-\frac{({r-{\acute{r}}})^2}{16}}
                e^{-ar\acute{r}}(1+|{r-{\acute{r}}}|^{\alpha}+(r{\acute{r}}[1-\cos{v}])^{\frac{\alpha}{2}})
                \acute{r}d\acute{r}d{v},
      \\&   U_{4}&=\int_{-\beta\pi}^{\beta\pi}\int_{0}^{\infty}e^{-\frac{r^2+\acute{r}^2}{16}}
                (1+|{r-{\acute{r}}}|^{\alpha}+(r{\acute{r}}[1-\cos{v}])^{\frac{\alpha}{2}})\acute{r}d\acute{r}d{v}.
 \end{eqnarray*}

 We first estimate the most critical term $U_1$. Using (\ref{inequality on power of x times exponential of x}) we have that
\begin{equation}\label{Number 1 application of the exponential inequality }e^{-\frac{({r-{\acute{r}}})^2}{32}}
                e^{-\frac{r{\acute{r}}[1-\cos{v}]}{16}}(1+|{r-{\acute{r}}}|^{\alpha}+(r{\acute{r}}[1-\cos{v}])^{\frac{\alpha}{2}})\leq C.
                \end{equation}
Then
\begin{equation}U_{1}\leq C\int_{-\beta\pi}^{\beta\pi}\int_{0}^{\infty}e^{-\frac{({r-{\acute{r}}})^2}{32}}
                e^{-\frac{r{\acute{r}}[1-\cos{v}]}{16}}|r{\acute{r}}\sin{v}|
               \frac{\acute{r}}{r}d\acute{r}d{v}.
\end{equation}
Notice that we have 
\[ \begin{array}{lcl} \int_{-\beta\pi}^{\beta\pi}|r{\acute{r}}\sin{v}| e^{-\frac{r{\acute{r}}[1-\cos{v}]}{16}}
                d{v}
                & = & 2\int_{0}^{\beta\pi}(r{\acute{r}}\sin{v} )e^{-\frac{r{\acute{r}}[1-\cos{v}]}{16}}
                d{v}\\
           & = & 32[1-e^{-\frac{r{\acute{r}}[1-\cos\beta\pi]}{16}}].\end{array}
\]

To bound $U_1$,  using the above identity we compute 
\begin{eqnarray}\label{U(1) general estimate}\nonumber U_{1}
&\leq& C\int_{0}^{\infty}\frac{\acute{r}}{r}e^{-\frac{({r-{\acute{r}}})^2}{32}}
                 d\acute{r}\int_{-\beta\pi}^{\beta\pi}e^{-\frac{r{\acute{r}}[1-\cos{v}]}{16}}|r{\acute{r}}\sin{v}|d{v}
      \nonumber  \\&\leq& C\int_{0}^{\infty}\frac{\acute{r}}{r}e^{-\frac{({r-{\acute{r}}})^2}{32}}
                 [1-e^{-\frac{r{\acute{r}}[1-\cos\beta\pi]}{16}}]d\acute{r}.
                \end{eqnarray}
                When $r\geq 1$, from (\ref{U(1) general estimate}) we have
                \begin{eqnarray*} U_{1}
                  &\leq & C[\int_{0}^{\infty}\frac{\acute{r}-r}{r}e^{-\frac{({r-{\acute{r}}})^2}{32}}
                 d\acute{r}+\int_{0}^{\infty}e^{-\frac{({r-{\acute{r}}})^2}{32}}
                 d\acute{r}]
                 \\&\leq& C.
\end{eqnarray*}
When $r\leq 1$, using the  inequality
 $$[1-e^{-\frac{r{\acute{r}}[1-\cos\beta\pi]}{16}}] \leq C r \cdot \acute{r}$$
and the inequality
$$\acute{r}^2\leq (\acute{r}-r)^2+2|\acute{r}-r|+1,$$
 we obtain  from (\ref{U(1) general estimate}) that
 \begin{eqnarray*}\label{U(1) general estimate}\nonumber U_{1}
  &\leq& C\int_{0}^{\infty}\frac{\acute{r}}{r}e^{-\frac{({r-{\acute{r}}})^2}{32}}
                 [1-e^{-\frac{r{\acute{r}}[1-\cos\beta\pi]}{16}}]d\acute{r}
      \leq C\int_{0}^{\infty}\acute{r}^2e^{-\frac{({r-{\acute{r}}})^2}{32}}
                 d\acute{r}
                 \\&\leq& C\int_{0}^{\infty}((\acute{r}-r)^2+2|\acute{r}-r|+1)e^{-\frac{({r-{\acute{r}}})^2}{32}}
                 d\acute{r}
                 \\&\leq& C.
                \end{eqnarray*}
Thus,   $U_{1}$ is bounded uniformly. Using (\ref{inequality on power of x times exponential of x}),  it's easy to see that
$$U_2,\ U_3,\ U_4\leq C.$$
Therefore, every term is controlled so the estimate of
 $$\int_{\mathbb{R}^2\times \mathbb{R}^{m}}|\sup_{1\leq t\leq 2}\frac{\partial}{\partial t} \mathfrak{D}H(x,y,t)||x-y|^{\alpha}dy$$ follows.\\

  The estimate of  $\int_{\mathbb{R}^2\times \mathbb{R}^{m}}|\sup_{1\leq t\leq 2}\frac{\partial^2}{\partial t^2} H(x,y,t)||x-y|^{\alpha}dy$ is a directly consequence of Lemma \ref{asymptotic estimate of time derivative of Heat kernel} and similar arguments as above.
\end{proof}

\begin{lem}\label{asymptotic estimate of theta derivative of Heat kernel}We have
\begin{eqnarray*}& &\sup_{1\leq t\leq 2}|\frac{1}{r}\frac{\partial\widehat{H}}{\partial \theta}|+\sup_{1\leq t\leq 2}|\frac{1}{r}\frac{\partial^2\widehat{H}}{\partial t \partial \theta}|
\\&\leq& \frac{C}{r}e^{-\frac{({r-{\acute{r}}})^2}{16}}e^{-\frac{r{\acute{r}}[1-\cos\beta(\theta-\theta^{'})]}{8}}|r{\acute{r}}\sin\beta(\theta-\theta^{'})|+
Ce^{-\frac{({r-{\acute{r}}})^2}{16}}e^{-ar{\acute{r}}}
\\& &+Ce^{-\frac{r^2+{\acute{r}}^2}{16}}.
\end{eqnarray*}

\end{lem}
\begin{proof}{of Lemma \ref{asymptotic estimate of theta derivative of Heat kernel}:} We still write $v=\beta(\theta-\theta^{'})$. We further assume that  $ v\neq \pi$ or $-\pi$ mod $2\beta\pi$ so the last term in (\ref{P's Expression formula}) vainishes. The conclusion for all $ v$ follows from continuity. There is an interesting point hiding  here: though $E(v,z)$ is not continuous with respect to $v$, the estimates we get by applying the bounds on $E(v,z)$ still works for all $v$, as the estimates are all continuous.\\

We compute
 \begin{eqnarray*}\label{Theta derivative of the 2-dim heat kernel}& &\frac{1}{\beta r}\frac{\partial\widehat{H}}{\partial  \theta}=\frac{1}{r}\frac{\partial\widehat{H}}{\partial  v}\nonumber
 \\&=&\frac{1}{r}\frac{1}{4\pi t}e^{-\frac{r^2+{{\acute{r}}}^2}{4t}}[-\Sigma_{k,| v+2k\beta \pi|<\pi}\frac{r\acute{r}}{2t}\sin( v+2k\beta \pi)e^{\frac{r\acute{r}}{2t}\cos( v+2k\beta \pi)}+\frac{1}{2\pi \beta}\frac{\partial E(\frac{r\acute{r}}{2t}, v)}{\partial  v} ]\nonumber
  \\&=&\frac{1}{2\pi \beta}\frac{1}{r}\frac{1}{4\pi t}e^{-\frac{r^2+{{\acute{r}}}^2}{4t}}\frac{\partial E(\frac{r\acute{r}}{2t}, v)}{\partial  v}-\frac{1}{r}\frac{1}{4\pi t}e^{-\frac{({r-{\acute{r}}})^2}{4t}}\frac{r\acute{r}}{2t}(\sin v ) e^{-\frac{r\acute{r}}{2t}(1-\cos v)}\nonumber
  \\& &-\Sigma_{k\neq 0,| v+2k\beta \pi|<\pi}\frac{1}{r}\frac{1}{4\pi t}e^{-\frac{({r-{\acute{r}}})^2}{4t}}\frac{r\acute{r}}{2t}[\sin( v+2k\beta \pi) ] e^{-\frac{r\acute{r}}{2t}[1-\cos( v+2k\beta \pi)]}.\nonumber
\end{eqnarray*}
Next,  note that when $k\neq 0$, we have
\begin{equation}\label{Def of  a}1-\cos( v+2k\beta \pi)\geq 2\sin^2(\frac{\beta\pi}{2})=8a.\end{equation}
Combining   the estimates in Theorem \ref{Decay estimates for E}, we have
\begin{eqnarray}\label{A02}& &\sup_{1\leq t\leq 2}|\Sigma_{k\neq 0,| v+2k\beta \pi|<\pi}\frac{1}{r}\frac{1}{4\pi t}e^{-\frac{({r-{\acute{r}}})^2}{4t}}\frac{r\acute{r}}{2t}[\sin( v+2k\beta \pi) ] e^{-\frac{r\acute{r}}{2t}[1-\cos( v+2k\beta \pi)]}|\nonumber
\\& & \qquad  \leq \frac{C}{r}(r{\acute{r}})e^{-\frac{({r-{\acute{r}}})^2}{8}}e^{-2ar{\acute{r}}} \leq  Ce^{-\frac{({r-{\acute{r}}})^2}{16}}e^{-ar{\acute{r}}}.
\end{eqnarray}
Inequality (\ref{A02}) is proved as the following. \begin{enumerate}
                                                  \item When $r\geq 1$, obviously we have
                                                  $$\frac{C}{r}(r{\acute{r}})e^{-\frac{({r-{\acute{r}}})^2}{8}}e^{-2ar{\acute{r}}}
\leq Ce^{-\frac{({r-{\acute{r}}})^2}{8}}e^{-2ar{\acute{r}}};$$
                                                  \item when $r\leq 1,\ \acute{r}\leq 2$,  we estimate
                                                  $$ C\acute{r}e^{-\frac{({r-{\acute{r}}})^2}{8}}e^{-2ar{\acute{r}}}\leq Ce^{-\frac{({r-{\acute{r}}})^2}{8}}e^{-2ar{\acute{r}}};$$
                                                  \item when $r\leq 1,\ \acute{r}\geq 2$,  we estimate
                                                  $$ C\acute{r}e^{-\frac{({r-{\acute{r}}})^2}{8}}e^{-2ar{\acute{r}}}\leq C(\acute{r}-r)e^{-\frac{({r-{\acute{r}}})^2}{8}}e^{-2ar{\acute{r}}}\leq Ce^{-\frac{({r-{\acute{r}}})^2}{16}}e^{-2ar{\acute{r}}}.$$
                                                \end{enumerate}

By Theorem \ref{Decay estimates for E} we immediately get
\begin{eqnarray*} \sup_{1\leq t\leq 2}|\frac{1}{2\pi \beta}\frac{1}{r}\frac{1}{4\pi t}e^{-\frac{r^2+{{\acute{r}}}^2}{4t}}\frac{\partial E(\frac{r\acute{r}}{2t}, v)}{\partial  v}|
&\leq &C\acute{r}\sup_{1\leq t\leq 2}e^{-\frac{r^2+{{\acute{r}}}^2}{8}}|\frac{1}{(\frac{r\acute{r}}{2t})}\frac{\partial E(\frac{r\acute{r}}{2t}, v)}{\partial  v}|
\\&\leq&Ce^{-\frac{r^2+{\acute{r}}^2}{16}}.
\end{eqnarray*}
Thus, we conclude
\begin{eqnarray}\label{Theta derivative of the 2-dim HK-part 1}& &\sup_{1\leq t\leq 2}|\frac{1}{r}\frac{\partial\widehat{H}}{\partial  v}|\nonumber
\\&\leq& \frac{C}{r}e^{-\frac{({r-{\acute{r}}})^2}{8}}e^{-\frac{r{\acute{r}}[1-\cos v]}{4}}|r{\acute{r}}\sin v|+
Ce^{-\frac{({r-{\acute{r}}})^2}{16}}e^{-ar{\acute{r}}}\nonumber
\\& &+Ce^{-\frac{r^2+{\acute{r}}^2}{16}}.
\end{eqnarray}

Continue differentiating (\ref{Theta derivative of the 2-dim heat kernel}) with respect to $t$ we get
 \begin{eqnarray*}\label{Time-Theta derivative of the 2-dim heat kernel}& &\frac{1}{\beta r}\frac{\partial^2\widehat{H}}{\partial t\partial  \theta}=\frac{1}{r}\frac{\partial^2\widehat{H}}{\partial t\partial  v}\nonumber
\\&=&e^{-\frac{({r-{\acute{r}}})^2+2r\acute{r}(1-\cos v)}{4t}} \{\frac{1}{r}\frac{r\acute{r}\sin v}{4\pi t^3}
-\frac{1}{r}\frac{r\acute{r}\sin v}{8\pi t^2}\frac{({r-{\acute{r}}})^2}{4t^2}
-\frac{1}{r}\frac{r\acute{r}\sin v}{8\pi t^2}\frac{r\acute{r}}{2t^2}(1-\cos v)\}
\\& &+\Sigma_{k\neq 0,| v+2k\beta \pi|<\pi}e^{-\frac{({r-{\acute{r}}})^2+2r\acute{r}[1-\cos( v+2k\beta\pi)]}{4t}} \{\frac{1}{r}\frac{r\acute{r}\sin( v+2k\beta\pi)}{4\pi t^3}\nonumber
\\& &-\frac{1}{r}\frac{r\acute{r}\sin( v+2k\beta\pi)}{8\pi t^2}\frac{({r-{\acute{r}}})^2}{4t^2}
-\frac{1}{r}\frac{r\acute{r}\sin( v+2k\beta\pi)}{8\pi t^2}\frac{r\acute{r}}{2t^2}[1-\cos( v+2k\beta\pi)]\}\nonumber
\\& &+e^{-\frac{r^2+{{\acute{r}}}^2}{4t}}\{-\frac{1}{2\pi \beta}\frac{1}{r}\frac{1}{4\pi t^2}\frac{\partial E(\frac{r\acute{r}}{2t}, v)}{\partial  v}
+\frac{1}{2\pi \beta}\frac{1}{r}\frac{1}{4\pi t}(\frac{r^2+{{\acute{r}}}^2}{4t^2})\frac{\partial E(\frac{r\acute{r}}{2t}, v)}{\partial  v}
\\& &+\frac{1}{2\pi \beta}\frac{1}{r}\frac{1}{4\pi t}\frac{\partial^2 E(\frac{r\acute{r}}{2t}, v)}{\partial z\partial  v}(-\frac{r\acute{r}}{2t^2})\}.
\end{eqnarray*}
Then,  by the same idea as in the estimate of $\sup_{1\leq t\leq 2}|\frac{1}{r}\frac{\partial\widehat{H}}{\partial  v}|$, using (\ref{Def of  a}), Theorem \ref{Decay estimates for E}, and Cauchy-Schwarz inequalities, we directly get
\begin{eqnarray*} & &\sup_{1\leq t\leq 2}|\Sigma_{k\neq 0,| v+2k\beta \pi|<\pi}e^{-\frac{({r-{\acute{r}}})^2+2r\acute{r}[1-\cos( v+2k\beta\pi)]}{4t}} \{\frac{1}{r}\frac{r\acute{r}\sin( v+2k\beta\pi)}{4\pi t^3}\nonumber
\\& &-\frac{1}{r}\frac{r\acute{r}\sin( v+2k\beta\pi)}{8\pi t^2}\frac{({r-{\acute{r}}})^2}{4t^2}
-\frac{1}{r}\frac{r\acute{r}\sin( v+2k\beta\pi)}{8\pi t^2}\frac{r\acute{r}}{2t^2}[1-\cos( v+2k\beta\pi)]\}|\nonumber
\\&\leq& Ce^{-\frac{({r-{\acute{r}}})^2}{16}}e^{-ar{\acute{r}}}
\end{eqnarray*}
and
\begin{eqnarray*}\label{Theta derivative of the 2-dim HK-part 2} & &\sup_{1\leq t\leq 2}|e^{-\frac{r^2+{{\acute{r}}}^2}{4t}}\{-\frac{1}{2\pi \beta}\frac{1}{r}\frac{1}{4\pi t^2}\frac{\partial E(\frac{r\acute{r}}{2t}, v)}{\partial  v}
+\frac{1}{2\pi \beta}\frac{1}{r}\frac{1}{4\pi t^2}(\frac{r^2+{{\acute{r}}}^2}{4t})\frac{\partial E(\frac{r\acute{r}}{2t}, v)}{\partial  v}\nonumber
\\& &+\frac{1}{2\pi \beta}\frac{1}{r}\frac{1}{4\pi t}\frac{\partial^2 E(\frac{r\acute{r}}{2t}, v)}{\partial z\partial  v}(-\frac{r\acute{r}}{2t^2})\}|\nonumber
\\&\leq& C(1+r^3+{\acute{r}}^3)e^{-\frac{r^2+{\acute{r}}^2}{8}},\nonumber
\\&\leq& Ce^{-\frac{r^2+{\acute{r}}^2}{16}}.
\end{eqnarray*}
Thus, the estimate of  $\sup_{1\leq t\leq 2}|\frac{1}{r}\frac{\partial^2\widehat{H}}{\partial t \partial  v}|$ follows.\\

The conclusion of Lemma \ref{asymptotic estimate of theta derivative of Heat kernel} directly follows from (\ref{Theta derivative of the 2-dim HK-part 1}) and (\ref{Theta derivative of the 2-dim HK-part 2}).
\end{proof}
The other lemmas in this section are proved under the same ideas as in the proof of Lemma \ref{asymptotic estimate of theta derivative of Heat kernel}.  For the sake of clarity,  we still  give detailed proof here,   since these lemmas  are crucial to the global estimates of the derivatives of the heat kernel and their conclusions are not the same.  

\begin{lem}\label{asymptotic estimate of radius derivative of Heat kernel}We have
\begin{eqnarray} \sup_{1\leq t\leq 2}|\frac{\partial\widehat{H}}{\partial r}|+\sup_{1\leq t\leq 2}|\frac{\partial^2\widehat{H}}{\partial t \partial r}| &\leq& Ce^{-\frac{({r-{\acute{r}}})^2}{16}}e^{-\frac{r{\acute{r}}[1-\cos\beta( \theta-  \theta^{'})]}{8}}
+Ce^{-\frac{({r-{\acute{r}}})^2}{16}}e^{-ar{\acute{r}}}\nonumber
\\& &  + Ce^{-\frac{r^2+{\acute{r}}^2}{16}}.\nonumber
\end{eqnarray}

\end{lem}
\begin{proof}{of Lemma \ref{asymptotic estimate of radius derivative of Heat kernel}:} Denote $v=\beta(\theta-\theta^{'})$. We further assume that  $ v\neq \pi$ or $-\pi$ mod $2\beta\pi$ so the last term in (\ref{P's Expression formula}) vainishes. The conclusion for all $ v$ follows from continuity of the estimates with respect to $v$.\\
 First we compute
 \begin{eqnarray*}\label{radius derivative of the 2-dim heat kernel}\frac{\partial\widehat{H}}{\partial r}
 &=&e^{-\frac{({r-{\acute{r}}})^2+2r\acute{r}(1-\cos v)}{4t}} \{-\frac{({r-{\acute{r}}})}{8\pi t^2}
 -\frac{1}{r}\frac{r\acute{r}(1-\cos v)}{8\pi t^2}\} \nonumber
  \\& & + \Sigma_{k\neq 0,| v+2k\beta \pi|<\pi}e^{-\frac{({r-{\acute{r}}})^2}{4t}} e^{-\frac{r\acute{r}}{2t}[1-\cos( v+2k\beta \pi)]}\{-\frac{({r-{\acute{r}}})}{8\pi t^2}
 \\& &-\frac{1}{r}\frac{r\acute{r}[1-\cos( v+2k\beta \pi)]}{8\pi t^2}\}\nonumber
  \\& &+e^{-\frac{r^2+{{\acute{r}}}^2}{4t}}\{\frac{1}{2\pi \beta}\frac{1}{r}\frac{1}{4\pi t}\frac{r\acute{r}}{2t}\frac{\partial E(\frac{r\acute{r}}{2t}, v)}{\partial z}-\frac{1}{2\pi \beta}\frac{1}{4\pi t}\frac{r}{2t} E(\frac{r\acute{r}}{2t}, v)\}.
\end{eqnarray*}

Then by the same idea as in the estimate of $\sup_{1\leq t\leq 2}|\frac{1}{r}\frac{\partial\widehat{H}}{\partial  v}|$, using (\ref{Def of  a}), Theorem \ref{Decay estimates for E}, (\ref{Holder inequality-simple version}), and (\ref{inequality on power of x times exponential of x}), we  get
\begin{eqnarray*}& &\sup_{1\leq t\leq 2}|\Sigma_{k\neq 0,| v+2k\beta \pi|<\pi}e^{-\frac{({r-{\acute{r}}})^2}{4t}} e^{-\frac{r\acute{r}}{2t}[1-\cos( v+2k\beta \pi)]}\{-\frac{({r-{\acute{r}}})}{8\pi t^2}
 \\& &-\frac{1}{r}\frac{r\acute{r}[1-\cos( v+2k\beta \pi)]}{8\pi t^2}\}|
\\&\leq& C(|{r-{\acute{r}}}|+{\acute{r}})e^{-\frac{({r-{\acute{r}}})^2}{8}}e^{-2ar{\acute{r}}}
\\&\leq& Ce^{-\frac{({r-{\acute{r}}})^2}{16}}e^{-ar{\acute{r}}}.
\end{eqnarray*}
Then   using the simple estimate $e^{-\frac{r^2+{\acute{r}}^2}{16}}(1+r^2)\leq C$ we have
\begin{eqnarray*}& &\sup_{1\leq t\leq 2}|e^{-\frac{r^2+{{\acute{r}}}^2}{4t}}\{\frac{1}{2\pi \beta}\frac{1}{r}\frac{1}{4\pi t}\frac{r\acute{r}}{2t}\frac{\partial E(\frac{r\acute{r}}{2t}, v)}{\partial z}-\frac{1}{2\pi \beta}\frac{1}{4\pi t}\frac{r}{2t} E(\frac{r\acute{r}}{2t}, v)\}|
\\&\leq& Ce^{-\frac{r^2+{\acute{r}}^2}{16}}.
\end{eqnarray*}
Thus the estimate for $\sup_{1\leq t\leq 2}|\frac{\partial\widehat{H}}{\partial r}|$ follows. \\

Continue differentiating (\ref{radius derivative of the 2-dim heat kernel}) with respect to $t$ we get
 \begin{eqnarray*}\label{Time-radius derivative of the 2-dim heat kernel}& &\frac{\partial^2\widehat{H}}{\partial t\partial r}\nonumber
\\&=&e^{-\frac{({r-{\acute{r}}})^2+2r\acute{r}(1-\cos v)}{4t}} \{\frac{({r-{\acute{r}}})}{4\pi t^3}
-\frac{({r-{\acute{r}}})}{16\pi t^4}[r\acute{r}(1-\cos v)]
-\frac{({r-{\acute{r}}})^3}{32\pi t^4}
 \\& &+\frac{1}{r}\frac{r\acute{r}(1-\cos v)}{4\pi t^3}
 -\frac{1}{r}\frac{r\acute{r}(1-\cos v)}{8\pi t^2}\frac{({r-{\acute{r}}})^2}{4t^2}
 -\frac{1}{r}\frac{[r\acute{r}(1-\cos v)]^2}{16\pi t^4}\}
 \end{eqnarray*}
 \begin{eqnarray*}&+&\Sigma_{k\neq 0,| v+2k\beta \pi|<\pi}e^{-\frac{({r-{\acute{r}}})^2+2r\acute{r}[1-\cos( v+2k\beta\pi)]}{4t}} \{\frac{({r-{\acute{r}}})}{4\pi t^3}\nonumber
\\& &-\frac{({r-{\acute{r}}})}{16\pi t^4}(r\acute{r})[1-\cos( v+2k\beta\pi)]
-\frac{({r-{\acute{r}}})^3}{32\pi t^4}
\\& &+\frac{1}{r}\frac{r\acute{r}[1-\cos( v+2k\beta\pi)]}{4\pi t^3}
 -\frac{1}{r}\frac{r\acute{r}[1-\cos( v+2k\beta\pi)]}{8\pi t^2}\frac{({r-{\acute{r}}})^2}{4t^2}
 \\& &-\frac{1}{r}\frac{[(r\acute{r}[1-\cos( v+2k\beta\pi)])^2}{16\pi t^4}\}
 \end{eqnarray*}
 \begin{eqnarray*}&+&e^{-\frac{r^2+{{\acute{r}}}^2}{4t}}\{\frac{1}{2\pi \beta}\frac{r}{4\pi t^3}E(\frac{r\acute{r}}{2t}, v)
-\frac{1}{2\pi \beta}\frac{r}{32\pi t^4}(r^2+{{\acute{r}}}^2)E(\frac{r\acute{r}}{2t}, v)
\\& &+\frac{1}{2\pi \beta}\frac{r(r\acute{r})}{16\pi t^4}\frac{\partial E(\frac{r\acute{r}}{2t}, v)}{\partial z}
-\frac{1}{2\pi \beta}\frac{1}{r}\frac{r\acute{r}}{4\pi t^3}\frac{\partial E(\frac{r\acute{r}}{2t}, v)}{\partial z}
\\& &+\frac{1}{2\pi \beta}\frac{1}{r}\frac{r\acute{r}}{8\pi t^2}\frac{r^2+{{\acute{r}}}^2}{4t^2}\frac{\partial E(\frac{r\acute{r}}{2t}, v)}{\partial z}
-\frac{1}{2\pi \beta}\frac{1}{r}\frac{(r\acute{r})^2}{16\pi t^4}\frac{\partial^2 E(\frac{r\acute{r}}{2t}, v)}{\partial z^2}\}.
\end{eqnarray*}
Then using (\ref{Def of  a}), Theorem \ref{Decay estimates for E},  (\ref{Holder inequality-simple version}), and (\ref{inequality on power of x times exponential of x}), we  get
\begin{eqnarray*}& &\sup_{1\leq t\leq 2}|\Sigma_{k\neq 0,| v+2k\beta \pi|<\pi}e^{-\frac{({r-{\acute{r}}})^2+2r\acute{r}[1-\cos( v+2k\beta\pi)]}{4t}} \{\frac{({r-{\acute{r}}})}{4\pi t^3}\nonumber
\\& &-\frac{({r-{\acute{r}}})}{16\pi t^4}[r\acute{r}(1-\cos v)]
-\frac{({r-{\acute{r}}})^3}{32\pi t^4}
\\& &+\frac{1}{r}\frac{r\acute{r}[1-\cos( v+2k\beta\pi)]}{4\pi t^3}
 -\frac{1}{r}\frac{r\acute{r}[1-\cos( v+2k\beta\pi)]}{8\pi t^2}\frac{({r-{\acute{r}}})^2}{4t^2}
 \\& &-\frac{1}{r}\frac{[(r\acute{r}[1-\cos( v+2k\beta\pi)])^2}{16\pi t^4}\}|\nonumber
\\&\leq& Ce^{-\frac{({r-{\acute{r}}})^2}{16}}e^{-ar{\acute{r}}}
\end{eqnarray*}
and
\begin{eqnarray*}& &\sup_{1\leq t\leq 2}|e^{-\frac{r^2+{{\acute{r}}}^2}{4t}}\{\frac{1}{2\pi \beta}\frac{r}{4\pi t^3}E(\frac{r\acute{r}}{2t}, v)
-\frac{1}{2\pi \beta}\frac{r}{32\pi t^4}(r^2+{{\acute{r}}}^2)E(\frac{r\acute{r}}{2t}, v)
\\& &+\frac{1}{2\pi \beta}\frac{r(r\acute{r})}{16\pi t^4}\frac{\partial E(\frac{r\acute{r}}{2t}, v)}{\partial z}
-\frac{1}{2\pi \beta}\frac{1}{r}\frac{r\acute{r}}{4\pi t^3}\frac{\partial E(\frac{r\acute{r}}{2t}, v)}{\partial z}
\\& &+\frac{1}{2\pi \beta}\frac{1}{r}\frac{r\acute{r}}{8\pi t^2}\frac{r^2+{{\acute{r}}}^2}{4t^2}\frac{\partial E(\frac{r\acute{r}}{2t}, v)}{\partial z}
-\frac{1}{2\pi \beta}\frac{1}{r}\frac{(r\acute{r})^2}{16\pi t^4}\frac{\partial^2 E(\frac{r\acute{r}}{2t}, v)}{\partial z^2}\}|
\\&\leq&C(1+\acute{r})[1+(r^2+{\acute{r}}^2)^2]e^{-\frac{r^2+{\acute{r}}^2}{12}}
\\&\leq& C e^{-\frac{r^2+{\acute{r}}^2}{16}}.
\end{eqnarray*}

Thus the estimate of  $\sup_{1\leq t\leq 2}|\frac{\partial^2\widehat{H}}{\partial t \partial r}|$ follows.
\end{proof}
Compared to the previous ones, the next lemma seems to be a lot easier. Since there is no term concerning  $\frac{1}{r}$.

\begin{lem}\label{asymptotic estimate of time derivative of Heat kernel}We have
\begin{eqnarray*}& &\sup_{1\leq t\leq 2}|\frac{\partial\widehat{H}}{\partial t}|+\sup_{1\leq t\leq 2}|\frac{\partial^2\widehat{H}}{\partial t^2}|
\\&\leq& Ce^{-\frac{({r-{\acute{r}}})^2}{16}}e^{-\frac{r{\acute{r}}[1-\cos\beta(\theta-\theta^{'})]}{8}}
+Ce^{-\frac{({r-{\acute{r}}})^2}{16}}e^{-ar{\acute{r}}}
+Ce^{-\frac{r^2+{\acute{r}}^2}{16}}.
\end{eqnarray*}

\end{lem}
\begin{proof}{of Lemma \ref{asymptotic estimate of radius derivative of Heat kernel}:} We write $v=\beta(\theta-\theta^{'})$. We still assume that  $ v\neq \pi$ or $-\pi$ mod $2\beta\pi$ so the last term in (\ref{P's Expression formula}) vainishes. The conclusion for all $ v$ follows from continuity of the estimates with respect to $v$. The proof is more or less routine after applying the ideas in Lemma \ref{asymptotic estimate of theta derivative of Heat kernel}, but we  want to emphasize the corresponding bound for each term here. We compute
 \begin{eqnarray}\label{time derivative of the 2-dim heat kernel}& &\frac{\partial\widehat{H}}{\partial t}\nonumber
 \\&=&e^{-\frac{({r-{\acute{r}}})^2+2r\acute{r}(1-\cos v)}{4t}} \{-\frac{1}{4\pi t^2}
 +\frac{(r-\acute{r})^2}{16\pi t^3}
+\frac{r\acute{r}(1-\cos v)}{8\pi t^3}\} \nonumber
  \\& &+\Sigma_{k\neq 0,| v+2k\beta \pi|<\pi}e^{-\frac{({r-{\acute{r}}})^2}{4t}} e^{-\frac{r\acute{r}}{2t}[1-\cos( v+2k\beta \pi)]}\{-\frac{1}{4\pi t^2}\nonumber
 \\& &+\frac{(r-\acute{r})^2}{16\pi t^3}
+\frac{r\acute{r}(1-\cos[ v+2k\beta \pi])}{8\pi t^3}\}\nonumber
  \\& &+e^{-\frac{r^2+{{\acute{r}}}^2}{4t}}\{-\frac{1}{2\pi \beta}\frac{1}{4\pi t^2}E(\frac{r\acute{r}}{2t}, v)
  +\frac{1}{2\pi \beta}\frac{r^2+{{\acute{r}}}^2}{16\pi t^3}E(\frac{r\acute{r}}{2t}, v)\nonumber
  \\& &-\frac{1}{2\pi \beta}\frac{r\acute{r}}{8\pi t^3} \frac{\partial E(\frac{r\acute{r}}{2t}, v)}{\partial z}\}.
\end{eqnarray}

Then by the same idea as in the estimate of $\sup_{1\leq t\leq 2}|\frac{1}{r}\frac{\partial\widehat{H}}{\partial  v}|$, using (\ref{Def of  a}),  (\ref{Holder inequality-simple version}), and (\ref{inequality on power of x times exponential of x}), we  get
\begin{eqnarray*}& &\sup_{1\leq t\leq 2}|\Sigma_{k\neq 0,| v+2k\beta \pi|<\pi}e^{-\frac{({r-{\acute{r}}})^2}{4t}} e^{-\frac{r\acute{r}}{2t}[1-\cos( v+2k\beta \pi)]}\{-\frac{1}{4\pi t^2}
 \\& &+\frac{(r-\acute{r})^2}{16\pi t^3}
+\frac{r\acute{r}(1-\cos[ v+2k\beta \pi])}{8\pi t^3}\}|
\\&\leq&C(1+|{r-{\acute{r}}}|^2+r{\acute{r}})e^{-\frac{({r-{\acute{r}}})^2}{8}}e^{-2ar{\acute{r}}}
\\&\leq&Ce^{-\frac{({r-{\acute{r}}})^2}{16}}e^{-ar{\acute{r}}}
\end{eqnarray*}
and
\begin{eqnarray*}& &\sup_{1\leq t\leq 2}|e^{-\frac{r^2+{{\acute{r}}}^2}{4t}}\{-\frac{1}{2\pi \beta}\frac{1}{4\pi t^2}E(\frac{r\acute{r}}{2t}, v)
  +\frac{1}{2\pi \beta}\frac{r^2+{{\acute{r}}}^2}{16\pi t^3}E(\frac{r\acute{r}}{2t}, v)
  \\& &-\frac{1}{2\pi \beta}\frac{r\acute{r}}{8\pi t^3} \frac{\partial E(\frac{r\acute{r}}{2t}, v)}{\partial z}\}|
\\&\leq& C(1+r^2+\acute{r}^2)e^{-\frac{r^2+{\acute{r}}^2}{8}}
\\&\leq& Ce^{-\frac{r^2+{\acute{r}}^2}{16}}.
\end{eqnarray*}
thus the estimate for $\sup_{1\leq t\leq 2}|\frac{\partial\widehat{H}}{\partial t}|$ follows. \\

Continue differentiating (\ref{time derivative of the 2-dim heat kernel}) with respect to $t$ we get
 \begin{eqnarray*}\label{Time-time derivative of the 2-dim heat kernel}& &\frac{\partial^2\widehat{H}}{\partial t^2}\nonumber
\\&=&e^{-\frac{({r-{\acute{r}}})^2+2r\acute{r}(1-\cos v)}{4t}} \{\frac{1}{2\pi t^3}
-\frac{(r-\acute{r})^2}{16\pi t^4}-\frac{r\acute{r}(1-\cos v)}{8\pi t^4}-\frac{3(r-\acute{r})^2}{16\pi t^4}
 \\& &+\frac{(r-\acute{r})^4}{64\pi t^5}
 +\frac{(r-\acute{r})^2r\acute{r}(1-\cos v)}{32\pi t^5}
 \\& &-\frac{3r\acute{r}(1-\cos v)}{8\pi t^4}+\frac{[r\acute{r}(1-\cos v)]^2}{16\pi t^5}
+\frac{(r-\acute{r})^2r\acute{r}(1-\cos v)}{32\pi t^5}\}
\end{eqnarray*}
\begin{eqnarray*}&+&\Sigma_{k\neq 0,| v+2k\beta \pi|<\pi}e^{-\frac{({r-{\acute{r}}})^2+2r\acute{r}[1-\cos( v+2k\beta\pi)]}{4t}} \{\frac{1}{2\pi t^3}
-\frac{(r-\acute{r})^2}{16\pi t^4}\\& &-\frac{r\acute{r}[1-\cos( v+2k\beta\pi)]}{8\pi t^4}
-\frac{3(r-\acute{r})^2}{16\pi t^4}
\\& +&\frac{(r-\acute{r})^4}{64\pi t^5}
 +\frac{(r-\acute{r})^2r\acute{r}[1-\cos( v+2k\beta\pi)]}{32\pi t^5}
 -\frac{3r\acute{r}[1-\cos( v+2k\beta\pi)]}{8\pi t^4}\\& &+\frac{[r\acute{r}[1-\cos( v+2k\beta\pi)]]^2}{16\pi t^5}
+\frac{(r-\acute{r})^2r\acute{r}[1-\cos( v+2k\beta\pi)]}{32\pi t^5}\}
\end{eqnarray*}
\begin{eqnarray*}&+&e^{-\frac{r^2+{{\acute{r}}}^2}{4t}}\{\frac{1}{2\pi \beta}\frac{1}{2\pi t^3}E(\frac{r\acute{r}}{2t}, v)
-\frac{1}{2\pi \beta}\frac{r^2+{{\acute{r}}}^2}{16\pi t^4}E(\frac{r\acute{r}}{2t}, v)
+\frac{1}{2\pi \beta}\frac{r\acute{r}}{8\pi t^4}\frac{\partial E(\frac{r\acute{r}}{2t}, v)}{\partial z}
  \\& &-\frac{1}{2\pi \beta}\frac{3(r^2+{{\acute{r}}}^2)}{16\pi t^4}E(\frac{r\acute{r}}{2t}, v)+\frac{1}{2\pi \beta}\frac{(r^2+{{\acute{r}}}^2)^2}{64\pi t^5}E(\frac{r\acute{r}}{2t}, v)
  \\& &-\frac{1}{2\pi \beta}\frac{(r^2+{{\acute{r}}}^2)r\acute{r}}{32\pi t^5} \frac{\partial E(\frac{r\acute{r}}{2t}, v)}{\partial z}
  +\frac{1}{2\pi \beta}\frac{3r\acute{r}}{8\pi t^4} \frac{\partial E(\frac{r\acute{r}}{2t}, v)}{\partial z}
  \\& &-\frac{1}{2\pi \beta}\frac{r\acute{r}}{8\pi t^3}\frac{r^2+{{\acute{r}}}^2}{4t^2} \frac{\partial E(\frac{r\acute{r}}{2t}, v)}{\partial z}
  +\frac{1}{2\pi \beta}\frac{(r\acute{r})^2}{16\pi t^5} \frac{\partial^2 E(\frac{r\acute{r}}{2t}, v)}{\partial z^2}\}.
\end{eqnarray*}
Then using (\ref{Def of  a}), Theorem \ref{Decay estimates for E},   (\ref{Holder inequality-simple version}), (\ref{inequality on power of x times exponential of x}), and the inequality $$\{1+({r-{\acute{r}}})^4+(r{\acute{r}})^2\}e^{-\frac{({r-{\acute{r}}})^2}{16}}e^{-ar{\acute{r}}}\leq C,$$
we  get
\begin{eqnarray*}& &\sup_{1\leq t\leq 2}|\Sigma_{k\neq 0,| v+2k\beta \pi|<\pi}e^{-\frac{({r-{\acute{r}}})^2+2r\acute{r}[1-\cos( v+2k\beta\pi)]}{4t}} \{\frac{1}{2\pi t^3}
-\frac{(r-\acute{r})^2}{16\pi t^4}\\& &-\frac{r\acute{r}[1-\cos( v+2k\beta\pi)]}{8\pi t^4}
-\frac{3(r-\acute{r})^2}{16\pi t^4}
\\& &+\frac{(r-\acute{r})^4}{64\pi t^5}
 +\frac{(r-\acute{r})^2r\acute{r}[1-\cos( v+2k\beta\pi)]}{32\pi t^5}
 -\frac{3r\acute{r}[1-\cos( v+2k\beta\pi)]}{8\pi t^4}\\& &+\frac{[r\acute{r}[1-\cos( v+2k\beta\pi)]]^2}{16\pi t^5}
+\frac{(r-\acute{r})^2r\acute{r}[1-\cos( v+2k\beta\pi)]}{32\pi t^5}\}|\nonumber
\\&\leq&C\{1+({r-{\acute{r}}})^4+(r{\acute{r}})^2\}e^{-\frac{({r-{\acute{r}}})^2}{8}}e^{-2ar{\acute{r}}}
\\&\leq&Ce^{-\frac{({r-{\acute{r}}})^2}{16}}e^{-ar{\acute{r}}}
\end{eqnarray*}
and
\begin{eqnarray*}& &\sup_{1\leq t\leq 2}|e^{-\frac{r^2+{{\acute{r}}}^2}{4t}}\{\frac{1}{2\pi \beta}\frac{1}{2\pi t^3}E(\frac{r\acute{r}}{2t}, v)
-\frac{1}{2\pi \beta}\frac{r^2+{{\acute{r}}}^2}{16\pi t^4}E(\frac{r\acute{r}}{2t}, v)
\\& &+\frac{1}{2\pi \beta}\frac{r\acute{r}}{8\pi t^4}\frac{\partial E(\frac{r\acute{r}}{2t}, v)}{\partial z}
  -\frac{1}{2\pi \beta}\frac{3(r^2+{{\acute{r}}}^2)}{16\pi t^4}E(\frac{r\acute{r}}{2t}, v)+\frac{1}{2\pi \beta}\frac{(r^2+{{\acute{r}}}^2)^2}{64\pi t^5}E(\frac{r\acute{r}}{2t}, v)
  \\& &-\frac{1}{2\pi \beta}\frac{(r^2+{{\acute{r}}}^2)r\acute{r}}{32\pi t^5} \frac{\partial E(\frac{r\acute{r}}{2t}, v)}{\partial z}
  +\frac{1}{2\pi \beta}\frac{3r\acute{r}}{8\pi t^4} \frac{\partial E(\frac{r\acute{r}}{2t}, v)}{\partial z}
  \\& &-\frac{1}{2\pi \beta}\frac{r\acute{r}}{8\pi t^3}\frac{r^2+{{\acute{r}}}^2}{4t} \frac{\partial E(\frac{r\acute{r}}{2t}, v)}{\partial z}
  +\frac{1}{2\pi \beta}\frac{(r\acute{r})^2}{16\pi t^5} \frac{\partial^2 E(\frac{r\acute{r}}{2t}, v)}{\partial z^2}\}|
\\&\leq&C[1+(r^2+\acute{r}^2)^2]e^{-\frac{r^2+{\acute{r}}^2}{8}}
\\&\leq&Ce^{-\frac{r^2+{\acute{r}}^2}{16}}.
\end{eqnarray*}

Thus,  the estimate of  $\sup_{1\leq t\leq 2}|\frac{\partial^2\widehat{H}}{\partial t^2}|$ follows.
\end{proof}
\section{Decay estimates for the heat kernel.\label{Decay estimates for the heat kernel}}

In this section, we prove the following theorem on the estimate of the $E$ term in the heat kernel representation in Theorem  \ref{Carslaw-Wang-Chen representation for hk}.
\begin{thm}\label{Decay estimates for E} For all $z\geq 0$ and $v$ such that  $v\neq\pi$ nor $-\pi$ $mod\ 2\beta\pi$, the function $E(z,v)$ in Theorem  \ref{Carslaw-Wang-Chen representation for hk} satisfies the following estimates.
\begin{itemize} 
\item For any $p\geq n-1,\ n\geq 0$, there exists a $ C_{p,n}$  such that
$$|\frac{\partial}{\partial v}\{ z^p\frac{\partial^n}{\partial z^n}  E(z,v)\}|\leq C_{p,n};$$
\item For any $p\geq \max\{n-1,0\},\ n\geq 0$, there exists a $ \widehat{C}_{p,n}$  such that
$$| z^p\frac{\partial^n}{\partial z^n}  E(z,v)\}|\leq \widehat{C}_{p,n}.$$
\end{itemize}

\end{thm}
\begin{rmk} The points  $v=\pi$ or $-\pi$ $mod\ 2\beta\pi$ are the discontiuities of $E(z,v)$. Our theorem shows that $E(z,v)$ behaves very well away from these discontinuity points of $v$, which is sufficient to apply. 
\end{rmk}
\begin{proof}{of Theorem \ref{Decay estimates for E}:}
By using inductions and inequality \ref{inequality on power of x times exponential of x}, these estimates are directly corollaries of  Lemma \ref{lemma on E's z-derivative bounds} and Lemma \ref{Bound on the angle derivative} below.
\end{proof}

\begin{lem}\label{lemma on E's z-derivative bounds}For any nonnegative integer $n\geq 1$, there exists a $C_n$ with the following property.  For all $z\geq 0$,  $v$ such that  $v\neq\pi$ nor $-\pi$ $mod\ 2\beta\pi$, the following estimates hold:

 $$| z^{n-1}\frac{\partial^n}{\partial z^n} [ e^zE(z,v)]|\leq C_n,$$
and
$$|  e^zE(z,v)|\leq C.$$
\end{lem}
\begin{proof}{of Lemma \ref{lemma on E's z-derivative bounds}:} Actually the estimates also hold when $v=\pi$ or $-\pi$ $mod\ 2\beta\pi$. However,  since $E(z,v)$ is not continuous when  $v=\pi$ or $-\pi$ $mod\ 2\beta\pi$, we want to concentrate on the estimates close to these discontinuity points,  but not exactly at these points. 
When $n\geq 1$, we compute from Theorem \ref{Carslaw-Wang-Chen representation for hk}  that
\begin{eqnarray}\label{integral formula for z-derivatives of E}& &z^{n-1}\frac{d^n[e^{z}E(z,v)]}{dz^n} \nonumber
\\&=&\int_{0}^{\infty}(-z(\cosh y-1))^{n-1}e^{-z(\cosh y-1)} (1-\cosh y)\nonumber
\\& &  \frac{2\sin\frac{\pi}{\beta}[\cos\frac{\pi}{\beta}-\cos\frac{v}{\beta}\cosh\frac{y}{\beta}]}{[\cosh\frac{y}{\beta}-\cos\frac{v-\pi}{\beta}][\cosh\frac{y}{\beta}-\cos\frac{v+\pi}{\beta}]}dy.
\end{eqnarray}
  To prove $|z^{n-1}\frac{d^n[e^{z}E(z,v)]}{dz^n}|\leq  C_n$, we divide the situation into two cases below.\\

 Case 1, suppose $\frac{1}{\beta}$ is an integer, then $\sin\frac{\pi}{\beta}=0$. Then $E(z,v)\equiv 0$, the conclusion of Lemma \ref{lemma on E's z-derivative bounds} immediately follows.\\

 Case 2, suppose   $\frac{1}{\beta}$ is not an integer. Then there is a $C$ such that $|\cos\frac{v-\pi}{\beta}-1|\leq \frac{1}{C}$ implies $|\cos\frac{v+\pi}{\beta}-1|\geq \frac{1}{C}$,
 and $|\cos\frac{v+\pi}{\beta}-1|\leq \frac{1}{C}$ implies $|\cos\frac{v-\pi}{\beta}-1|\geq \frac{1}{C}$.\\

 Then the only issue we should worry about is that when $\cos\frac{v+\pi}{\beta}$ (or $\cos\frac{v-\pi}{\beta}$) is very close to $1$, and $y$ is close to
 $0$, the integral  (\ref{integral formula for z-derivatives of E}) might not  be uniformly bounded with respect to $v$. We will show that this bad situation won't happen. Notice that if
 both $|\cos\frac{v-\pi}{\beta}-1|\geq \frac{1}{C}$ and $|\cos\frac{v+\pi}{\beta}-1|\geq \frac{1}{C}$ hold, we trivially have  $|z^{n-1}\frac{d^n[e^{z}E(z,v)]}{dz^n}|\leq  C_n$.\\

  Thus it suffices to assume $|\cos\frac{v-\pi}{\beta}-1|\leq \frac{1}{C}$ and consequently
 $|\cos\frac{v+\pi}{\beta}-1|\geq \frac{1}{C}$. The case when $|\cos\frac{v+\pi}{\beta}-1|\leq \frac{1}{C}$ is the same.\\

   First we compute
 \begin{equation}\label{integrand of E algebraic trick}\frac{\cos\frac{\pi}{\beta}-\cos\frac{v}{\beta}\cosh\frac{y}{\beta}}{\cosh\frac{y}{\beta}-\cos\frac{v-\pi}{\beta}}=\frac{\cos\frac{\pi}{\beta}-\cos\frac{v}{\beta}\cos\frac{v-\pi}{\beta}}{\cosh\frac{y}{\beta}-\cos\frac{v-\pi}{\beta}}-\cos\frac{v}{\beta},
 \end{equation}
 thus our  primary interest is in the term $\frac{\cos\frac{\pi}{\beta}-\cos\frac{v}{\beta}\cos\frac{v-\pi}{\beta}}{\cosh\frac{y}{\beta}}$. Then using
 the inequality $$\cosh\frac{y}{\beta}-\cos\frac{v-\pi}{\beta}\geq \frac{y^2}{2\beta^2}+2\sin^2\frac{v-\pi}{2\beta},$$
 we have
   for any positive real number $A$ that
  \begin{eqnarray}& &|\int_{0}^{A}\frac{\cos\frac{\pi}{\beta}-\cos\frac{v}{\beta}\cos\frac{v-\pi}{\beta}}{\cosh\frac{y}{\beta}-\cos\frac{v-\pi}{\beta}}dy| \nonumber
  \\&\leq&2\beta^2|\cos\frac{\pi}{\beta}-\cos\frac{v}{\beta}\cos\frac{v-\pi}{\beta}|\int_{0}^{A}\frac{1}{y^2+\beta^2\sin^2\frac{v-\pi}{2\beta}}dy \nonumber
  \\&=&2\beta^2|\cos\frac{\pi}{\beta}-\cos\frac{v}{\beta}\cos\frac{v-\pi}{\beta}|\times\frac{\arctan\frac{A}{\beta|\sin\frac{v-\pi}{2\beta}|}}{\beta|\sin\frac{v-\pi}{2\beta}|} \nonumber
  \\&=&\frac{2\beta|\cos\frac{\pi}{\beta}\sin^2\frac{v-\pi}{\beta}+\sin\frac{\pi}{\beta}\cos\frac{v-\pi}{\beta}\sin\frac{v-\pi}{\beta}|}{|\sin\frac{v-\pi}{2\beta}|}\times \arctan\frac{A}{\beta|\sin\frac{v-\pi}{2\beta}|}\nonumber
  \\&\leq&4 \pi\beta. \label{bounding the main term in E after algebraic trick}
  \end{eqnarray}

 By the same method, we have for any $A$ that
  \begin{eqnarray*}\int^{A}_{0}\frac{1}{\cosh\frac{y}{\beta}-\cos\frac{v+\pi}{\beta}}dy
  \leq\beta^2 \int^{A}_{0}\frac{1}{y^2+\frac{\beta^2}{C}}dy
  =\sqrt{C}\beta \arctan\frac{A\sqrt{C}}{\beta}\leq \sqrt{C}\beta \frac{\pi}{2}.
  \end{eqnarray*}

 It's easy to deduce the following inequality.
 \begin{equation}\label{E03}
    |z^{n-1}\frac{d^n e^{-z(\cosh y-1)}}{dz^n}|\leq C_n(\cosh y-1).
    \end{equation}
Furthurmore, Using the inequality $|\cos\frac{v+\pi}{\beta}-1|\geq \frac{1}{C}$ and equation (\ref{integrand of E algebraic trick}), it's easy to see that
\begin{eqnarray}\nonumber & &|\frac{2\sin\frac{\pi}{\beta}[\cos\frac{\pi}{\beta}-\cos\frac{v}{\beta}\cosh\frac{y}{\beta}]}{[\cosh\frac{y}{\beta}-\cos\frac{v-\pi}{\beta}][\cosh\frac{y}{\beta}-\cos\frac{v+\pi}{\beta}]} | \\
\label{the integrand of E splits}  &\leq&  C|\frac{\cos\frac{\pi}{\beta}-\cos\frac{v}{\beta}\cosh\frac{v-\pi}{\beta}}{\cosh\frac{y}{\beta}-\cos\frac{v-\pi}{\beta}}+\frac{1}{\cosh\frac{y}{\beta}-\cos\frac{v+\pi}{\beta}}|.
\end{eqnarray}

  Then using (\ref{E03}) and (\ref{the integrand of E splits}), we deduce  \begin{eqnarray}\label{E02}& &|z^{n-1}\frac{d^ne^{z}E(z,v)}{dz^n}|\nonumber
 \\&=&|\int_{0}^{\infty}z^{n-1}\frac{d^ne^{-z(\cosh y-1)}}{dz^n}\frac{2\sin\frac{\pi}{\beta}[\cos\frac{\pi}{\beta}-\cos\frac{v}{\beta}\cosh\frac{y}{\beta}]}{[\cosh\frac{y}{\beta}-\cos\frac{v-\pi}{\beta}][\cosh\frac{y}{\beta}-\cos\frac{v+\pi}{\beta}]}dy|
 \\&\leq&C_n\int_{0}^{\infty}\frac{(\cosh y-1)|\cos\frac{\pi}{\beta}-\cos\frac{v}{\beta}\cos\frac{v-\pi}{\beta}|}{\cosh\frac{y}{\beta}-\cos\frac{v-\pi}{\beta}}+\frac{(\cosh y-1)}{\cosh\frac{y}{\beta}-\cos\frac{v+\pi}{\beta}}dy.\nonumber
 \end{eqnarray}

 Since $\frac{1}{\beta}>1$ and $|\cos\frac{v+\pi}{\beta}-1|\geq \frac{1}{C}$, we have
 \begin{equation}\label{bounding z derivative of E  using splitting 2}\int_{0}^{\infty}\frac{\cosh y-1}{\cosh\frac{y}{\beta}-\cos\frac{v+\pi}{\beta}}dy\leq C.\end{equation}
 Using almost the same trick as in (\ref{bounding the main term in E after algebraic trick}) we make the following estimate:
   \begin{eqnarray}& &\int_{0}^{\infty}\frac{(\cosh y-1)|\cos\frac{\pi}{\beta}-\cos\frac{v}{\beta}\cos\frac{v-\pi}{\beta}|}{\cosh\frac{y}{\beta}-\cos\frac{v-\pi}{\beta}}dy \nonumber
   \\&\leq&5\int_{0}^{1}\frac{|\cos\frac{\pi}{\beta}-\cos\frac{v}{\beta}\cos\frac{v-\pi}{\beta}|}{\cosh\frac{y}{\beta}-\cos\frac{v-\pi}{\beta}}dy +3\int_{1}^{\infty}\frac{\cosh y-1}{\cosh\frac{y}{\beta}-\cos\frac{v-\pi}{\beta}}dy \nonumber
  \\& \leq& C.\label{bounding z derivative of E  using splitting}
  \end{eqnarray}

Thus, combining  (\ref{bounding z derivative of E  using splitting}), (\ref{bounding z derivative of E  using splitting 2}), and (\ref{E02}), we conclude that $$| z^{n-1}\frac{\partial^n}{\partial z^n} [ e^zE(z,v)]|\leq C_n.$$

The conclusion $| e^zE(z,v)|\leq C$ is obvious from the discussion above.
\end{proof}

\begin{lem}\label{Bound on the angle derivative}For any nonnegative integer $n\geq 0$, there exists a $C_n$ with the following property.  For all $z\geq 0$,  $v$ such that  $v\neq\pi$ nor $-\pi$ $mod\ 2\beta\pi$, the following estimate holds

 $$|\frac{\partial}{\partial v}\{ z^{n-1}\frac{\partial^n}{\partial z^n} [ e^zE(z,v)]\}|\leq C_n.$$

\end{lem}
\begin{proof}{of Lemma \ref{Bound on the angle derivative}:} The proof here  is  more complicated than that of Lemma \ref{lemma on E's z-derivative bounds}.  We only prove the statement for $|\frac{\partial^2}{\partial v\partial z} e^z E(z,v)|$, the proofs of $|\frac{\partial}{\partial v} [z^{n-1}\frac{\partial}{\partial z^n}e^z  E(z,v)]|$
are  by the same way and even easier. Recall that
$$E(z,v)=\int_{0}^{\infty}e^{-z\cosh y}F(v,y)dy,$$

 where
$$F(v,y)=\frac{2\sin\frac{\pi}{\beta}[\cos\frac{\pi}{\beta}-\cos\frac{v}{\beta}\cosh\frac{y}{\beta}]}{[\cosh\frac{y}{\beta}-\cos\frac{v-\pi}{\beta}][\cosh\frac{y}{\beta}-\cos\frac{v+\pi}{\beta}]}.$$
We compute
\begin{eqnarray*}& &\frac{\partial}{\partial v} F(v,y)
\\&=&\frac{2\sin\frac{\pi}{\beta}\sin\frac{v}{\beta}[\cosh^3\frac{y}{\beta}-(2+\cos\frac{v-\pi}{\beta}\cos\frac{v+\pi}{\beta})\cosh\frac{y}{\beta}+\cos\frac{v+\pi}{\beta}+\cos\frac{v-\pi}{\beta}]}{\beta[\cosh\frac{y}{\beta}-\cos\frac{v-\pi}{\beta}]^2[\cosh\frac{y}{\beta}-\cos\frac{v+\pi}{\beta}]^2}.
\end{eqnarray*}

Therefore we have the following integral representation of $\frac{\partial^2}{\partial v\partial z}[ e^z E(z,v)]$:
\begin{eqnarray}\label{integral formula for v mixed derivatives of E}& &\frac{\partial^2}{\partial v\partial z} e^z E(z,v) \nonumber
\\&=&\int_{0}^{\infty}\frac{2\sin\frac{\pi}{\beta}\sin\frac{v}{\beta}[\cosh^3\frac{y}{\beta}-(2+\cos\frac{v-\pi}{\beta}\cos\frac{v+\pi}{\beta})\cosh\frac{y}{\beta}+\cos\frac{v+\pi}{\beta}+\cos\frac{v-\pi}{\beta}]}{\beta[\cosh\frac{y}{\beta}-\cos\frac{v-\pi}{\beta}]^2[\cosh\frac{y}{\beta}-\cos\frac{v+\pi}{\beta}]^2}\nonumber
\\& &-(\cosh y-1)e^{-z(\cosh y-1) }dy.
\end{eqnarray}

With the same idea as in Lemma \ref{lemma on E's z-derivative bounds} in mind, we divide the situation into two cases below.\\

 Case 1, suppose $\frac{1}{\beta}$ is an integer, then $\sin\frac{\pi}{\beta}=0$. Then $E(z,v)\equiv 0$.\\

 Case 2, suppose   $\frac{1}{\beta}$ is not an integer. Then there is a $C$ such that $|\cos\frac{v-\pi}{\beta}-1|\leq \frac{1}{C}$ implies $|\cos\frac{v+\pi}{\beta}-1|\geq \frac{1}{C}$,
 and $|\cos\frac{v+\pi}{\beta}-1|\leq \frac{1}{C}$ implies $|\cos\frac{v-\pi}{\beta}-1|\geq \frac{1}{C}$.\\

 Again, the only issue we should worry about is that when $\cos\frac{v+\pi}{\beta}$ or $\cos\frac{v-\pi}{\beta}$ is very close to $1$, and $y$ is close to
 $0$, the integral  (\ref{integral formula for v mixed derivatives of E}) might not  be uniformly bounded with respect to $v$. We will show this bad situation won't happen. Notice that if
 both $|\cos\frac{v-\pi}{\beta}-1|\geq \frac{1}{C}$ and $|\cos\frac{v+\pi}{\beta}-1|\geq \frac{1}{C}$ hold, we trivially have  $|\frac{\partial^2}{\partial v\partial z} e^z E(z,v)|\leq C$.\\

  Thus it suffices to assume $|\cos\frac{v-\pi}{\beta}-1|\leq \frac{1}{C}$ and consequently
 $|\cos\frac{v+\pi}{\beta}-1|\geq \frac{1}{C}$. The case when $|\cos\frac{v+\pi}{\beta}-1|\leq \frac{1}{C}$ is the same.\\

 Next we start discussing on the case $|\cos\frac{v-\pi}{\beta}-1|\leq \frac{1}{C}$. It is easy to see that
 \begin{eqnarray*}& &\frac{\cosh^3\frac{y}{\beta}-(2+\cos\frac{v-\pi}{\beta}\cos\frac{v+\pi}{\beta})\cosh\frac{y}{\beta}+\cos\frac{v+\pi}{\beta}+\cos\frac{v-\pi}{\beta}}{[\cosh\frac{y}{\beta}-\cos\frac{v-\pi}{\beta}]^2}
 \\&=&\cosh\frac{y}{\beta}+2\cos\frac{v-\pi}{\beta}+\frac{3\cos^2\frac{v-\pi}{\beta}-\cos\frac{v-\pi}{\beta}\cos\frac{v+\pi}{\beta}-2}{\cosh\frac{y}{\beta}-\cos\frac{v-\pi}{\beta}}
 \\&& -\ \frac{\sin^2\frac{v-\pi}{\beta}[\cos\frac{v-\pi}{\beta}-\cos\frac{v+\pi}{\beta}]}{(\cosh\frac{y}{\beta}-\cos\frac{v-\pi}{\beta})^2}.
 \end{eqnarray*}
 Denote

    \begin{equation*}S(z,y)=\frac{2\sin\frac{\pi}{\beta}\sin\frac{v}{\beta}(1-\cosh y)e^{-z(\cosh y-1)}}{\beta[\cosh\frac{y}{\beta}-\cos\frac{v+\pi}{\beta}]^2},\end{equation*}

\begin{equation*}D_1(v,y)=\frac{1}{\cosh\frac{y}{\beta}-\cos\frac{v-\pi}{\beta}}-\frac{1}{\frac{y^2}{2\beta^2}+1-\cos\frac{v-\pi}{\beta}}.\end{equation*}

 Then we have the following splitting
 \begin{eqnarray*}& &\frac{\partial^2}{\partial v\partial z} e^z E(z,v)
 \\&=&\int_{0}^{+\infty}S(z,y)\frac{\cosh^3\frac{y}{\beta}-(2+\cos\frac{v-\pi}{\beta}\cos\frac{v+\pi}{\beta})\frac{y}{\beta}+\cos\frac{v+\pi}{\beta}+\cos\frac{v-\pi}{\beta}}{[\cosh\frac{y}{\beta}-\cos\frac{v-\pi}{\beta}]^2}dy
 \\&=&\int_{1}^{+\infty}(1-\cosh y)e^{-z(\cosh y-1)}\frac{\partial F(y,v)}{\partial v}dy+\int_{0}^{1}S(z,y)(\cosh\frac{y}{\beta}\\& &+2\cos\frac{v-\pi}{\beta})dy
+\int_{0}^{1}S(z,y)\frac{3\cos^2\frac{v-\pi}{\beta}-\cos\frac{v-\pi}{\beta}\cos\frac{v+\pi}{\beta}-2}{\cosh\frac{y}{\beta}-\cos\frac{v-\pi}{\beta}}dy
\\& &-\int_{0}^{1}S(z,y)\frac{\sin^2\frac{v-\pi}{\beta}[\cos\frac{v-\pi}{\beta}-\cos\frac{v+\pi}{\beta}]}{(\cosh\frac{y}{\beta}-\cos\frac{v-\pi}{\beta})^2}dy
 \\&=&\Sigma_{k=1}^{7}II_k.
 \end{eqnarray*}
 The terms in the splitting are  
 \begin{eqnarray*} & &II_1=\int_{1}^{+\infty}-(\cosh y-1)e^{-z(\cosh y-1)}\frac{\partial F(y,v)}{\partial v}dy,
            \\& &II_2=\int_{0}^{1}S(z,y)(\cosh\frac{y}{\beta}+2\cos\frac{v-\pi}{\beta})dy,
            \\& & II_3=(3\cos^2\frac{v-\pi}{\beta}-\cos\frac{v-\pi}{\beta}\cos\frac{v+\pi}{\beta}-2)\int_{0}^{1}S(z,y)D_1(v,y)dy,
              \\& & II_4=(3\cos^2\frac{v-\pi}{\beta}-\cos\frac{v-\pi}{\beta}\cos\frac{v+\pi}{\beta}-2)\int_{0}^{1}\frac{S(z,y)}{\frac{y^2}{2\beta^2}+1-\cos\frac{v-\pi}{\beta}}dy,
\\& &II_5=-\sin^2\frac{v-\pi}{\beta}[\cos\frac{v-\pi}{\beta}-\cos\frac{v+\pi}{\beta}]\int_{0}^{1}S(z,y)D^2_1(v,y)dy,
          \\& & II_6=-2\sin^2\frac{v-\pi}{\beta}[\cos\frac{v-\pi}{\beta}-\cos\frac{v+\pi}{\beta}]\int_{0}^{1}S(z,y)\frac{D_1(v,y)}{\frac{y^2}{2\beta^2}+1-\cos\frac{v-\pi}{\beta}}dy,
          \\& &II_7=-\sin^2\frac{v-\pi}{\beta}[\cos\frac{v-\pi}{\beta}-\cos\frac{v+\pi}{\beta}]\int_{0}^{1}\frac{S(z,y)}{|\frac{y^2}{2\beta^2}+1-\cos\frac{v-\pi}{\beta}|^2}dy.
 \end{eqnarray*}

We have the following claim regarding  $D_1$ and $S$.

 \begin{clm}\label{Bounding D_1 and D_2} There exists a  $C$ such that when $y\leq 1$, we have
 $$|D_1(v,y)|\leq C,\ \  |S(z,y)|\leq Cy^2.$$
 \end{clm}
The assumption $y\leq 1$ is important. It's easy to see  $|S(z,y)|\leq Cy^2$ from the 
inequality $|\cos\frac{v+\pi}{\beta}-1|\geq \frac{1}{C}$ .

The estimate $|D_1(v,y)|\leq C$ follows from the following inequality
$$D_1(v,y)=\frac{O(y^4)}{(\cosh\frac{y}{\beta}-\cos\frac{v-\pi}{\beta})(\frac{y^2}{2\beta^2}+1-\cos\frac{v-\pi}{\beta})}\leq \frac{O(y^4)}{y^4}\leq C$$
when $y\leq 1$, thanks to the fact that the Taylor series of $\cosh y$ at $0$ only contain even order terms. Then Claim \ref{Bounding D_1 and D_2} is proved.\\

          From Claim  \ref{Bounding D_1 and D_2} we immediately see 
          \begin{equation}\label{bounding term 2,3,6,7 in v derivative of E}|II_2|+|II_3|+|II_4|+|II_5|+|II_6|\leq C,\end{equation}
          and
          \begin{eqnarray}\label{bounding term 4,7 in v derivative of E}|II_7|
          &\leq& C\int_{0}^{1}\frac{y^2\sin^2(\frac{v-\pi}{\beta})}{|\frac{y^2}{2\beta^2}+2\sin^2\frac{v-\pi}{2\beta}|^2}dy \nonumber
          \\&\leq& C.
          \end{eqnarray}

          Next we show that $II_1$ could also be easily bounded.
          \begin{clm}\label{Bounding II_1} $|II_1|<C.$
 \end{clm}

          Claim \ref{Bounding II_1} follows from the trivial fact that $\frac{1}{\beta}>1$ and the following estimates when $y\geq 1$.
          \begin{eqnarray}\label{integral formula for v mixed derivatives of E}
|(-(\cosh y-1))e^{-z(\cosh y-1) }\frac{\partial}{\partial v} F(v,y)|
& \leq &  C\frac{(\cosh y-1)(\cosh^3\frac{y}{\beta})}{\cosh^4\frac{y}{\beta}}\nonumber
\\& \leq &Ce^{-(\frac{1}{\beta}-1)y}.\nonumber
\end{eqnarray}
Since the last term $Ce^{-(\frac{1}{\beta}-1)y}$ is integrable, then Claim \ref{Bounding II_1} follows.\\

 Then when $n=1$, Lemma \ref{Bound on the angle derivative}  follows from (\ref{bounding term 2,3,6,7 in v derivative of E}), (\ref{bounding term 4,7 in v derivative of E}), claim \ref{Bounding II_1}.\\

 When $n=0$, there is  a point we should clearify. With respect to $\frac{1}{z}\frac{\partial}{\partial v} e^z E(z,v)$, we apply the following fact which is straight forward from Theorem \ref{Representation formulas for the heat kernel}:
\begin{equation}\label{E01}\frac{\partial}{\partial v}  E(0,v)=0\ \textrm{when}\ v\neq\ \pi\ \textrm{nor} -\pi\ \textrm{mod}\ 2\beta\pi.
\end{equation}
 Notice that $E(z,v)$ is not differentiable when $v=\pi$ or $-\pi$ $mod\ 2\beta\pi$.\\

 The fact in (\ref{E01}) means that  when $v=\pi$ or $-\pi$ $mod\ 2\beta\pi$, we have
 \begin{eqnarray*}& \int_{0}^{\infty}&\frac{2\sin\frac{\pi}{\beta}\sin\frac{v}{\beta}[\cosh^3\frac{y}{\beta}-(2+\cos\frac{v-\pi}{\beta}\cos\frac{v+\pi}{\beta})\cosh\frac{y}{\beta}+\cos\frac{v+\pi}{\beta}+\cos\frac{v-\pi}{\beta}]}{\beta[\cosh\frac{y}{\beta}-\cos\frac{v-\pi}{\beta}]^2[\cosh\frac{y}{\beta}-\cos\frac{v+\pi}{\beta}]^2}dy
 \\&=&0.
\end{eqnarray*}
Then
 \begin{eqnarray*}& & \frac{1}{z}\frac{\partial}{\partial v} e^z E(z,v)
 \\&= &\int_{0}^{\infty}\frac{2\sin\frac{\pi}{\beta}\sin\frac{v}{\beta}[\cosh^3\frac{y}{\beta}-(2+\cos\frac{v-\pi}{\beta}\cos\frac{v+\pi}{\beta})\cosh\frac{y}{\beta}+\cos\frac{v+\pi}{\beta}+\cos\frac{v-\pi}{\beta}]}{\beta[\cosh\frac{y}{\beta}-\cos\frac{v-\pi}{\beta}]^2[\cosh\frac{y}{\beta}-\cos\frac{v+\pi}{\beta}]^2}
 \\&\times &\frac{e^{-z(\cosh y-1) }-1}{z}dy.
\end{eqnarray*}
Then instead of considering $S(z,y)=\frac{2\sin\frac{\pi}{\beta}\sin\frac{v}{\beta}(1-\cosh y)e^{-z(\cosh y-1)}}{\beta[\cosh\frac{y}{\beta}-\cos\frac{v+\pi}{\beta}]^2}$ (in the case of $ \frac{\partial^2}{\partial v\partial z} e^z E(z,v)$), we consider
$$\widehat{S}(z,y)=\frac{2\sin\frac{\pi}{\beta}\sin\frac{v}{\beta}(e^{-z(\cosh y-1)}-1)}{\beta[\cosh\frac{y}{\beta}-\cos\frac{v+\pi}{\beta}]^2z}$$
in the case of $\frac{\partial}{\partial v} e^z E(z,v)$. Apparently  when $v\neq\pi$ nor $-\pi$ $mod\ 2\beta\pi$, we still have $\frac{1}{[\cosh\frac{y}{\beta}-\cos\frac{v+\pi}{\beta}]^2}\leq C$; when  and $y\leq 1$, we still have that
$$\widehat{S}(z,y)\leq Cy^2.$$
Therefore the proof for $\frac{\partial^2}{\partial v\partial z} e^z E(z,v)$ exactly carries over to the proof of $\frac{\partial}{\partial v} e^z E(z,v)$ line by line.
\end{proof}
\section{Appendix: some lower order estimates.\label{Appendix: some lower order estimates}}
In this appendix, we enumerate some interpolation formulas needed for the Schauder estimate.  We also state some lower order Schauder estimates. 
As in \cite{GT}, they are necessary but routine and simple.  Some of them are not proved  by  the same way as in \cite{GT}. We would like to point out that comparing to
the proof of proposition \ref{Schauder estimate for the weak solution which is compactly supported}(which is the top order seminorm estimate and is our main work in this article,) the following lower order estimates are really similar  and  much much easier.
\begin{lem}\label{parabolic intepolations} For any $\epsilon>0$ and domain $\Omega$,  the following parabolic interpolation inequalities hold.
\[\left\{ \begin{array}{lcl}
 [\nabla u]^{(1)}_{0,\Omega\times[0,T]} & \leq & C(\epsilon)|u|_{0,\Omega \times[0,T]}+\epsilon[\nabla u]^{(1)}_{\alpha,\beta,\Omega\times[0,T]}; \\

 [u]^{(*)}_{\alpha,\beta,\Omega\times[0,T]}& \leq & \epsilon[\nabla u]^{(1)}_{0,\Omega\times[0,T]}+C(\epsilon)[u]_{0,\Omega\times[0,T]};\\

  |\Delta u|^{(2)}_{0,\Omega\times[0,T]}& \leq & C(\epsilon) |\nabla u|^{(1)}_{0,\Omega\times[0,T]}+ \epsilon|\Delta u|^{(2)}_{\alpha,\beta,\Omega\times[0,T]};\\
  
  |\sqrt{-1}\partial \bar{\partial} u|^{(2)}_{0,\Omega\times[0,T]}& \leq & C(\epsilon) |\nabla u|^{(1)}_{0,\Omega\times[0,T]}+ \epsilon|\sqrt{-1}\partial \bar{\partial} u|^{(2)}_{\alpha,\beta,\Omega\times[0,T]}.
   \end{array}\right.\]
\end{lem}
\begin{proof} {of Lemma \ref{parabolic intepolations}: }The second item is by  standard calculations in \cite{GT}. We  need to point out that, on the first  inequality $$[\nabla u]^{(1)}_{0,\Omega\times[0,T]}\leq C(\epsilon)|u|_{0,\Omega \times[0,T]}+\epsilon[\nabla u]^{(1)}_{\alpha,\beta,\Omega\times[0,T]}, $$
we have to deal with the term $\frac{1}{r}\frac{\partial u}{\partial \theta}$, in a different way. And the observation is that, for any $\widehat{x}$ and $r_0>0$, over every level surface $r=r_0$, there is a $\theta_0$, such that $\frac{\partial u}{\partial \theta}(r_0,\theta_0,\widehat{x})=0$.\\

The third item concerns integration by parts. The fourth item concern integration by parts over holomorphic disks. The rest  of the proof are regular as in .
\end{proof}

Similarly, using parabolic weights we can prove the following intepolation inequalities, which are applied also in the proof of Theorem \ref{Schauder estimate of the linear equation: single divisor} in section \ref{Holder estimate of the singular integrals and proof of the main Schauder estimates}.
\begin{lem}\label{parabolic timewise intepolations} For any $\epsilon>0$ and domain $\Omega$,  the following parabolic interpolation inequalities hold.
\[\left\{ \begin{array}{lcl}
 [\nabla u]^{[1]}_{0,\Omega\times[0,T]} & \leq & C(\epsilon)|u|_{0,\Omega \times[0,T]}+\epsilon[\nabla u]^{(1)}_{\alpha,\frac{\alpha}{2},\beta,\Omega\times[0,T]}; \\

 [u]^{[*]}_{\alpha,\frac{\alpha}{2},\beta,\Omega\times[0,T]}& \leq & \epsilon[\nabla u]^{[1]}_{0,\Omega\times[0,T]}+C(\epsilon)[u]_{0,\Omega\times[0,T]};\\

  |\Delta u|^{[2]}_{0,\Omega\times[0,T]}& \leq & C(\epsilon) |\nabla u|^{[1]}_{0,\Omega\times[0,T]}+ \epsilon|\Delta u|^{[2]}_{\alpha,\frac{\alpha}{2},\beta,\Omega\times[0,T]};\\
  
  |\sqrt{-1}\partial \bar{\partial} u|^{[2]}_{0,\Omega\times[0,T]}& \leq & C(\epsilon) |\nabla u|^{[1]}_{0,\Omega\times[0,T]}+ \epsilon|\sqrt{-1}\partial \bar{\partial} u|^{[2]}_{\alpha,\frac{\alpha}{2},\beta,\Omega\times[0,T]}.
   \end{array}\right.\]
\end{lem}
Apply the techniques in the previous sections we have the following lower order Schauder estimates.
\begin{lem}\label{parabolic intepolations 2}Suppose $\alpha<(\min{\frac{1}{\beta}-1,1})$. The following lower order Schauder estimates for the heat operator are true for all $u\in C^{2+\alpha,1+\frac{\alpha}{2},\beta}(\Omega\times[0,T])$ with zero initial value.
\[ \left\{ \begin{array} {lcl}
 |i\partial\bar{\partial}u|^{(2)}_{0,\Omega\times[0,T]} & \leq & C(\alpha)|u|_{0,\Omega\times[0,T]}+C(\alpha)|(\Delta -\frac{\partial }{\partial t})u|^{(2)}_{\alpha,\beta,\Omega\times[0,T]}; \\

  |\frac{\partial u}{\partial t}|^{(2)}_{0,\Omega\times[0,T]}& \leq & C(\alpha)|u|_{0,\Omega\times[0,T]}+ C(\alpha)|(\Delta -\frac{\partial }{\partial t})u|^{(2)}_{\alpha,\beta,\Omega\times[0,T]};  \\

  |\nabla u|^{(1)}_{\alpha,\frac{\alpha}{2},\beta,\Omega\times[0,T]}& \leq & C_4|u|_{0,\Omega\times[0,T]}+C_4|(\Delta-\frac{\partial}{\partial t}) u|^{(2)}_{0,\Omega\times[0,T]}.
\end{array} \right.
\]
\end{lem}

The next theorem is about the necessary $C^{0}$ estimate. The proof of it  only involves the parabolic maximal principle, which works in our conical $C^{2+\alpha,1+\frac{\alpha}{2},\beta}[0,T]$ setting. We refer the readers  to \cite{CYW} for a full proof. The idea originates from Jeffres' work on the elliptic case in \cite{Jeffres}.  
\begin{lem}\label{C0 estimate} Suppose 
 $\{ a(t),\ t\in [0,T]\} $ is a $C^{\alpha,\frac{\alpha}{2},\beta}[0,T]$  family of $(\alpha, \beta)$ metrics (See Definition \ref{Def of a C 2 alpha, alpha over 2 family of  metrics}).   Suppose $u \in C^{2+\alpha,1+\frac{\alpha}{2},\beta}[0,T]$ solves the following parabolic equation. 
 \begin{equation}\label{heat equation with metric g}
\frac{\partial u}{\partial t}=\Delta_{a(t)}u+v,\ u(0,x)=0. \end{equation}
Then  the following estimate holds. $$|u|_{0, {M}\times[0,T]}\leq C \ast T\ast |v|_{0,{M}\times[0,T]}.$$
 
\end{lem}

Xiuxiong chen, Department of Mathematics, Stony Brook University,
NY, USA;\ \ xiu@math.sunysb.edu.\\

Yuanqi Wang, Department of Mathematics, University of California  at Santa Barbara, Santa Barbara,
CA,  USA;\ \ wangyuanqi@math.ucsb.edu.
\end{document}